\documentclass{article}
\usepackage{mystyle}

\usetikzlibrary{external}
\tikzset{external/only named=true}

\input{band-systems.tikz}
\input{knotted-surfaces.tikz}
\input{torus.tikz}

\title{Classification of genus\=/two surfaces in \texorpdfstring{$\sphere[3]$}{S3}}
\author{Filippo Baroni}
\date{}

\begin{document}

\maketitle
\thispagestyle{fancy}
\begin{abstract}
We describe an algorithm to decide whether two genus\=/two surfaces embedded in the $3$\=/sphere are isotopic or not. The algorithm employs well\=/known techniques in $3$\=/manifolds topology, as well as a new algorithmic solution to a problem on free groups.
\end{abstract}



\section{Introduction}
\label{sec:introduction}

\subsection{The classification problem}

In mathematics, by ``classification'' we usually mean a list, finite or infinite, of all objects of a given type, up to a given equivalence relation. For instance, the list 
\[
\left\{\quot{\ZZ}{4\ZZ},\quot{\ZZ}{2\ZZ}\times\quot{\ZZ}{2\ZZ}\right\}
\]
is a classification of groups of order $4$ up to isomorphism, and the list
\[
\left\{\;
\tikzsetnextfilename{introduction-genus-g-surface}
\raisebox{-.8cm}{
\begin{tikzpicture}[torus/setup=65]   
\foreach \xs in {1,-1} {
\begin{scope}[xscale=\xs]
\tikzset{
torus/new={name=T1,at={(.9,0)},R=.45,r=.15},
torus/new={name=T2,at={(1.8,0)},R=.45,r=.15},
}
\draw[-,black,line cap=round][torus/outline={torus=T1,cut right,cut left}][torus/outline={torus=T2,cut left}];
\end{scope}
}
\node[black!60] {$\cdots$};
\draw[yshift=-.6cm,decorate,decoration={brace,mirror}] (-2.4,0) -- (2.4,0) node[midway,below=6pt] {$g$ holes};
\end{tikzpicture}}
\;:g\in\ZZ_{\ge 0}\right\}
\]
is a classification of closed orientable surfaces up to homeomorphism. Naturally, we want the description of such a list to be somewhat explicit: the reader will surely agree that
\[
\{\text{finitely generated Abelian groups up to isomorphism}\}
\]
is not a classification of finitely generated Abelian groups up to isomorphism, but
\[
\left\{\ZZ^k\times\quot{\ZZ}{a_1\ZZ}\times\cdots\times\quot{\ZZ}{a_n\ZZ}:\text{$k\in\ZZ_{\ge0}$, $n\in\ZZ_{\ge0}$, $a_1,\ldots,a_n\in\ZZ_{\ge 2}$, $a_1\mid\ldots\mid a_n$}\right\}
\]
is. This example raises the question of what degree of explicitness is required for something to be considered a classification.

For instance, let us consider the set of knots in the $3$\=/sphere: can we classify them up to isotopy?
\begin{itemize}
\item A task that is surely within our grasp is deciding whether two given knots in $\sphere[3]$ are isotopic. In fact, two knots being isotopic is equivalent to their two complements being homeomorphic via a homeomorphism which is meridian\=/preserving; this can be decided algorithmically thanks to \cref{thm:matveev-homeomorphism}.
\item Moreover, it is not hard to devise an algorithm which produces an infinite list $\LLL_0$ of knot diagrams, so that every knot is isotopic to one element (and possibly more) of our list; this follows from the fact that there are only finitely many knot diagrams with $n$ crossings up to isotopy for every $n\ge 0$, and they can be enumerated algorithmically.
\end{itemize}
We can combine these two ingredients to produce:
\begin{enumarabic}
\item a computable list $\LLL$ of knots containing exactly one representative for each isotopy class, by taking knots in $\LLL_0$ which are not isotopic to elements appearing earlier in the list;
\item an algorithm which takes a knot as input and returns the unique knot in $\LLL$ which is isotopic to it.
\end{enumarabic}
Given how complex the world of knots is, we cannot expect to find a classification which is as neat as the ones given above for surfaces or Abelian groups. This algorithmic answer is probably the best we can hope for, and we believe it is explicit enough to be considered a classification.

We hope that, in light of the above discussion, the reader will accept the following definition as a sensible one. A \emph{classification} of a set $\XXX$ up to an equivalence relation $\sim$ is the datum of
\begin{enumarabic}
\item an algorithm producing a list $\LLL$ containing exactly one representative for each equivalence class of $\quot{\XXX}{\sim}$, and
\item an algorithm which takes an element of $\XXX$ as input and returns the unique element of $\LLL$ equivalent to it.
\end{enumarabic}
Like in the knots example, in order to have a classification of $\XXX$ up to $\sim$, it is enough to provide a computable list $\LLL_0$ containing \emph{at least} one representative for each equivalence class of $\quot{\XXX}{\sim}$, and an algorithm to decide equivalence; since the second ingredient is usually the hardest to find, we call it the \emph{classification algorithm}.

Going back to the case of knots, a routine topological argument shows that the knot classification problem is equivalent to the classification of tori embedded in $\sphere[3]$ up to isotopy: the one\=/to\=/one correspondence is given by associating each knot to the boundary of its regular neighbourhood. It is only natural to then try and address the classification problem of higher genus surfaces in the $3$\=/sphere up to isotopy\footnote{It should be noted that the classification of genus\=/zero surfaces in $\sphere[3]$ is trivial, since there is only one isotopy class of embedded $2$\=/spheres.}. This is precisely the aim of this article, which provides a classification of genus\=/two surfaces in $\sphere[3]$. Producing a redundant list $\LLL_0$ is very easy, by simply enumerating all simplicial genus\=/two surfaces in all subdivisions of $\sphere[3]$ (see \cref{sec:introduction:classical algorithms}). Therefore, the only goal of this article can be summarised as follows.

\begin{theorem*}
There is an algorithm to decide whether two genus\=/two surfaces in $\sphere[3]$ are isotopic.
\end{theorem*}

The article is structured as follows.
In the rest of \cref{sec:introduction}, we introduce the notation and terminology we will use in our exposition, and we briefly discuss some foundational results in low\=/dimensional algorithmic topology.
We also give a few examples of phenomena which make the genus\=/two classification problem intrinsically harder than the knot classification problem.
\Cref{sec:homeomorphisms} is devoted to the computation of mapping class groups of $3$\=/manifolds.
The exposition in this section closely follows the work of \textcite{johannson-jsj}, revisited to provide constructive and effective proofs.
The arguments here are somewhat technical and involved, and may be skipped on a first reading; the main result of this section, that is \cref{thm:mapping class group of 3-manifold}, can be assumed as a black box without compromising the understanding of the rest of the article.
In \cref{sec:free groups}, we present a solution to an algorithmic problem on free groups, as well as an application to a topological decision problem.
Finally, in \cref{sec:classification algorithm}, we focus on the genus\=/two classification problem.
A careful case analysis and elementary topological arguments, combined with \cref{thm:mapping class group of 3-manifold} and \cref{thm:product of dehn twists extension to handlebody}, allow us to prove the theorem stated above.

\subsection{Notation}

\step{Manifolds.}
\begin{itemize}[beginpenalty=10000]
\item Throughout this article, we will always be working in the PL category: all manifolds will have a PL structure, and functions between manifolds will be assumed to be PL.
\item All $3$\=/manifolds will be compact, connected and oriented unless otherwise stated. Codimension-zero submanifolds and boundaries of $3$\=/manifolds will be implicitly oriented accordingly.
\item All surfaces ($2$\=/manifolds) will be compact, but we make no assumptions about connectedness or orientability.
\item For integers $g\ge 0$ and $k\ge 0$, we denote by $\surf{g,k}$ the connected orientable surface of genus $g$ with $k$ holes (also referred to as ``punctures''). If $g\ge 1$, we denote by $\nsurf{g,k}$\todo{Better symbol?} the surface obtained by taking the connected sum of $g$ projective planes and then removing $k$ open discs. In both cases, the integer $k$ may be omitted when it is equal to $0$. We will occasionally refer to the torus $\surf{2}$ by the symbol $\torus[2]$. 
\item If $X$ is a manifold and $Z\subs X$ is a subspace, we denote by $\closure{Z}$ and $\interior{Z}$ its closure and interior in $X$ respectively; the ambient space $X$ will always be clear from the context.
\end{itemize}

\step{Fibre bundles.}
\begin{itemize}[beginpenalty=10000]
\item For a fibre bundle $\map{p}{M}{B}$ with fibre $F$, we denote by $\boundary_h M$ its horizontal boundary, that is the induced $\boundary F$\=/bundle on $B$; we will refer to the vertical boundary $p^{-1}(\boundary B)$ as $\boundary_v M$.
\item We denote by $I$ the $1$\=/manifold $[0,1]$. For every surface $B$, there is exactly one (orientable) $I$\=/bundle over $B$; we denote it by $B\times I$ (the ``product $I$\=/bundle'') if $B$ is orientable, and by $B\twtimes I$ (the ``twisted $I$\=/bundle'') if it is not.
\item If $X$ is a manifold and $Z\subs X$ is a properly embedded submanifold, we denote by $\nbhd{Z}$ and $\onbhd{Z}$ the closed and open regular neighbourhoods of $Z$ in $X$ respectively. The closed neighbourhood $\nbhd{Z}$ is naturally endowed with a bundle structure over $Z$ with fibre the $(\codim Z)$\=/dimensional disc.
\end{itemize}

\step{Surfaces in $3$\=/manifolds.}
\begin{itemize}[beginpenalty=10000]
\item If $X$ is a manifold and $Z\subs X$ is a properly embedded codimension\=/one submanifold, we define $X\cut Z=X\setminus\onbhd{Z}$ to be the result of cutting $X$ along $Z$; if $Z$ is separating, the manifold $X\cut Z$ will be disconnected.
\item As far as the notions of irreducible, boundary irreducible, incompressible, boundary incompressible, and sufficiently large are concerned, we adopt the same definitions as \textcite{matveev}. Unless otherwise stated, by ``(boundary) compression disc'' we will always mean ``non\=/trivial (boundary) compression disc''.
\end{itemize}

\step{Homeomorphisms.}
\begin{itemize}[beginpenalty=10000]
\item Unless otherwise stated, homeomorphisms between $3$\=/manifolds will be orientation\=/preserving.
\item If $X$ and $Y$ are manifolds, and $A_1,\ldots,A_n$ and $B_1,\ldots,B_n$ are subspaces, a function
\[
\map{f}{(X,A_1,\ldots,A_n)}{(Y,B_1,\ldots,B_n)}
\]
is a function $\umap{X}{Y}$ such that $f(A_i)\subs B_i$ for $1\le i\le n$. We say that $f$ is a homeomorphism if it is bijective and $f(A_i)=f(B_i)$ for $1\le i\le n$. When the number $n$ is clear from the context or irrelevant, we will write $(X,\vec{A})$ as a shorthand for $(X,A_1,\ldots,A_n)$.
\item If $X$ and $Y$ are oriented, we denote by $\homeo{(X,\vec{A});(Y,\vec{B})}$ the set of homeomorphisms from $(X,\vec{A})$ to $(Y,\vec{B})$, modulo isotopies through homeomorphisms of the same kind. Note that $\homeo{(X,\vec{A});(X,\vec{A})}$ has a natural group structure; we use the shorthand $\homeo{X,\vec{A}}$ to refer to it. Moreover, if $S\subs X$ is a subspace, we denote by $\fixhomeo{S}{X,\vec{A}}$ the group of self\=/homeomorphisms of $(X,\vec{A})$ fixing $S$ pointwise, modulo isotopies through homeomorphisms of the same kind.
\item If $\map{f}{(X,\vec{A})}{(Y,\vec{B})}$ is a function, we define its \emph{trace}
\[
\map{\trace{f}}{(\boundary X,\vec{A}\cap\boundary X)}{(\boundary Y,\vec{B}\cap\boundary Y)}
\]
to be its restriction to the boundary. The trace behaves well with isotopies -- that is, isotopic homeomorphisms have isotopic traces. As a consequence, when $X$ and $Y$ are oriented, we have a well\=/defined \emph{trace map}
\[
\map{\trace{(-)}}{\homeo{(X,\vec{A});(Y,\vec{B})}}{\homeo{(\boundary X,\vec{A}\cap\boundary X);(\boundary Y,\vec{B}\cap\boundary Y)}},
\]
which is functorial in the sense that it preserves composition. For the sake of convenience, if $W$ is a set of (isotopy classes of) homeomorphisms, we define $\trace{W}=\{\trace{f}:f\in W\}$.
\end{itemize}

\step{Dehn twists.}
\begin{itemize}[beginpenalty=10000]
\item In the most general sense, a Dehn twist of a manifold $X$ about a properly embedded two\=/sided  codimension\=/one submanifold $Z$ is a self\=/homeomorphism of $X$ which is the identity outside $\nbhd{Z}$.
\item When $Z$ is an essential two\=/sided curve in a surface $X$, the group of Dehn twists about $Z$ modulo isotopies is isomorphic to $\ZZ$. If $X$ is oriented there is a way to select a preferred generator of this group, which we call \emph{the} Dehn twist about $Z$ and denote by $\twist{Z}$. If $X$ is orientable without a preferred orientation, fixing one arbitrarily allows us to pick the generating Dehn twists about different curves consistently. When $X$ is non\=/orientable, instead, we arbitrarily call one of the generating Dehn twists $\twist{Z}$, so that the other one is $\twist{Z}^{-1}$.
\item Similarly, if $Z$ is an incompressible boundary incompressible annulus in an (oriented) $3$\=/manifold $X$ the group of Dehn twists about $Z$ modulo isotopies is isomorphic to $\ZZ$. In this case, there is no preferred generator of this group\todo{Is this right?}, so we arbitrarily pick one to call $\twist{Z}$. We remark that, if $z_1$ and $z_2$ are the two boundary curves of $Z$, then $\map{\trace{\twist{Z}}}{\boundary X}{\boundary X}$ is equal to either $\twist{z_1}\twist{z_2}^{-1}$ or $\twist{z_1}^{-1}\twist{z_2}$, depending on the choice of $\twist{Z}$.
\end{itemize}

\subsection{Classical algorithms on \texorpdfstring{$3$\=/manifolds}{3-manifolds}}
\label{sec:introduction:classical algorithms}

\step{Triangulations of $3$\=/manifolds.} We are, of course, interested in algorithms on $3$\=/manifolds. The first issue we should address is, perhaps, what a $3$\=/manifold \emph{is} from an algorithmic point of view. The simplest way to think about a $3$\=/manifold $M$ in this setting is seeing it as a collection of $3$\=/simplices (or ``tetrahedra''), with some pairs of faces identified by simplicial isomorphisms. Such a description is called a \emph{triangulation} of $M$, and it fully encodes the topology of $M$ in a discrete and combinatorial fashion. Of course, from the perspective of algorithmic implementation, the name ``tetrahedra'' is purely suggestive: $3$\=/simplices are encoded as ordered $4$\=/tuples of ``vertices'' -- represented, for example, by integers -- and an identification between faces can be specified by two ordered triples of vertices, each describing a face of a tetrahedron.

Some care must be taken in order to ensure that the topological space $M$ obtained by gluing the tetrahedra is actually a manifold. We deal with orientation by orienting the tetrahedra, and requiring that gluing maps are orientation\=/reversing. Then $M$ is a $3$\=/manifold if and only if the link of each vertex is a disc -- for vertices on the boundary -- or a sphere. Note that both conditions can be easily checked algorithmically.

\step{Simplicial subdivisions.} Ideally, we would want all the ``objects'' we have to deal with algorithmically (namely, maps and submanifolds) to be simplicial; in other words, we want to describe them discretely at the level of simplices. Fortunately, in the world of $3$\=/manifolds, everything can be made simplicial, up to subdivision and isotopy.

\begin{definition}
Let $\TTT$ be a triangulation of a $3$\=/manifold $M$. A \emph{subdivision} of $\TTT$ is a triangulation $\TTT'$ of a $3$\=/manifold $M'$ together with a choice, for each vertex $v$ of $\TTT'$, of a tetrahedron $\Delta(v)$ of $\TTT$ and four non\=/negative rational numbers $q_1(v)$, $q_2(v)$, $q_3(v)$ and $q_4(v)$, satisfying the following constraints.
\begin{enumroman}
\item For every vertex $v$ of $\TTT'$, the equality $q_1(v)+q_2(v)+q_3(v)+q_4(v)=1$ holds.
\item Fix a vertex $v$ of $\TTT'$, and denote by $\vec{u}_1$, $\vec{u}_2$, $\vec{u}_3$ and $\vec{u}_4$ the vertices of $\Delta(v)$, which we interpret as the four canonical base vectors in $\RR^4$. If we think of $\Delta(v)$ as the convex hull of its vertices, then there is a unique point
\[
p(v)=\sum_{i=1}^4q_i(v)\vec{u}_i\in\Delta(v)\subs M.
\]
We require that, if $v_1$, $v_2$, $v_3$ and $v_4$ are vertices of a tetrahedron of $\TTT'$, then there is some tetrahedron of $\TTT$ containing $p(v_1)$, $p(v_2)$, $p(v_3)$ and $p(v_4)$.
\item Define a function $\map{f}{M'}{M}$ by extending the map $v\mapsto p(v)$ linearly on the tetrahedra of $\TTT'$; the previous constraint guarantees that $f$ is well\=/defined. We require that $f$ is a homeomorphism.
\end{enumroman}
\end{definition}

With slight abuse of notation, we will omit the choices of $\Delta(v)$ and $q_i(v)$, and simply call the triangulation $\TTT'$ a subdivision of $\TTT$. With this definition, subdivisions can be described in a fully combinatorial fashion. Note that, with some work, the statement ``$\TTT'$ is a subdivision of $\TTT$'' can be decided algorithmically.

The following two facts are crucial.
\begin{itemize}
\item Let $\TTT$ be a triangulation of a $3$\=/manifold $M$. Then every submanifold $N\subs M$ (of dimension $0$, $1$, $2$, or $3$) is isotopic to a simplicial submanifold of a subdivision of $\TTT$.
\item Let $\TTT_M$ and $\TTT_N$ be triangulations of the $3$\=/manifolds $M$ and $N$ respectively. Then every homeomorphism $\map{f}{M}{N}$ is isotopic to a simplicial isomorphism between a subdivision of $\TTT_M$ and a subdivision of $\TTT_N$.
\end{itemize}
Therefore, in every step of our algorithms, we will always assume that submanifolds of $3$\=/manifolds are simplicial. More precisely, a submanifold of a $3$\=/manifold $M$ with a triangulation $\TTT$ will be represented as a subdivision $\TTT'$ of $\TTT$, together with a set of simplices of $\TTT'$ (of the appropriate dimension). Similarly, a homeomorphism between $3$\=/manifolds $M$ and $N$ with triangulations $\TTT_M$ and $\TTT_N$ respectively will be represented as subdivisions $\TTT_M'$ and $\TTT_N'$ of $\TTT_M$ and $\TTT_N$ respectively, together with a simplicial isomorphism between $\TTT_M'$ and $\TTT_N'$.

\step{Algorithmic operations.} A useful property of subdivisions is the following: every two subdivisions $\TTT_1$ and $\TTT_2$ of the same triangulation $\TTT$ have a common subdivision $\TTT_{12}$, which can be computed algorithmically. Without going into too much detail, let us simply remark that this implies that we can -- and will -- always assume that different ``objects'' in the same $3$\=/manifold are simplicial with respect to the same subdivision.

All the natural operations one may want to perform on $3$\=/manifolds can be done algorithmically; here is a long -- but by no means complete -- list of examples:
\begin{itemize}
\item composing and inverting homeomorphisms;
\item finding images of submanifolds under homeomorphisms;
\item cutting along a properly embedded codimension\=/one submanifold;
\item finding (simplicial) regular neighbourhoods of properly embedded submanifolds;
\item isotoping properly embedded submanifolds so that they are in general position;
\item finding intersections of properly embedded submanifolds in general position.
\end{itemize}
Some of these constructions are trivial, while some require a lot of work. For more details on triangulations and subdivisions of PL manifolds, we refer the reader to the excellent introductory book by \textcite{rourke-pl-topology}.

\step{Infinite search template.} The infinite search template algorithm can be informally described as follows: if the elements of a set $\SSS$ can be enumerated algorithmically and $\SSS$ is guaranteed to be non\=/empty, then there is an algorithm to construct an element of $\SSS$. As obvious as it may sound, this template will often allow us to blur the distinction between ``proving that something exists'' and ``being able to algorithmically construct it''.

As an example, consider two $3$\=/manifolds $M$ and $N$, with triangulations $\TTT_M$ and $\TTT_N$ respectively. We can algorithmically enumerate all triples $(\TTT_M',\TTT_N',f)$ where $\TTT_M'$ is a subdivision of $\TTT_M$, $\TTT_N'$ is a subdivision of $\TTT_N$, and $f$ is a simplicial isomorphism between $\TTT_M'$ and $\TTT_M'$. If we know that $M$ and $N$ are homeomorphic, then the algorithm enumerating all the triples will eventually find one, effectively constructing a homeomorphism from $M$ to $N$. As a consequence, after proving that two $3$\=/manifolds are homeomorphic, we will always assume to have a homeomorphism available for our algorithmic purposes.

Let us remark that what we have described above is \emph{not} an algorithm to solve the homeomorphism problem -- that is, to decide whether $M$ and $N$ are homeomorphic. In fact, the algorithm will terminate if $M$ and $N$ are guaranteed to be homeomorphic, but will run forever if they are not; at any given point in time, there is no way to know if the algorithm has not found a homeomorphism because we have not waited long enough, or because there is none.

\step{Solved algorithmic problems.} In the realm of surfaces, all the elementary questions can be settled algorithmically.

\begin{theorem}
Given a surface $F$, the following operations can be carried out algorithmically.
\begin{enumarabic}
\item Deciding whether $F$ is orientable or not.
\item Computing the unique integers $g\ge 0$ and $k\ge 0$ such that $F$ is homeomorphic to $\surf{g,k}$ if it is orientable, and to $\nsurf{g,k}$ if it is not.
\item Given another surface $F'$, deciding if $F$ and $F'$ are homeomorphic.
\item Given a curve in $F$, deciding if it is trivial (that is, it bounds a disc in $F$) and if it is boundary parallel.
\item Given two (multi)curves in $F$, deciding if they are isotopic or not.
\item Given another surface $F'$ and two homeomorphisms $\umap{F}{F'}$, deciding whether they are isotopic or not.
\end{enumarabic}
\end{theorem}

Of course, the situation for $3$\=/manifolds is substantially more complicated. The homeomorphism problem has only been recently settled, the last piece needed being the geometrisation theorem, and complete written accounts only exist for the closed case -- we refer the reader to \textcite{kuperberg-homeomorphism}'s exposition for details. The isotopy problem for surfaces in a $3$\=/manifolds is, perhaps, even more difficult to investigate. The sole purpose of this article, after all, is to provide an algorithm solving the isotopy problem for one specific surface -- namely, $\surf{2}$ -- embedded in one specific $3$\=/manifold -- namely, $\sphere[3]$. However, classical algorithms exist to answer some basic questions about $3$\=/manifolds.

\begin{theorem}
Given a $3$\=/manifold $M$, the following operations can be carried out algorithmically.
\begin{enumarabic}\def\myitem#1{\item\label[statement]{thm:algorithms on 3-manifolds:#1}}
\myitem{irreducible} Deciding whether $M$ is irreducible.
\myitem{boundary irreducible} If $M$ is irreducible, deciding whether $M$ is boundary irreducible.
\myitem{non-separating disc} If $M$ is irreducible, deciding whether $\boundary M$ admits a non\=/separating compression disc in $M$.
\myitem{handlebody} Deciding whether $M$ is a handlebody (possibly a $3$\=/ball); if it is, computing its genus.
\myitem{boundary parallel} Given a surface properly embedded in $M$, deciding if it is boundary parallel.
\myitem{incompressible} Given an orientable surface which is either properly embedded in $M$ or embedded in $\boundary M$, deciding whether it is incompressible.
\myitem{boundary incompressible} If $M$ is irreducible and boundary irreducible, given an orientable incompressible surface properly embedded in $M$, deciding whether it is boundary incompressible.
\myitem{isotopic surfaces} If $M$ is irreducible, given two connected incompressible surfaces properly embedded in $M$, deciding whether they are isotopic.
\end{enumarabic}
\end{theorem}
\begin{proof}
\ExplSyntaxOn
\def\mycref#1{\seq_set_split:Nnn \l_tmpa_seq {,} {#1} \seq_set_map:NNn \l_tmpb_seq \l_tmpa_seq {thm:algorithms~ on~ 3-manifolds:##1} \cref{\seq_use:Nn \l_tmpb_seq {,}}}
\def\myCref#1{\seq_set_split:Nnn \l_tmpa_seq {,} {#1} \seq_set_map:NNn \l_tmpb_seq \l_tmpa_seq {thm:algorithms~ on~ 3-manifolds:##1} \Cref{\seq_use:Nn \l_tmpb_seq {,}}}
\ExplSyntaxOff
\myCref{irreducible,boundary irreducible,incompressible,boundary incompressible} are proved by \textcite{matveev} in Theorems 4.1.12, 4.1.13, 4.1.15, and 4.1.19 respectively. \myCref{handlebody} is solved by \cite[Algorithm 9.3]{jaco-jsj-algorithm}. Concerning \mycref{boundary parallel}, recall that a surface $F\subs M$ is boundary parallel if and only if it cobounds a product $F\times I$ with some surface in $\boundary M$, where $F\subs M$ is identified with $F\times\{0\}$. It is then enough to check whether $F$ is separating in $M$ and, if it is, whether there is a component $N$ of $M\setminus F$ such that $(\closure{N},\boundary F)$ is homeomorphic to $(F\times I,\boundary F\times\{0\})$; we can decide this thanks to \cref{thm:matveev-homeomorphism} below.

In order to address \mycref{isotopic surfaces}, let us loosely rephrase a result of \textcite[Proposition 5.4]{waldhausen-on-irreducible-manifolds}: if $F_1$ and $F_2$ are isotopic incompressible surfaces properly embedded in $M$ intersecting transversely with $\boundary F_1=\boundary F_2$, then a component of $F_1\setminus F_2$ is isotopic to a component of $F_2\setminus F_1$. First of all, we check if $\boundary F_1$ and $\boundary F_2$ are isotopic in $\boundary M$. If they are not, then $F_1$ and $F_2$ are not isotopic. Otherwise, we can assume that $\boundary F_1=\boundary F_2$. Then we iterate over all components of $M\cut(F_1\cup F_2)$ and check whether any of these is a product $G\times I$ with $G\times\{0\}\subs F_1$ and $G\times\{1\}\subs F_2$, once again by means of \cref{thm:matveev-homeomorphism}. If we find such a region, then either $G\times\{0\}=F_1$ and $G\times\{1\}=F_2$ -- thus showing that $F_1$ and $F_2$ are isotopic -- or we use the product $G\times I$ to isotope $G\times\{1\}\subs F_2$ across $G\times\{0\}\subs F_1$, so as to reduce the number of components of $F_1\cap F_2$. Hence, after finitely many steps, we either prove that $F_1$ and $F_2$ are isotopic, or we cannot find any more product regions; in this case, the two surfaces are not isotopic.

Finally, we prove \mycref{non-separating disc}. The crucial claim is the following: let $D\subs M$ be a separating compression disc for $\boundary M$; then $M$ admits a non\=/separating compression disc for its boundary if and only if the same holds for one component of $M\cut D$. In fact, let $E\subs M$ be a non\=/separating compression disc for $\boundary M$. Since $M$ is irreducible, we can arrange for $D$ and $E$ to intersect transversely in a collection of arcs. Denote by $E_1,\ldots,E_k$ the closures of the components of $E\setminus D$; they are all discs, and we can think of them as being properly embedded in $M\cut D$. Restriction induces an isomorphism $H^1(M)\iso H^1(M\cut D)$ on cohomology groups, sending $[E]$ -- the cohomology class represented by $E$ -- to $[E_1]+\ldots+[E_k]$. Since $E$ is non\=/separating in $M$, it is non\=/trivial in cohomology. Therefore, the same must hold for one of the discs $E_1,\ldots,E_k$, thus showing that at least one component of $M\cut D$ admits a non\=/separating compression disc for the boundary. The algorithm then proceeds as follows. If $M$ is boundary irreducible, then clearly the answer is no. Otherwise, let $D$ be a compression disc for the boundary; if it is non\=/separating we are done. Otherwise, we run the algorithm on the components of $M\cut D$, and return yes if and only if we find a non\=/separating disc in at least one of them. The algorithm will eventually terminate, since $\boundary M$ can only be compressed finitely many times.
\end{proof}

Concerning \cref{thm:algorithms on 3-manifolds:isotopic surfaces}, let us remark that surfaces in $\sphere[3]$ are quite the opposite of incompressible. This is what makes classifying them a challenging task, even for surfaces of genus two. We refer the reader to \cref{sec:introduction:examples} for a survey of the many different topological phenomena which can arise when embedding a genus\=/two surface in $\sphere[3]$, despite the trivial topology of the ambient space.

\step{Homeomorphism problem for manifolds with boundary pattern.} Most importantly, the homeomorphism problem has been solved for what \textcite{matveev} calls ``Haken $3$\=/manifolds with boundary pattern''. It is impossible to overstate how crucial this result is for our classification algorithm. In fact, the reader will probably find us referring to the following theorem more frequently then any other, often implicitly.

\begin{theorem}\label{thm:matveev-homeomorphism}
Let $M$ and $M'$ be irreducible $3$\=/manifolds with non\=/empty boundary, and let $p\subs\boundary M$ and $p'\subs\boundary M'$ be (possibly empty) unions of curves such that $\boundary M\cut p$ and $\boundary M'\cut p'$ are incompressible in $M$ and $M'$ respectively. There is an algorithm which, given as input $(M,p)$ and $(M',p')$, decides whether they are homeomorphic.
\end{theorem}

\Cref{thm:matveev-homeomorphism} is a restricted version of \cite[Theorem 6.1.6]{matveev} which deals with the more general settings in which $p$ and $p'$ are allowed to be arbitrary graphs (what \citeauthor{matveev} calls ``boundary patterns'').
There is, however, a technical point that needs to be addressed. \Citeauthor{matveev}'s theorem only guarantees an algorithmic solution to the question ``is there a possibly orientation\=/reversing homeomorphism $\umap{(M,p)}{(M',p')}$?''. One could, in theory, go through the proofs in \cite[Chapter 6]{matveev} and convince themselves that only trivial modifications are needed to answer the oriented version of the homeomorphism question. There is, however, a trick which allows us to derive \cref{thm:matveev-homeomorphism} directly from \cite[Theorem 6.1.6]{matveev}. Simply augment the boundary pattern $p$ by adding the same small chiral graph to each component of $M\cut p$, obtaining the boundary pattern $\Gamma\subs\boundary M$; carry out the same procedure for $p'$ to construct the boundary pattern $\Gamma'\subs M'$. Then $(M,p)$ and $(M',p')$ are orientation\=/preservingly homeomorphic if and only if $(M,\Gamma)$ and $(M,\Gamma')$ are possibly orientation\=/reversingly homeomorphic; a graphical explanation is provided in \cref{fig:adding chiral graphs}.
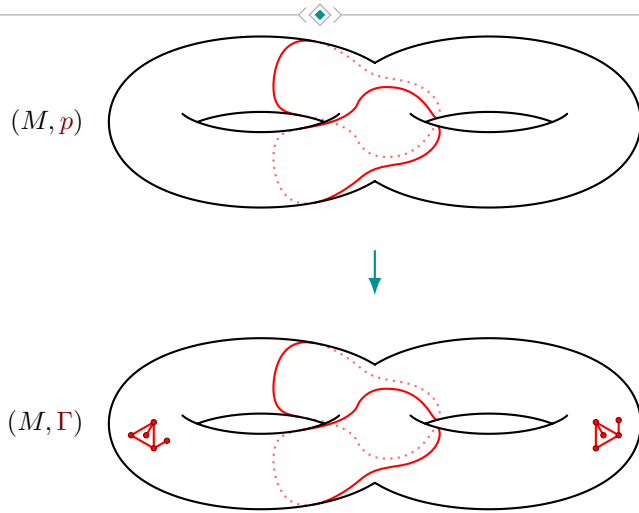
\begin{figure}
\centering
\tikzsetnextfilename{adding-chiral-graphs}
\begin{tikzpicture}[thick,torus/setup=65,vertex/.pic={\fill[main color 1,postaction={draw=main color 1!50!black,thin}] circle (1pt);}]
\pgfinterruptboundingbox
\tikzset{
torus/new={r=.5,R=1.5,name=T1},
torus/new={r=.5,R=1.5,at={(3,0)},name=T2},
torus/new={r=.5,R=1.5,at={(0,4)},name=T3},
torus/new={r=.5,R=1.5,at={(3,4)},name=T4},
torus/set canvas box={T1}{-.5}{0}{.5}{1},
torus/set canvas box={T2}{0}{0}{1}{1},
}
\path foreach \a in {0,90,180,270} { [torus/right boundary point={T1}{\a}] coordinate (b-\a) };
\path[spath/save=curve l]($(0,1)+(b-180)$) 
 to[out=-120,in=-150] ($(-.05,.05)+(b-90)$) to[out=30,in=90,in looseness=.5] ($(b-90)$);
\path[spath/save=curve l][spath/use=curve l,spath/use={curve l,transform={rotate=180,shift={(0,-1)}}}];
\path foreach \a in {0,90,180,270} { [torus/left boundary point={T2}{\a}] coordinate (b-\a) };
\path[spath/save=curve r] ($(0,1)+(b-270)$) to[out=-90,in=45,out looseness=.5] (b-180) (b-90) to[out=90,in=-135,out looseness=.5] ($(0,1)+(b-0)$);

\tikzset{
torus/split visible={curve l}{torus=T1,to visible=front l,to invisible=back l},
torus/split visible={curve r}{torus=T2,to visible=front r,to invisible=back r},
torus/path={front l}{torus=T1,to=front l},
torus/path={back l}{torus=T1,to=back l},
torus/path={front r}{torus=T2,to=front r},
torus/path={back r}{torus=T2,to=back r},
torus/sew={front l}{front r}{front}{5pt},
torus/sew={back l}{back r}{back}{5pt},
}
\endpgfinterruptboundingbox
\foreach \ys in {4,0} {
    \draw[main color 1!60,dotted] [spath/use={back,transform={shift={(0,\ys)}}}];
    \draw[main color 1] [spath/use={front,transform={shift={(0,\ys)}}}];
}
\draw[black,line cap=round,torus/outline={torus=T1,cut right},torus/outline={torus=T2,cut left}];
\draw[black,line cap=round,torus/outline={torus=T3,cut right},torus/outline={torus=T4,cut left}];
\tikzset{chiral graph/.pic={
\coordinate (a) at (90:1);
\coordinate (b) at (-30:1);
\coordinate (c) at (-150:1);
\draw[red,line join=round] (0,0) pic{vertex} -- (a) -- (b) -- (c) -- (a) (b) -- +(60:1) pic{vertex};
\foreach \p in {a,b,c} \pic at (\p) {vertex};
}}
\tikzset{torus/point={(.6,.12)}{torus=T1,to=c1},torus/point={(-.1,.12)}{torus=T2,to=c2}}
\pic[scale=.2,rotate=-30] at (c1) {chiral graph};
\pic[scale=.2,rotate=30] at (c2) {chiral graph};

\draw[->,theme color] (1.5,{2+.3}) -- (1.5,{2-.3});
\node[left] at (-2.2,4) {$(M,\textcolor{main color 1!50!black}{p})$};
\node[left] at (-2.2,0) {$(M,\textcolor{main color 1!50!black}{\Gamma})$};
\end{tikzpicture}
\caption{The homeomorphism problem can be reduced to the unoriented homeomorphism problem by adding chiral graphs on the boundary.}
\label{fig:adding chiral graphs}
\end{figure}

\subsection{Examples}\label{sec:introduction:examples}

Classification of genus-two surfaces in $\sphere[3]$ is significantly harder than the same task for tori. In fact, increasing the genus by one is enough for a set of ``wild'' topological phenomena to appear; this is in contrast with the relative ``tameness'' of tori embedded in $\sphere[3]$.
\begin{substeps}
\item First of all, it is not always true that a genus\=/two surface $S\subs\sphere[3]$ bounds a handlebody. A class of examples, as seen in \cref{fig:example no handlebody 1}, can be constructed by starting with a knotted torus $T$ and adding a tube inside the solid torus component of $\sphere[3]\setminus T$. Alternatively, one can start with two knotted tori and join them with a tube; see \cref{fig:example no handlebody 2} for an example.
\item Even if one component of $\sphere[3]\setminus S$ is a handlebody, the other is not necessarily boundary irreducible; an example is described in \cref{fig:example not boundary irreducible}, where two knotted tori are connected by a ``trivial'' tube.
\item Moreover, even in the case where one component of $\sphere[3]\setminus S$ is a handlebody and the other is boundary irreducible (like in \cref{fig:example handlebody boundary irreducible}), it is not always easy to identify a canonical meridian curve on $S$; we invite the reader to compare this with the existence of a standard meridian in tori embedded in $\sphere[3]$, which makes the classification problem significantly easier.
\item Finally, for the genus\=/one case, one can completely bypass the matter of meridian curves by exploiting the fact that tori in $\sphere[3]$ are isotopic if and only if their complementary regions are homeomorphic. This result, which follows from the work of \textcite{gordon-luecke}, does not hold for genus\=/two surfaces, even when they bound a handlebody on one side; see \cite{lee-handlebody-knots} for an example of this phenomenon.
\end{substeps}

\begin{figure}
\centering
\tikzsetnextfilename{example-no-handlebody-1}
\begin{tikzpicture}[use Hobby shortcut,thick]
\path[scale=1.67,spath/save=trefoil] ([closed]90:2) foreach \k in {1,...,3} { .. (-30+\k*240:.5) .. (90+\k*240:2) } (90:2);
\tikzset{ks/double={closed,Hobby,path=trefoil,to=thick trefoil,width={50+80*\t*(\t-1)}}}
\tikzset{ks/extract components={draft mode=false,draft mode scale=.3,path=thick trefoil,to={thick trefoil transverse}{3,7,11,17,21,25},not={thick trefoil}{5,19,13,15,9,23}}}
\pic at (spath cs:trefoil 0) {code={\begin{scope}[yshift=-2.35cm,scale=1,rotate=90]\path[spath/save global=curve endpoint 1] (2.8,.3) to[out=200,in=0,out looseness=2,in looseness=.6] (2.1,.6) to[out=180,in=180,looseness=.6] (2.2,-.6) to[out=0,in=0,in looseness=.5] (2.4,.7) to[out=180,in=0] (2.0,.3); \path[spath/save global=curve endpoint 2] (2.3,-.2) to[out=20,in=160,out looseness=1.8,in looseness=1](2.8,-.3);\end{scope}}};
\path[spath/save=curve] (60:1) arc(60:120-360:1);
\tikzset{ks/subdivide={path=curve,to=curve,n=3},ks/path along={path=curve,to=curve,along=trefoil,width to=30pt}}
\path[spath/save=curve][spath/use=curve endpoint 1] to[out=-90,in=90,out looseness=.5,in looseness=1.5] (spath cs:curve 0)  [spath/use=curve] to[out=90,in=-70,out looseness=1.5,in looseness=.5] (spath cs:{curve endpoint 2} 0) [spath/use={curve endpoint 2,weld}];
\tikzset{spath/remove empty components=curve,spath/spot weld=curve}
\tikzset{ks/double={path=curve,to=thick curve,width={5+max(0,-(\t-.3)*(\t-.9)*100)}}}
\tikzset{ks/extract components={draft mode=false,draft mode scale=.4,path=thick curve,split with self=false,split with={thick trefoil transverse},to={thick curve}{3,10,5,12},to={thick curve b}{1,8},to={thick curve e}{7,14}}}
\tikzset{ks/extract components={draft mode=false,draft mode scale=.2,path=thick curve b,to={thick curve b}{1,2,3,14,15,16},to={thick curve bxx}{9,10,11,13,22,23,24,26},to={thick curve bx}{5,6,7,18,19,20}}}
\tikzset{ks/extract components={draft mode=false,draft mode scale=.2,path=thick curve e,split with self=false,split with={thick curve bx},not={thick curve ex}{2,7}}}
\tikzset{ks/extract components={draft mode=false,draft mode scale=.2,path=thick curve bx,split with self=false,split with={thick curve e},not={thick curve bx}{4,9}}}
\path[spath/save=thick curve][spath/use/.list={thick curve b,thick curve bx,thick curve bxx,thick curve,thick curve ex}];
\draw[black!60][spath/use=thick curve];
\path[spath/use=curve] \foreach \pos/\xs in {0/1,0.9999/-1} {pic[sloped,xscale=\xs,draw=black!60,pos=\pos] {ks/tube end={5pt}{arcs/.style={black}}}};
\draw[black][spath/use=thick trefoil];
\end{tikzpicture}
\caption{A genus\=/two surface in $\sphere[3]$ which does not bound a handlebody on either side.\label{fig:example no handlebody 1}}
\end{figure}
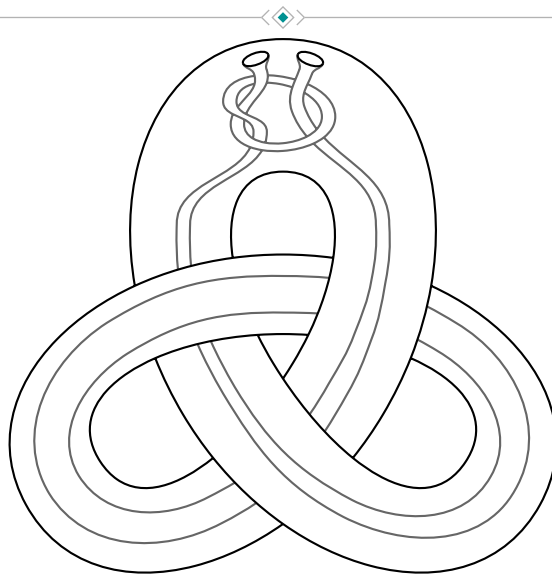

\begin{figure}
\centering
\tikzsetnextfilename{example-no-handlebody-2}
\begin{tikzpicture}[thick]
\path[spath/save=hole 1] (0,1.2) to[out=-90,in=90,out looseness=2] (-.75,-.4) to[out=-90,in=-90,out looseness=1.2,in looseness=2] (.75,0);
\path[spath/save=hole 1] [spath/use=hole 1,spath/transform={hole 1}{yscale=-1},spath/append reverse=hole 1];
\tikzset{ks/subdivide={path=hole 1,to=hole 1,n=2},ks/double={path=hole 1,to=thick hole 1,width=12}}
\tikzset{ks/extract components={draft mode=false,draft mode scale=.2,path=thick hole 1,not={thick hole 1}{8,21,12,25,4,17},to={hole transversals 1}{2,15,10,23,6,19}}}
\path[spath/use=hole 1] pic[sloped,pos=0,draw=black!60] {ks/tube end={12 pt}{disc/.style={spath/save global=tube end u 1},arcs/.style={draw=none}}} pic[sloped,pos=.999,xscale=-1,draw=black!60] {ks/tube end={12 pt}{disc/.style={spath/save global=tube end l 1},arcs/.style={draw=none}}};
\draw[black!60][spath/use=thick hole 1];
\begin{scope}[shift={(6,0)}]
\path[spath/save=hole 2] (0,1.2) to[out=-90,in=90,out looseness=2.2] (-.6,-.33) to[out=-90,in=-135] (.6,-.7) to[out=45,in=-15,out looseness=2] (0,0);
\path[spath/save global=hole 2] [spath/use=hole 2,spath/transform={hole 2}{xscale=-1,yscale=-1},spath/append reverse=hole 2];
\tikzset{ks/subdivide={path=hole 2,to=hole 2,n=2},ks/double={path=hole 2,to=thick hole 2,width=12}}
\tikzset{ks/extract components={draft mode=false,draft mode scale=.2,path=thick hole 2,not={thick hole 2}{12,29,16,33,4,21,8,25},to={hole transversals 2}{2,19,10,27,6,23,14,31}}}
\path[spath/save global=hole transversals 2,spath/use=hole transversals 2];
\path[spath/use=hole 2] pic[sloped,pos=0,draw=black!60] {ks/tube end={12 pt}{disc/.style={spath/save global=tube end u 2},arcs/.style={draw=none}}} pic[sloped,pos=.999,draw=black!60] {ks/tube end={12 pt}{disc/.style={spath/save global=tube end l 2},arcs/.style={draw=none}}};
\draw[black!60][spath/use=thick hole 2];
\end{scope}
\path[ks/subdivide={path=hole 2,to=hole 2,n=2},spath/save=curve] (45:1.6) to[out=30,in=90] (spath cs:{hole 2} 0) [spath/append=hole 2] to[out=-90,in=0,in looseness=.5] (3,-2.2) to[out=180,in=-90,out looseness=.5] (spath cs:hole 1 1) [spath/append={hole 1,reverse}] to[out=90,in=150] ($(6,0)+(135:1.6)$);
\tikzset{ks/double={path=curve,to=thick curve,width=5}}
\tikzset{ks/extract components={draft mode=false,draft mode scale=.2,path=thick curve,split with self=false,split with={hole transversals 1,hole transversals 2},to={thick curve}{3,18,5,20,7,22,11,26,13,28},to={thick curve ul}{1,16},to={thick curve ur}{15,30},to={thick curve d}{9,24}}}
\path[spath/save=spheres] (0,0) circle(2) (6,0) circle(2);
\tikzset{
ks/extract components={draft mode=false,draft mode scale=.2,path=thick curve d,split with self=false,split with={tube end l 1,tube end l 2,spheres},to={thick curve dx}{4,11},to={thick curve dxx}{1,2,8,9,6,7,13,14}},
ks/extract components={draft mode=false,draft mode scale=.2,path=tube end l 1,split with self=false,split with={thick curve d},not={tube end l 1}{2}},
ks/extract components={draft mode=false,draft mode scale=.2,path=tube end l 2,split with self=false,split with={thick curve d},not={tube end l 2}{2}},
ks/extract components={draft mode=false,draft mode scale=.2,path=thick curve ul,split with self=false,split with={tube end u 2},not={thick curve ul}{3,6},to={thick curve ulx}{3,6}},
ks/extract components={draft mode=false,draft mode scale=.2,path=thick curve ur,split with self=false,split with={tube end u 1,thick curve ul},not={thick curve ur}{1,6,4,9},to={thick curve urx}{1,6}},
ks/extract components={draft mode=false,draft mode scale=.2,path=tube end u 1,split with self=false,split with={thick curve ur},not={tube end u 1}{2}},
ks/extract components={draft mode=false,draft mode scale=.2,path=tube end u 2,split with self=false,split with={thick curve ul},not={tube end u 2}{3}},
spath/remove empty components=spheres,
ks/extract components={draft mode=false,draft mode scale=.2,path=spheres,split with self=false,split with={thick curve d,thick curve ul,thick curve ur},not={spheres}{3,5,9,11}},
}
\draw[black!60][spath/use/.list={thick curve ulx,thick curve urx,thick curve,thick curve dxx,tube end l 1,tube end l 2}];
\draw[black][spath/use/.list={thick curve dx,thick curve ul,thick curve ur,tube end u 1,tube end u 2,spheres}];
\path[spath/use=curve] \foreach \pos/\xs in {0/1,0.9999/-1} {pic[sloped,xscale=\xs,draw=black,pos=\pos] {ks/tube end={5pt}{tube above,arc 1/.style={black!60}}}};
\end{tikzpicture}
\caption{Another genus\=/two surface in $\sphere[3]$ which does not bound a handlebody on either side.\label{fig:example no handlebody 2}}
\end{figure}

\begin{figure}[htb]
\centering
\tikzsetnextfilename{example-not-boundary-irreducible}
\begin{tikzpicture}[thick]
\draw[black] circle (2cm);
\path[spath/save=hole 1] (0,1.2) to[out=-90,in=90,out looseness=2] (-.75,-.4) to[out=-90,in=-90,out looseness=1.2,in looseness=2] (.75,0);
\path[spath/save=hole 1] [spath/use=hole 1,spath/transform={hole 1}{yscale=-1},spath/append reverse=hole 1];
\tikzset{ks/subdivide={path=hole 1,to=hole 1,n=2},ks/double={path=hole 1,to=thick hole 1,width=12}}
\tikzset{ks/extract components={draft mode=false,draft mode scale=.2,path=thick hole 1,not={thick hole 1}{8,21,12,25,4,17}}}
\path[spath/use=hole 1] \foreach \pos/\xs/\as in {0/1/{black},0.9999/-1/{black!60}} {pic[sloped,xscale=\xs,draw=black!60,pos=\pos] {ks/tube end={12pt}{arcs/.style={\as}}}};
\draw[black!60][spath/use=thick hole 1];
\begin{scope}[shift={(6,0)}]
\draw[black] circle (2cm);
\path[spath/save=hole 2] (0,1.2) to[out=-90,in=90,out looseness=2.2] (-.6,-.33) to[out=-90,in=-135] (.6,-.7) to[out=45,in=-15,out looseness=2] (0,0);
\path[spath/save=hole 2] [spath/use=hole 2,spath/transform={hole 2}{xscale=-1,yscale=-1},spath/append reverse=hole 2];
\tikzset{ks/subdivide={path=hole 2,to=hole 2,n=2},ks/double={path=hole 2,to=thick hole 2,width=12}}
\tikzset{ks/extract components={draft mode=false,draft mode scale=.2,path=thick hole 2,not={thick hole 2}{12,29,16,33,4,21,8,25}}}
\path[spath/use=hole 2] \foreach \pos/\as in {0/{black},0.9999/{black!60}} {pic[sloped,draw=black!60,pos=\pos] {ks/tube end={12pt}{arcs/.style={\as}}}};
\draw[black!60][spath/use=thick hole 2];
\end{scope}
\path[spath/save=handle] (1.7,-.2) to[out=-15,in=-165] ({6-1.7},-.2);
\tikzset{ks/subdivide={path=handle,to=handle,n=2},ks/double={path=handle,to=thick handle,to contour=handle contour,width={5+15*sin(\t*180)}}}
\fill[white] [spath/use=handle contour];
\draw[black] [spath/use=thick handle];
\path[spath/use=handle] \foreach \pos/\xs in {0/1,0.9999/-1} {pic[sloped,xscale=\xs,draw=black,pos=\pos] {ks/tube end={5pt}{tube above,arc 1/.style={black!60}}}};
\end{tikzpicture}
\caption{A genus\=/two surface splitting $\sphere[3]$ into two components: one is a handlebody, but the other is not boundary irreducible.\label{fig:example not boundary irreducible}}
\end{figure}

\begin{figure}[htb]
\centering
\tikzsetnextfilename{example-handlebody-boundary-irreducible}
\begin{tikzpicture}[thick]
\path[use as bounding box] (-1.8,-2.2) (4.2,2.2);
\path[spath/save=curve,use Hobby shortcut] ([out angle=180]1,1) .. (120:.5) .. (-120:2) .. (0:.5) .. (120:2) .. (-120:.5) .. ([in angle=180]1,-1);
\tikzset{ks/subdivide={path=curve,to=curve,n=2},ks/double={path=curve,to=thick curve,width=10,to 1=thick curve 1,to 2=thick curve 2}}
\path[spath/save=ring 1,use Hobby shortcut] ([out angle=0]spath cs:thick curve 1 0) .. (1.5,.5) .. (2.5,-.4) .. (3.3,.4) .. (2.5,1.2) .. ([in angle=0]spath cs:thick curve 2 0) (2.5,.4) circle (.5);
\tikzset{spath/remove empty components=ring 1,spath/clone={ring 2}{ring 1},spath/transform={ring 2}{yscale=-1}}
\tikzset{
ks/extract components={draft mode=false,draft mode scale=.3,path=ring 1,split with self=false,split with=ring 2,not={ring 1}{2,6}},
ks/extract components={draft mode=false,draft mode scale=.3,path=ring 2,split with self=false,split with=ring 1,not={ring 2}{2,5}},
ks/extract components={draft mode=false,draft mode scale=.3,path=thick curve,not={thick curve}{8,21,4,17,12,25}}
}
\draw[black] [spath/use/.list={ring 1,ring 2,thick curve}];
\end{tikzpicture}
\caption{A genus\=/two surface splitting $\sphere[3]$ into two components: one is a handlebody, and the other is boundary irreducible.\label{fig:example handlebody boundary irreducible}}
\end{figure}

\section{Homeomorphisms of \texorpdfstring{$3$\=/manifolds}{3-manifolds}}\label{sec:homeomorphisms}

\subsection{Rationale}

\Cref{sec:homeomorphisms} is dedicated to the study of the mapping class groups of $3$\=/manifolds from an algorithmic point of view. These groups, of course, are often infinite, but they can be described with finite amount of data. To be more specific, it has been known for a long time (see \cite[Corollary 27.6]{johannson-jsj}) that the mapping class group of an irreducible sufficiently large $3$\=/manifold contains a finite\=/index subgroup $H$ which is generated by Dehn twists about annuli and tori. By going through the proofs in \citeauthor{johannson-jsj}'s book, one realises that, in fact, only finitely many Dehn twists\todo{How many different Dehn twists about a torus are there?} are required to generate $H$. Therefore, theoretically, the mapping class group of a $3$\=/manifold can be fully described by providing finitely many Dehn twists about annuli and tori, together with representatives of the cosets of $H$ in the mapping class group.

There is, however, a substantial gap between acknowledging the existence of $H$ and being able to algorithmically find generators of $H$ and representatives of its cosets. The proofs by \citeauthor{johannson-jsj} are for the most part constructive, and one could carefully convert them into algorithms for finding the generating Dehn twists (whereas representatives of the cosets of $H$ are somewhat trickier to extract from said proofs). This is exactly the path we will follow in \cref{sec:homeomorphisms}, albeit with a few \emph{caveats}.
\begin{substeps}
\item Instead of working in the full generality of what \citeauthor{johannson-jsj} calls ``$3$\=/manifolds with complete and useful boundary pattern'', we restrict our attention to \emph{irreducible $3$\=/manifold pairs}, following the work of \textcite{jaco-shalen-jsj}. We admit that our approach is perhaps less elegant and powerful that \citeauthor{johannson-jsj}'s, but the decrease in flexibility is compensated by the (relative) conciseness of some of the algorithmic procedures we will describe.
\item We make use of the geometrisation theorem (\cite{thurston-geometrisation,morgan-geometrisation}) and \textcite{kuperberg-homeomorphism}'s excellent exposition in order to compute the (finite) mapping class group of simple $3$\=/manifolds. This could probably be avoided, following \citeauthor{johannson-jsj}'s inductive argument which relies on hierarchies. We decided to go for the shorter solution, rather than the one which was more faithful to \citeauthor{johannson-jsj}'s original work.
\item As will become clear in \cref{sec:classification algorithm}, we are only interested in the trace of the mapping class group, and care little about what happens in the interior of the $3$\=/manifold. This is why \cref{thm:mapping class group of 3-manifold} is stated the way it is. It is true that little additional effort would have been required to give a complete description of the mapping class group, but once again we decided to avoid exceeding in generality, in order to keep this section reasonably short.
\end{substeps}

In \cref{sec:homeomorphisms:jsj decomposition}, we start by recalling the definition of JSJ system for an irreducible sufficiently large $3$\=/manifold pair, following \textcite{jaco-shalen-jsj} as closely as possible. \Cref{sec:homeomorphisms:seifert,sec:homeomorphisms:i-bundle,sec:homeomorphisms:simple} address the computation of the mapping class groups and the homeomorphism problem for Seifert fibred\=/spaces, $I$\=/bundles, and simple manifolds respectively. In \cref{sec:homeomorphisms:computing jsj}, we show how to actually compute the JSJ decomposition of a $3$\=/manifold pair (with some additional assumptions). Finally, in \cref{sec:homeomorphisms:piecing}, we put the results of the previous section together to deliver, as anticipated, a description of the mapping class group in the form of \cref{thm:mapping class group of 3-manifold}.

\subsection{JSJ decomposition}\label{sec:homeomorphisms:jsj decomposition}

This section is little more than a restatement of the definitions and the main theorem of \cite[Chapter V, \textsection 6]{jaco-shalen-jsj}.

\begin{definition}
For $n\ge 2$, a \emph{$3$\=/manifold $n$\=/tuple} is an $n$\=/tuple $(M,R_1,\ldots,R_{n-1})$ where $M$ is a $3$\=/manifold, and $R_1,\ldots,R_{n-1}$ are surfaces in $\boundary M$ such that $R_i\cap R_j$ is a collection of curves for $1\le i<j\le n-1$. We say \emph{pair} and \emph{triple} instead of $2$\=/tuple and $3$\=/tuple respectively.
\end{definition}

For a $3$\=/manifold pair $(M,R)$, we give the following definitions.
\begin{itemize}
\item We say that $(M,R)$ is \emph{irreducible} if $M$ is irreducible and $R$ is incompressible in $M$.
\item We say that $(M,R)$ is sufficiently large if $M$ is.
\item A surface $F\subs M$ is \emph{properly embedded} in $(M,R)$ if it is properly embedded in $M$ and $\boundary F$ lies in $R$.
\item Two disjoint surfaces $F$ and $F'$ in $M$ with $\boundary F$ and $\boundary F'$ lying in $R$ are \emph{parallel} in $(M,R)$ if there is a $3$\=/manifold $W\subs M$ homeomorphic to $F\times I$ such that $\boundary_h W=F\cup F'$ and $\boundary W\setminus\boundary_h W\subs R$.
\item A surface $F$ in $M$ is \emph{boundary parallel} in $(M,R)$ if it is parallel in $(M,R)$ to a surface in $R$.
\end{itemize}

Our aim is to develop an algorithmic theory for the JSJ decomposition of irreducible sufficiently large $3$\=/manifold pairs. Loosely speaking, as is the case for ordinary $3$\=/manifolds, the decomposition for a pair $(M,R)$ comes from an incompressible surface $F$ properly embedded in $(M,R)$ such that cutting $M$ along it yields pieces which are in some sense ``simpler'' than the original pair. Note that these pieces immediately inherit a $3$\=/manifold pair structure by considering the remnants of $R$ as a surface in the boundary of $M\cut F$. By doing so, however, we lose information about what parts of $\boundary M$ come from $F$. This is why, whenever we have a $3$\=/manifold pair $(M,R)$ with a properly embedded surface $F$, we naturally think of $M\cut F$ as a $3$\=/manifold triple. More precisely, we define $(M,R)\cut F$ to be the $3$\=/manifold triple $(M',R',F')$, where:
\begin{itemize}
\item $M'=M\cut F=M\setminus\onbhd{F}$;
\item $R'=R\cap M'$;
\item $F'=\nbhd{F}\cap M'$.
\end{itemize}

We now proceed to define the types of pieces we allow in the JSJ decomposition of a $3$\=/manifold pair.

\begin{definition}
A $3$\=/manifold $n$\=/tuple $(M,\vec{R})$ is a \emph{Seifert $n$\=/tuple} if $\boundary M=R_1\cup\ldots\cup R_{n-1}$ and $M$ admits a Seifert fibration such that $R_1,\ldots,R_{n-1}$ are unions of fibres. Such a fibration is called a \emph{Seifert fibration} for $(M,\vec{R})$. A $3$\=/manifold $n$\=/tuple  equipped with a Seifert fibration is called a \emph{Seifert\=/fibred $n$\=/tuple}.
\end{definition}

\begin{definition}
A $3$\=/manifold triple is an \emph{$I$\=/bundle triple} if it is homeomorphic to $(X,\boundary_h X,\boundary_v X)$ for some $I$\=/bundle $X$ over a surface.
\end{definition}

\begin{definition}
A $3$\=/manifold triple $(M,R,F)$ is a \emph{simple triple} if the following hold:
\begin{itemize}
\item every incompressible torus properly embedded in $M$ is parallel in $M$ to a component of $R$ or of $F$;
\item every incompressible annulus $A$ properly embedded in $M$ with $\boundary A\subs\interior{R}$ is parallel in $(M,R)$ to an annulus in $R$ or in $F$.
\end{itemize}
\end{definition}

We can finally define the JSJ system for a $3$\=/manifold pair. As expected, the JSJ system is unique up to isotopy.

\begin{definition}\label{def:jsj-system}
Let $(M,R)$ be an irreducible sufficiently large $3$\=/manifold pair. A \emph{JSJ system} of $(M,R)$ is a surface $F$ properly embedded in $(M,R)$, possibly empty, satisfying the following properties.
\begin{enumroman}
\item\label[property]{def:jsj system:1} Each component of $F$ is an incompressible torus or annulus, and it is not boundary parallel in $(M,R)$.
\item\label[property]{def:jsj system:2} Each component of $(M,R)\cut F$ is a Seifert triple, an $I$\=/bundle triple, or a simple triple.
\item\label[property]{def:jsj system 3} The surface $F$ is minimal with respect to inclusion among all surfaces properly embedded in $(M,R)$ satisfying \cref{def:jsj system:1,def:jsj system:2}.
\end{enumroman}
\end{definition}

\begin{theorem}[{\cite[Chapter V,\textsection 6, Generalised Splitting Theorem]{jaco-shalen-jsj}}]\label{thm:jsj}
Let $(M,R)$ be an irreducible sufficiently large $3$\=/manifold pair. Then $(M,R)$ admits a JSJ system, which is unique up to isotopy in $M$ fixing $\boundary M\setminus\interior{R}$.
\end{theorem}

We now embark on an in\=/depth study of the pieces of the JSJ decomposition, with the aim of achieving a better understanding of the original $3$\=/manifold pair. Specifically, for a piece $(N,R',F')$ of $(M,R)\cut F$, we  will be interested in the following algorithmic problems:
\begin{itemize}
\item solving the homeomorphism problem for $(N,R',F')$;
\item deciding whether a homeomorphism $\umap{F'}{F'}$ extends to a self\=/homeomorphism of $(N,R',F')$;
\item describing the group $\fixhomeo{F'}{N,R'}$ in a computationally feasible way (that is, with a finite amount of data).
\end{itemize}

\subsection{Seifert pieces}\label{sec:homeomorphisms:seifert}

We start with an analysis of Seifert pieces. In fact, we will work with general Seifert $n$\=/tuples instead of just triples. The reason will become apparent in \cref{sec:homeomorphisms:piecing}, when we will need to deal with Seifert $4$\=/tuples as well. We follow the convention that a \emph{Seifert manifold} is a $3$\=/manifold which admits a Seifert fibration, whereas a \emph{Seifert\=/fibred manifold} is a $3$\=/manifold endowed with a fixed Seifert fibration.

Now is a good time to briefly discuss a ``computationally friendly'' definition of Seifert fibration. For our algorithmic purposes, a Seifert fibration of a triangulated $3$\=/manifold is given by a $2$\=/subcomplex $Z$ of (some subdivision of) $M$ satisfying the following properties.
\begin{enumroman}
\item\label[property]{def:computational seifert:1} Every point of $Z$ has a neighbourhood which looks like one of the local models in \cref{fig:seifert 2-complex models}; the set of points whose local models are of type \subref{fig:seifert 2-complex models:b} or \subref{fig:seifert 2-complex models:c} is a union of circles called \emph{fibres}.
\item\label[property]{def:computational seifert:2} The closure of each component of $M\setminus Z$ is a solid torus, and has at least one fibre on its boundary.
\item\label[property]{def:computational seifert:3} No fibre bounds a disc in $M\setminus Z$.
\end{enumroman}
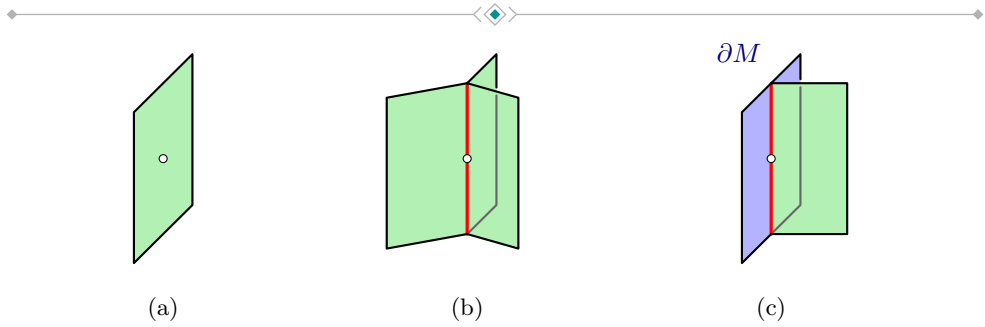
\begin{figure}
\centering
\begin{subcaptiongroup}
\begin{tikzpicture}[thick,line join=round,line cap=round,vertex/.pic={\fill[pic actions,postaction={draw=black,thin}] circle (1.5pt);}]
\def\sep{4}
\begin{scope}[shift={(0*\sep,0)}]
\draw[fill=main color 3!30] (0,-1,-1) -- (0,-1,1) -- (0,1,1) -- (0,1,-1) -- cycle;
\pic[white]{vertex};
\end{scope}
\begin{scope}[shift={(1*\sep,0)}]
\foreach \k/\name in {30/a,150/b,270/c} {
	\fill[main color 3!30,spath/save global=\name] (0,-1,0) -- ({cos(\k)},-1,{sin(\k)}) coordinate (\name-l) -- +(0,2,0) coordinate (\name-u) -- (0,1,0);
}
\path[spath/save=front] (0,1,0) -- (a-u);
\tikzset{spath/split at intersections with={c}{front},spath/get components of={c}{\components}}
\edef\cu{\getComponentOf{\components}{2}}
\edef\cl{\getComponentOf{\components}{1}}
\tikzset{spath/shorten at end={\cl}{1.5pt},spath/shorten at start={\cu}{1.5pt}}
\draw[black!60][spath/use=\cl];
\draw[main color 1,very thick] (0,-1,0) -- (0,1,0);
\draw[black] [spath/use/.list={a,b,\cu}];
\pic[white]{vertex};
\end{scope}
\begin{scope}[shift={(2*\sep,0)}]
\fill[main color 2!30,spath/save=a] (0,-1,-1) -- (0,-1,1) -- (0,1,1) -- (0,1,-1) -- cycle;
\fill[main color 3!30,spath/save=b] (0,-1,0) -- (1,-1,0) -- (1,1,0) -- (0,1,0);
\tikzset{spath/split at intersections with={a}{b},spath/join components={a}{2},spath/get components of={a}{\components}}
\edef\au{\getComponentOf{\components}{1}}
\edef\al{\getComponentOf{\components}{2}}
\tikzset{spath/shorten at end={\au}{1.5pt},spath/shorten at start={\al}{1.5pt}}
\draw[black!60][spath/use=\al];
\draw[main color 1,very thick] (0,-1,0) -- (0,1,0);
\draw[spath/use/.list={\au,b}];
\pic[white]{vertex};
\node[main color 2!50!black,left] at (0,1.4) {$\boundary M$};
\end{scope}
\foreach \k/\name in {0/a,1/b,2/c} {
\phantomcaption{}
\label{fig:seifert 2-complex models:\name}
\node at ({\k*\sep},-2) {\captiontext*{}};
}
\end{tikzpicture}
\end{subcaptiongroup}
\caption{Local models for a $2$\=/complex defining a Seifert fibration for a $3$\=/manifold $M$. The $2$\=/complex is coloured green, while the fibres are highlighted in red.}
\label{fig:seifert 2-complex models}
\end{figure}

On one hand, if $M$ is equipped with a Seifert fibration in the usual sense and $\map{p}{M}{B}$ is the projection to the base surface, we can obtain a $2$\=/complex $Z\subs M$ by taking the preimage under $p$ of a suitable $1$\=/complex cutting $B$ into discs which contain at most one singular point. Conversely, given a $2$\=/complex $Z$ as above, we can recover the Seifert fibration by fibring each complementary solid torus compatibly with the fibres of $Z$. Note that \cref{def:computational seifert:1,def:computational seifert:2,def:computational seifert:3} are easy to check algorithmically. As a consequence, given a Seifert manifold, we can find a Seifert fibration for it. This discussion naturally extends to a Seifert $n$\=/tuple $(M,\vec{R})$; in this case, we additionally require that $\boundary R_i$ is a union of fibres of $Z$ for $1\le i\le n-1$.

With the following we provide a version of \cite[Lemma VI.19]{jaco-lectures} which better suits our needs.
\begin{lemma}\label{thm:seifert manifolds fibre preserving on the boundary}
Let $M$ and $M'$ be Seifert\=/fibred manifolds. Let $\map{f}{M}{M'}$ be a homeomorphism, and suppose that there is a non\=/empty finite union of fibres, fibred annuli, and fibred tori $X\subs\boundary M$ such that the restriction of $f$ to $X$ is fibre\=/preserving. Then $f$ can be isotoped, fixing $X$ pointwise, to a fibre\=/preserving homeomorphism.
\end{lemma}
\begin{proof}
The proof is essentially identical to that of \citeauthor{jaco-lectures}, and we present it here for the sake of completeness. Let $\map{p'}{M'}{B'}$ be the projection to the base surface of the Seifert fibration, and let $M'$ have $n$ exceptional fibres. The proof proceeds by induction on $-\chi(B')+n$.

We can easily isotope $f$, fixing $X$, to be fibre\=/preserving on each component of $\boundary M$ which intersects $X$. Therefore, if $Y\subs\boundary M$ is the union of all such components, then we can assume that $\map{f|_Y}{Y}{\boundary M'}$ is fibre\=/preserving, and prove that $f$ can be isotoped, fixing $Y$, to a fibre\=/preserving homeomorphism.

If $B'$ is a disc and $n\le 1$, then $M$ and $M'$ are fibred solid tori of the same type, and it is easy to construct the desired isotopy. If instead $B'$ is not a disc or $n\ge 2$, then let $y'\in\boundary B'$ the image under $p'$ of one of the fibres in $f(Y)$. There exists an arc $a'$ properly embedded in $B'$ connecting $y'$ to a different point in $\boundary B'$ and avoiding exceptional points, such that $a'$ is not trivial in $\pi_1(B'_{\bullet},\boundary B')$, where $B'_{\bullet}$ is the punctured surface obtained by removing the exceptional points from $B'$. Then $A'=(p')^{-1}(a')$ is an incompressible boundary incompressible vertical annulus properly embedded in $M'$. The annulus $A=f^{-1}(A')$ is incompressible and boundary incompressible in $M$. It need not be vertical, but at least one of the components of $\boundary A$ is a fibre, since it lies in $Y$. We can isotope $A$ to be vertical in $M$, and by \cite[Proposition 5.6]{johannson-jsj}\footnote{It is immediate to see that exceptions 5.1.1--5.1.5 listed by \citeauthor{johannson-jsj} do not occur, since $M$ has non\=/empty boundary and the case of a solid torus with a single exceptional fibre has already been addressed.}\todo{Would like to find a better reference, without boundary patterns.} we can assume that the isotopy fixes $A\cap Y$ pointwise. In fact, it is not hard to see that we can even take the isotopy to fix all of $Y$. As a consequence, we find that $f$ restricts to a homeomorphism $\umap{M\cut A}{M'\cut A'}$ which is fibre\=/preserving on $Y\cap(M\cut A)$ as well as on the two annuli $\closure{\boundary(M\cut A)\setminus \boundary M}$. The conclusion now follows by induction.
\end{proof}

The next two propositions generalise well known facts about Seifert spaces -- namely uniqueness of the Seifert fibration and solution to the homeomorphism problem -- to the setting of Seifert $n$\=/tuples.

\begin{proposition}\label{thm:unique seifert fibration}
Let $(M,\vec{R})$ be a Seifert $(n+1)$\=/tuple with non\=/empty boundary. Then $(M,\vec{R})$ admits exactly one Seifert fibration up to isotopy in $(M,\vec{R})$, unless the surfaces $R_1,\ldots,R_n$ are unions of tori and one of the following holds.
\begin{enumarabic}
\item\label[exception]{thm:unique seifert fibration:1} The $3$\=/manifold $M$ is a solid torus; in this case, every fibration of $\boundary M$ which does not contain a meridian of $M$ extends to a unique Seifert fibration of $M$ with at most one exceptional fibre; the fibration of $\boundary M$ containing a meridian of $M$ does not extend to a Seifert fibration of $M$.
\item\label[exception]{thm:unique seifert fibration:2} The $3$\=/manifold $M$ is homeomorphic to $\torus[2]\times I$; in this case, every fibration of $\torus[2]\times\{0\}$ extends to a unique fibration of $M$.
\item\label[exception]{thm:unique seifert fibration:3} The $3$\=/manifold $M$ is homeomorphic to $\nsurf{2}\twtimes I$; in this case, $M$ admits two Seifert fibrations, one over a M\"obius band and one over a disc.
\end{enumarabic}
\end{proposition}
\begin{proof}
First of all, let us remark that, as soon as one of the surfaces $R_1,\ldots,R_n$ has an annulus component, then the Seifert fibration for $(M,\vec{R})$ is necessarily unique: given two such fibrations, the homeomorphism $\map{\id}{M}{M}$ is fibre\=/preserving on $X=\boundary R_1\cup\ldots\cup\boundary R_n$; by \cref{thm:seifert manifolds fibre preserving on the boundary}, we can isotope one fibration to the other in $(M,R_1,\ldots,R_n)$.

When instead all the components of $R_1,\ldots,R_n$ are tori, it is clear that two Seifert fibrations for $(M,\vec{R})$ are isotopic in $(M,\vec{R})$ if and only if they are isotopic in $M$. The conclusion then follows from \cite[Theorem VI.18]{jaco-lectures}\todo{Not quite, it doesn't address the exceptions.}.
\end{proof}
\begin{remark}\label{rmk:inequivalent seifert fibration on i-bundle over klein bottle}
The two Seifert fibrations of $\nsurf{2}\twtimes I$ are \emph{inequivalent}, by which we mean that no self\=/homeomorphism of $\nsurf{2}\twtimes I$ can send one to the other. In other words, if we fix a Seifert fibration for $\nsurf{2}\twtimes I$, then every self\=/homeomorphism of this $3$\=/manifold can be isotoped to be fibre\=/preserving.
\end{remark}

\begin{proposition}\label{thm:seifert homeomorphism problem}
There is an algorithm which, given as input two Seifert $(n+1)$\=/tuples $(M,\vec{R})$ and $(M',\vec{R}')$, decides whether they are homeomorphic or not.
\end{proposition}
\begin{proof}
Fix two Seifert fibrations for $(M,\vec{R})$ and $(M',\vec{R}')$. It is very well known (for a thorough and complete discussion, see \cite[Theorem 10.4.19]{martelli-geometric-topology}) that computing the Seifert invariants for the fibrations of $M$ and $M'$ leads to a straightforward algorithm for deciding whether $M$ and $M'$ are homeomorphic. If they are not, then clearly neither are $(M,\vec{R})$ and $(M',\vec{R}')$. If $M$ and $M'$ are closed and homeomorphic, then we are done. On the other hand, if $(M',\vec{R}')$ is one of the exceptions described in \ref{thm:unique seifert fibration:1} and \ref{thm:unique seifert fibration:2} of \cref{thm:unique seifert fibration}, then answering the homeomorphism problem is easy.

Therefore, let us assume that $(M',\vec{R}')$ has non\=/empty boundary and is not one of these exceptions, and let $\map{f}{M}{M'}$ be a homeomorphism. We can find a finite set $\FFF_0$ of fibre\=/preserving self\=/homeomorphisms of $M'$ such that every permutation of the boundary components of $M'$ is induced by an element of $\FFF_0$. Let $\map{\iota}{M'}{M'}$ be a fibre\=/preserving homeomorphism such that $\trace{\iota}$ acts as $-\id$ on $H_1(\boundary M')$, and define
\[
\FFF=\bigcup_{g\in\FFF}\{g,\iota g\}.
\]
We claim that $(M,\vec{R})$ and $(M',\vec{R}')$ are homeomorphic if and only if $\map{g^{-1}f}{M}{M'}$ can be isotoped to a homeomorphism $\umap{(M,\vec{R})}{(M',\vec{R}')}$ for some $g\in\FFF$ -- and note that this condition is easy to check algorithmically. The reverse implication is trivial. Conversely, assume that there is a homeomorphism
\[
\map{f'}{(M,\vec{R})}{(M',\vec{R}')}.
\]
We can isotope $f$ in $M$ so that $\map{f'f^{-1}}{M'}{M'}$ is fibre\=/preserving. There is a homeomorphism $g\in\FFF$ such that $f'f^{-1}g$ preserves each boundary component of $M'$, as well as the orientation of the fibres on $\boundary M'$. But then the homeomorphism $f'f^{-1}g$ acts as a power of the Dehn twist about a fibre on each boundary component of $M'$, hence it can be isotoped to a self\=/homeomorphism of $(M',\vec{R}')$. We immediately conclude that $g^{-1}f$ can be isotoped to a homeomorphism $\umap{(M,\vec{R})}{(M',\vec{R}')}$.
\end{proof}

We are now ready to give an algorithmic description of the mapping class group of (most) Seifert $n$\=/tuples with boundary.

\begin{proposition}\label{thm:mapping class group of seifert spaces}
There is an algorithm which, given as input a Seifert\=/fibred $(n+2)$\=/tuple $(M,\vec{R},F)$ with non\=/empty boundary and not homeomorphic to one of the exceptions described in \ref{thm:unique seifert fibration:1} and \ref{thm:unique seifert fibration:2} of \cref{thm:unique seifert fibration}, returns
\begin{itemize}
\item a finite collection $\FFF$ of self\=/homeomorphisms of $(\boundary M,\vec{R})$ fixing $F$ pointwise, and
\item a finite collection $\CCC=\{(a_1,b_1),\ldots,(a_m,b_m)\}$, where $a_1,\ldots,a_m,b_1,\ldots,b_m$ are  pairwise disjoint fibres in $\interior{R_1}\cup\ldots\cup\interior{R_n}$,
\end{itemize}
such that
\[
\trace{\fixhomeo{F}{M,\vec{R}}}=\bigcup_{f\in\FFF}\langle\twist{a_1}\twist{b_1}^{-1},\ldots,\twist{a_m}\twist{b_m}^{-1}\rangle f
\]
as a subgroup of $\fixhomeo{F}{\boundary M,\vec{R}}$.
\end{proposition}
\begin{proof}
Let $\map{p}{(M,\vec{R},F)}{(B,\bar{\vec{R}},\bar{F})}$ be the projection to the base surface of the Seifert fibration, and consider a homeomorphism $f\in\fixhomeo{F}{M,\vec{R}}$. We can isotope $f$ to be fibre\=/preserving fixing $F$ pointwise, using \cref{thm:unique seifert fibration,rmk:inequivalent seifert fibration on i-bundle over klein bottle} if $F$ is empty and \cref{thm:seifert manifolds fibre preserving on the boundary} otherwise. Denote by $\bar{f}$ the induced self\=/homeomorphism of $(B,\bar{\vec{R}},\bar{F})$, which is not necessarily orientation\=/preserving.
\step{Identity on $B$.} As we will see, the role of $\FFF$ is to encode the action of $\fixhomeo{F}{M,\vec{R}}$ on the space $(\boundary B,\bar{\vec{R}},\bar{F})$. First of all, note that we can easily list all permutations of the components of $\boundary B$ which are induced by elements of $\fixhomeo{F}{M,\vec{R}}$. Therefore, we can compute a set of representatives $\FFF_0\subs\fixhomeo{F}{M,\vec{R}}$ such that, up to replacing $f$ with $ff_0$ for some $f_0\in\FFF_0$, we can assume that $f$ sends each component of $\boundary M$ to itself.

The homeomorphism $f$ induces a self\=/homeomorphism of $(\boundary B,\bar{\vec{R}},\bar{F})$ which fixes $\bar{F}$ pointwise and sends each component of $\boundary B$ to itself. If $F$ is non\=/empty, the homeomorphism $\bar{f}$ acts orientation\=/preservingly on each component of $\boundary B$; otherwise, it will consistently preserve or reverse the orientations of these components. Either way, there are only finitely many such self\=/homeomorphisms of $(\boundary B,\bar{\vec{R}},\bar{F})$, and they are all induced by elements of $\fixhomeo{F}{M,\vec{R}}$. We can then compute a set $\FFF_1\subs\fixhomeo{F}{M,\vec{R}}$ such that, up to replacing $f$ with $ff_1$ for some $f_1\in\FFF_1$, we can assume that $\bar{f}$ is the identity on $\boundary B$.

In fact, up to further composing $f$ with a self\=/homeomorphism of $M$ fixing $\boundary M$ pointwise, we can assume that $\bar{f}$ is the identity on all of $B$; this follows from the fact that if a self\=/homeomorphism of $B$ fixing $\boundary B$  lifts to a self\=/homeomorphism of $M$, then it lifts to a self\=/homeomorphism of $M$ fixing $\boundary M$.

\step{Identity on almost all of $\boundary M$.} Let $\XXX$ be the set of components of $\boundary M\setminus\interior{F}$. For each pair $X,Y$ of different components in $\XXX$, let $A_{X,Y}$ be a vertical annulus connecting $X$ and $Y$. We can of course assume that the boundary curves of all these annuli lie in $\interior{R_1}\cup\ldots\cup\interior{R_n}$ and are pairwise disjoint. Moreover, for each component $X\in\XXX$, we can define $k_X$ to be the unique integer such that $f|_X$ is isotopic to the composition of $k_X$ Dehn twists of $X$ about a fibre, through an isotopy which fixes $\boundary X$ pointwise. Up to composing with powers of Dehn twists about the annuli $A_{X,Y}$ -- which induce the identity on $B$ -- we can assume that $k_X=0$ for all components $X\in\XXX$ but one. In other words, we may assume that the restriction of $f$ to $\boundary M\setminus\interior{X}$ is the identity for some component $X\in\XXX$.

\step{Identity on $\boundary M$.} In fact, we now prove that $f$ restricts to the identity on all of $\boundary M$. Let $N_1,\ldots,N_m$ be fibred regular neighbourhoods of the exceptional fibres of $M$. Without loss of generality, we may assume that $f$ restricts to the identity on $N_1\cup\ldots\cup N_m$. Define $M'=M\setminus\interior{N_1\cup\ldots\cup N_m}$, which is an $\sphere[1]$\=/bundle over the surface $B'=B\setminus\interior{p(N_1\cup\ldots\cup N_m)}$; since $B'$ has non\=/empty boundary, the bundle is uniquely determined by $B'$. The homeomorphism $f$ restricts to $\map{f'}{M'}{M'}$ such that $f'$ is the identity on $\boundary M'\setminus\interior{X}$.
\begin{itemize}
\item Let us first deal with the case where $B$ is orientable. Let $S$ be a horizontal copy of $B'$ in $M'$. At the level of homology, we have that
\[
0=f'_{\ast}[\boundary S]=[\boundary S]+k_X\cdot[c]=k_X\cdot[c]\in H_1(M'),
\]
where $[c]$ is the homology class of a fibre. We conclude that $k_X=0$ and that $f$ does indeed restrict to the identity on $\boundary M$.
\item Assume now that $B$ -- and hence $B'$ -- is non\=/orientable. Let $\wtilde{M'}$ be the double covering of $M'$ which is a product $\sphere[1]$\=/bundle over the orientable double covering $\wtilde{B'}$ of $B'$. Denote by $\map{q}{\wtilde{M'}}{M'}$ the covering map, and by $\map{\iota}{\wtilde{M'}}{\wtilde{M'}}$ the non\=/trivial deck transformation of $q$. The homeomorphism $f'$ lifts to $\map{\wtilde{f'}}{\wtilde{M'}}{\wtilde{M'}}$, which is the identity on $\boundary\wtilde{M'}\setminus q^{-1}(\interior{X})$; denote by $\wtilde{X}_1$ and $\wtilde{X}_2$ the two components of $q^{-1}(X)$, and by $k_1$ and $k_2$ the integers such that $\wtilde{f'}|_{\wtilde{X}_i}$ is the $k_i$\=/th power of the Dehn twist about a fibre for $i\in\{1,2\}$. Let $\wtilde{S}$ be a horizontal copy of $\wtilde{B'}$ in $\wtilde{M'}$. Fixing an orientation of $\wtilde{S}$ allows us to define canonically oriented horizontal curves $a_1$ and $a_2$, which are the boundary components of $\wtilde{S}$ intersecting $\wtilde{X}_1$ and $\wtilde{X}_2$ respectively. Moreover, let $b_1$ and $b_2$ be fibres of $\wtilde{X}_1$ and $\wtilde{X}_2$, oriented so that they intersect $a_1$ and $a_2$ with positive sign. We can then define $k_i$ as the integer such that $[\wtilde{f'}(a_i)]=[a_i]+k_i\cdot[b_i]$ in $H_1(\boundary\wtilde{M'})$ for $i\in\{1,2\}$. The homological argument described in the orientable case shows that $k_1+k_2=0$. On the other hand, note that $[\iota(a_1)]=-[a_2]$ and $[\iota(b_1)]=-[b_2]$. Then the relation $\wtilde{f'}\iota=\iota\wtilde{f'}$ readily implies that $k_1=k_2$, since
\begin{align*}
[\wtilde{f'}\iota(a_1)]&=-[\wtilde{f'}(a_2)]=-[a_2]-k_2\cdot[b_2],\\
[\iota\wtilde{f'}(a_1)]&=[\iota(a_1)]+k_1\cdot[\iota(b_2)]=-[a_2]-k_1\cdot[b_2].
\end{align*}
We conclude that $\wtilde{f'}$ restricts to the identity on $\boundary\wtilde{M'}$, and hence $f$ restricts to the identity on $\boundary M$.
\end{itemize}

\step{Conclusion.} We have decomposed $f$ as a product
\[
t\circ h\circ f_1\circ f_0\in\fixhomeo{F}{M,\vec{R}},
\]
where $f_0\in\FFF_0$, $f_1\in\FFF_1$, $h\in\fixhomeo{\boundary M}{M}$, and $t$ is a product of powers of Dehn twists of $M$ about the annuli $A_{X,Y}$ for $X,Y\in\XXX$. The algorithm will then return
\belowdisplayskip=-12pt\begin{align*}
\FFF=\{\trace{(f_0f_1)}:f_0\in\FFF_0,f_1\in\FFF_1\}&&\text{and}&&\CCC=\{\boundary A_{X,Y}:\text{$X,Y\in\XXX$ and $X\neq Y$}\}.
\end{align*}\qedhere
\end{proof}

As an easy consequence, we can also solve the extension problem for Seifert $n$\=/tuples.

\begin{corollary}\label{thm:seifert extension of homeomorphisms}
Let $(M,\vec{R},F)$ and $(M',\vec{R}',F')$ be Seifert\=/fibred $(n+2)$\=/tuples not homeomorphic to one of the exceptions described in \ref{thm:unique seifert fibration:1} and \ref{thm:unique seifert fibration:2} of \cref{thm:unique seifert fibration}, and let $\map{f}{F}{F'}$ be a fibre\=/preserving homeomorphism. There is an algorithm which, given as input the two $(n+2)$\=/tuples and $f$, decides whether $f$ extends to a (not necessarily fibre\=/preserving) homeomorphism
\[
\umap{(M,\vec{R},F)}{(M',\vec{R}',F')}.
\]
\end{corollary}
\begin{proof}\todo{Might need serious rewriting.}
It suffices to be able to answer the following two questions algorithmically.
\begin{enumarabic}
\item Are the $(n+2)$\=/tuples $(M,\vec{R},F)$ and $(M',\vec{R}',F')$ homeomorphic? Note that if they are, then we can choose the fibration on $(M,\vec{R},F)$ so that they are fibre\=/preservingly so.
\item Given a fibre\=/preserving homeomorphism $\map{g}{F}{F}$, does it extend to a self\=/homeomorphism of $(M,\vec{R},F)$?
\end{enumarabic}
The first question is addressed in \cref{thm:seifert homeomorphism problem}. Thanks to \cref{thm:mapping class group of seifert spaces}, we can reduce the second to the following problem (note that if $\boundary M=\emptyset$ we are already done).
\begin{itemize}
\item[\tikzenumlabel{2'}] Given fibres $a_1,\ldots,a_m,b_1,\ldots,b_m$ in $\boundary M$, does $g$ extend to a self\=/homeomorphism of $(\boundary M,\vec{R},F)$ belonging to the subgroup of $\homeo{\boundary M,\vec{R},F}$ generated by $\twist{a_1}\twist{b_1}^{-1},\allowbreak\ldots,\allowbreak\twist{a_m}\twist{b_m}^{-1}$?
\end{itemize}

If $g$ permutes the components of $F$ in a non\=/trivial way or it reverses the orientation of the fibres on some component of $F$ then the answer is no, since Dehn twists along fibres preserve components of $F$ and the orientation of the fibres. Otherwise, the homeomorphism $g$ acts like the identity on each annulus component of $F$, and is an integer power of the Dehn twist about a fibre on each torus component of $F$. Moreover, for each torus component $T$ of $F$, we can compute the integer $k_T$ such that $g|_T$ is the composition of $k_T$ Dehn twists about a fibre. Clearly, the answer to the extension problem for $g$ only depends on the values $k_T$, and is surely yes if they are all equal to zero.

Let us define a graph $\Gamma$, whose vertices are in one\=/to\=/one correspondence with components of $\boundary M$. Each pair $(a_i,b_i)$ defines an edge of $\Gamma$, joining the two vertices corresponding to the two (not necessarily distinct) boundary tori of $M$ containing $a_i$ and $b_i$.  It is not hard to see that, if $T$ is a torus component of $F$, the pair $(a_i,b_i)$ represents an edge of $\Gamma$ with $a_i\subs T$, and $X$ is the component of $\boundary M$ which contains $b_i$, then the following operations do not change the answer to the extension problem for $g$:
\begin{itemize}
\item if $X\subs F$, then increase $k_T$ by $1$ and decrease $k_X$ by $1$, or vice versa;
\item otherwise, increase or decrease $k_T$ by $1$.
\end{itemize}
By the proof of \cref{thm:mapping class group of seifert spaces}, the graph $\Gamma$ is connected. Therefore, we immediately conclude that the answer to the extension problem is yes if and only if one of the following condition holds:
\begin{enumroman}
\item the boundary of $M$ is not completely covered by $F$;
\item the boundary of $M$ is completely covered by $F$, and $\sum_{T}k_T=0$.\qedhere
\end{enumroman}
\end{proof}

\subsection{\texorpdfstring{$I$\=/bundle}{I-bundle} pieces}\label{sec:homeomorphisms:i-bundle}

We now move on to the -- somewhat easier -- study of $I$\=/bundle pieces. We first show how to solve the extension problem, and then address the mapping class group in \cref{thm:mapping class group of i-bundles}.

\begin{proposition}\label{thm:i-bundle extension of homeomorphisms}
Let $(M,R,F)$ and $(M',R',F')$ be $I$\=/bundle triples, and let $\map{f}{F}{F'}$ be a homeomorphism. There is an algorithm which, given as input the two triples, decides whether $f$ extends to a homeomorphism
\[
\umap{(M,R,F)}{(M',R',F')}.
\]
\end{proposition}
\begin{proof}
It suffices to be able to solve the following two questions algorithmically.
\begin{enumarabic}
\item Are the triples $(M,R,F)$ and $(M',R',F')$ homeomorphic?
\item Given a homeomorphism $\map{g}{F}{F}$, does it extend to a self\=/homeomorphism of $(M,R,F)$?
\end{enumarabic}
The first question can be easily answered, since we can compute the base surfaces of the $I$\=/bundles $(M,R,F)$ and $(M',R',F')$ and check if they are the same. As far as the second is concerned, we may assume that $g$ is fibre\=/preserving. Moreover, since the group $\homeo{M,R,F}$ acts transitively on the components of $F$, we can further assume that $g$ maps each component of $F$ to itself. Now, the restriction of $g$ to an annulus component $A$ of $F$ can either preserve or swap the boundary components of $A$. If $g$ consistently preserves or swaps the boundary components of every component of $F$, then it extends to a self\=/homeomorphism of $(M,R,F)$; otherwise, it does not.
\end{proof}

\begin{proposition}\label{thm:mapping class group of i-bundles}
There is an algorithm which, given as input an $I$\=/bundle triple $(M,R,F)$, returns
\begin{itemize}
\item a finite collection $\FFF$ of self\=/homeomorphisms of $(\boundary M,R,F)$ fixing $F$ pointwise, and
\item a finite collection $\CCC=\{(a_1,b_1),\ldots,(a_m,b_m)\}$, where $a_1,\ldots,a_m,b_1,\ldots,b_m$ are curves in $\interior{R}$ with $a_1\cap b_1=\ldots=a_m\cap b_m=\emptyset$,
\end{itemize}
such that
\[
\trace{\fixhomeo{F}{M,R}}=\bigcup_{f\in\FFF}\langle\twist{a_1}\twist{b_1}^{-1},\ldots,\twist{a_m}\twist{b_m}^{-1}\rangle f
\]
as a subgroup of $\fixhomeo{F}{\boundary M,R}$. Moreover, if $(M,R,F)$ is the product $I$\=/bundle over $\surf{0}$, $\surf{0,1}$, $\surf{0,2}$, or $\surf{0,3}$, or the twisted $I$\=/bundle over $\nsurf{1}$, $\nsurf{1,1}$, or $\nsurf{1,2}$, then the curves in $\CCC$ are pairwise disjoint.
\end{proposition}
\begin{proof}
Let $\map{p}{(M,F)}{(B,\boundary B)}$ be the projection to the base surface of the $I$\=/bundle, and consider a homeomorphism $f\in\fixhomeo{F}{M,R}$. We can isotope $f$ to be fibre\=/preserving fixing $F$ pointwise\todo{Why would I prove \cref{thm:seifert manifolds fibre preserving on the boundary} and not this? Is this even true when $\boundary B=\emptyset$?}. Denote by $\bar{f}$ the induced self\=/homeomorphism of $B$, which is not necessarily orientation\=/preserving. We now distinguish two cases, depending on whether $B$ is orientable or not.
\begin{substeps}
\item If $B$ is orientable, then $\bar{f}$ is necessarily orientation\=/preserving, unless $F=\emptyset$. If this is the case, let us choose a homeomorphism $\iota\in\homeo{M,R}$ such that the induced homeomorphism $\map{\bar{\iota}}{B}{B}$ is orientation\=/reversing; otherwise, we set $\map{\iota=\id}{M}{M}$. Up to composing $f$ with $\iota$, we can assume that $\bar{f}$ is orientation\=/preserving. There is a well\=/known explicit set of curves $c_1,\ldots,c_m\subs B$ such that $\twist{c_1},\ldots\twist{c_m}$ generate $\fixhomeo{\boundary B}{B}$ -- see, for instance, \cite[Section 4.4.4]{farb-margalit-primer}. As a consequence, the homeomorphism $f$ belongs to the subgroup of $\fixhomeo{F}{M,R}$ generated by $\twist{A_1},\ldots,\twist{A_m}$, where $A_i=p^{-1}(c_i)$ for $1\le i\le m$. The algorithm will then return
\begin{align*}
\FFF=\{\id,\trace{\iota}\}&&\text{and}&&\CCC=\{\boundary A_i:1\le i\le m\}.
\end{align*}
Note that, if $B$ is one of $\surf{0}$, $\surf{0,1}$, $\surf{0,2}$, or $\surf{0,3}$, then the curves $c_1,\ldots,c_m$ can be chosen to be disjoint.
\item If $B$ is non\=/orientable, denote by $\pmfixhomeo{\boundary B}{B}$ the group of self\=/homeomorphisms of $B$ fixing $\boundary B$ pointwise, modulo isotopies through homeomorphisms of the same kind. \textcite[Section 3]{chillingworth-mapping-class-group-non-orientable} and \textcite[Proposition 3.2]{kobayashi-mapping-class-group-non-orientable-boundary}%
, respectively for the cases $\boundary B=\emptyset$ and $\boundary B\neq\emptyset$, provide an explicit set of two\=/sided curves $c_1,\ldots,c_m\subs B$ and a homeomorphism $y\in\pmfixhomeo{\boundary B}{B}$ such that
\[
\pmfixhomeo{\boundary B}{B}=\langle\twist{c_1},\ldots,\twist{c_m},y\rangle.
\]
The homeomorphism $y$ is a \emph{Y\=/homeomorphism}, a description of which can be found in \cite[Section 2]{lickorish-mapping-class-group-non-orientable}; in particular, we have that $y^2$ is a Dehn twist about a curve $c_0$ which can be computed explicitly. Easy algebraic manipulations show that
\[
\pmfixhomeo{\boundary B}{B}=\bigcup_{g\in\{\id,y\}}\langle\twist{c_1},\ldots,\twist{c_m},y\twist{c_1}y^{-1},\ldots,y\twist{c_m}y^{-1},y^2\rangle g.
\]
Define $c_{m+i}=y(c_i)$ for $1\le i\le m$, so that $y\twist{c_i}y^{-1}=\twist{c_{m+i}}$; let $h\in\fixhomeo{F}{M,R}$ be a homeomorphism such that $\bar{h}=y$. The algorithm will then return
\begin{align*}
\FFF=\{\id,\trace{h}\}&&\text{and}&&\CCC=\{\boundary p^{-1}(c_i):0\le i\le 2m\}.
\end{align*}
Note that, if $B$ is one of $\nsurf{1,0}$, $\nsurf{1,1}$, or $\nsurf{1,2}$, then the curves $c_1,\ldots,c_m$ can be chosen to be disjoint.\qedhere
\end{substeps}
\end{proof}

\subsection{Simple pieces}\label{sec:homeomorphisms:simple}

We finally turn our attention to simple pieces. As anticipated, instead of using \citeauthor{johannson-jsj}'s original approach with hierarchies, we will prove the finiteness of the mapping class group of these pieces by means of hyperbolic geometry and the geometrisation theorem. While sacrificing generality, this strategy will allow us to bypass the algorithmically challenging induction procedure of \citeauthor{johannson-jsj}.


In \cref{sec:homeomorphisms:seifert,sec:homeomorphisms:i-bundle} we dealt with Seifert $n$\=/tuples and $I$\=/bundle triples in full generality (save for a few exceptions). For simple triples, instead, we will restrict our attention to a subclass satisfying additional conditions. This will, of course, impose further restrictions on the type of $3$\=/manifold pairs which our final algorithm will accept as input. However, we will maintain a level of generality which will be sufficient for our purposes. The additional constraints are listed in the following definition.

\begin{definition}
A $3$\=/manifold triple $(M,R,F)$ is \emph{reasonable} if it satisfies the following properties:
\begin{itemize}
\item the pair $(M,R)$ is irreducible;
\item no component of $R$ is a disc;
\item the surface $F$ is incompressible in $M$;
\item each component of $F$ is a torus or an annulus, and is not boundary parallel in $(M,R)$;
\item each component of $\boundary M\setminus\interior{R\cup F}$ is an annulus.
\end{itemize}
\end{definition}

Part of the appeal of simple reasonable triples is that the two surfaces $R$ and $F$ cover the whole boundary of the $3$\=/manifold $M$, save for a few explicit exceptions.

\begin{lemma}\label{thm:reasonable triple alternative}
Let $(M,R,F)$ be a reasonable $3$\=/manifold triple. Then one of the following holds.
\begin{enumroman}
\item The triple $(M,R,F)$ is not simple.
\item The surfaces $R$ and $F$ cover the boundary of $M$.
\item The $3$\=/manifold $M$ is a solid torus, and $\boundary M\setminus\interior{R\cup F}$ is a union of annuli which are incompressible in $M$.
\item The $3$\=/manifold $M$ is homeomorphic to $\torus[2]\times I$ in such a way that $\torus[2]\times\{1\}$ is a component of $R$ or $F$.
\end{enumroman}
\end{lemma}
\begin{proof}
Suppose that $(M,R,F)$ is simple, and that $\boundary M\setminus\interior{R\cup F}$ is non\=/empty. Let $A$ be a component of this surface; by assumption, $A$ is an annulus. Let $T$ be the boundary component of $M$ containing $A$; we show that $T$ is a torus. Let $A'$ be the component of $\boundary M\setminus\interior{R}$ containing $A$. If $A'$ is a torus, then $A'=T$ and we are done. Otherwise, the surface $A'$ is an annulus. If $A'$ were compressible in $M$, then one component of $\boundary A'\subs\boundary R$ would bound a disc in $M$; since $R$ is incompressible and has no disc components, this is impossible. Therefore, $A'$ is incompressible. Since the triple $(M,R,F)$ is simple, it follows that $T$ is indeed a torus.
\begin{itemize}
\item If $T$ is compressible, then $M$ is in fact a solid torus, and the components of $R$, $F$, and $\boundary M\setminus\interior{R\cup F}$ are parallel annuli in $\boundary M$; by assumption, these annuli are incompressible in $M$.
\item If $T$ is incompressible, then it must be parallel in $(M,R)$ to a torus in $R$ or in $F$. This immediately implies that $M$ is homeomorphic to $\torus[2]\times I$, where $\torus[2]\times\{0\}$ is identified with $T$ and $\torus[2]\times\{1\}$ is a component of either $R$ or $F$.\qedhere
\end{itemize}
\end{proof}

As a final restriction, we only want to deal with simple triples which do not belong to the classes we have already analysed -- namely, $I$\=/bundles and Seifert pieces.

\begin{definition}
A $3$\=/manifold triple $(M,R,F)$ is \emph{strongly simple} if it is simple, it is not an $I$\=/bundle triple, and $(M,R,F,\boundary M\setminus\interior{R\cup F})$ is not a Seifert $4$\=/tuple.
\end{definition}

Under the additional assumptions we have listed, we can now prove that the mapping class group of a simple triple is finite, and present an algorithm to compute it. At the same time, we will describe a solution to the homeomorphism problem for simple triples.

\begin{proposition}\label{thm:strongly simple homeomorphism}
The following hold for reasonable strongly simple $3$\=/manifold triples $(M,R,F)$ and $(M',R',F')$.
\begin{enumarabic}
\item\label[statement]{thm:strongly simple homeomorphism:1} There is an algorithm which, given as input the triples $(M,R,F)$ and $(M',R',F')$, decides whether they are homeomorphic.
\item\label[statement]{thm:strongly simple homeomorphism:2} The group $\homeo{M,R,F}$ is finite. Moreover, there is an algorithm which, given as input the triple $(M,R,F)$, returns representatives of $\homeo{M,R,F}$.
\end{enumarabic}
\end{proposition}
\begin{proof}
By \cref{thm:reasonable triple alternative}, we immediately see that, if $R$ and $F$ do not cover the boundary of $M$, the $4$\=/tuple $(M,R,F,\boundary M\setminus\interior{R\cup F})$ admits a Seifert fibration. Therefore, we have that $\boundary M=R\cup F$.

\step{Boundary incompressibility of $F$.} We claim that $F$ is boundary incompressible in $(M,R)$ in the following sense: for every disc $D$ in $M$ whose boundary can be written as the union of two arcs $a$ and $b$ with $a\cap b=\boundary a=\boundary b$, $D\cap F=a$, and $D\cap R=b$, we have that $a$ cuts a disc off of $F$.

In fact, suppose by contradiction that $D$ is a disc as described above, with $a$ an essential arc in an annulus component $A$ of $F$. Let $A'$ be the disc obtained by boundary compressing $A$ along $D$, that is $A'=(A\setminus\boundary_v\nbhd{D})\cup\boundary_h\nbhd{D}$. By irreducibility of $(M,R)$, the disc $A'$ cobounds a $3$\=/ball $B$ with some disc in $R$. If $B$ is disjoint from $D$ then $A$ is boundary parallel. If instead $B$ contains $D$ then $A$ is compressible, contradicting incompressibility of $F$.

\step{The easy case.} Let us deal first with the case where every component of $R$ is a torus; as a consequence, the same holds for the components of $F$. By assumption, we have that $M$ is irreducible, boundary irreducible, atoroidal\footnote{There are several slightly different definitions of ``atoroidal'' in the literature. In this article, by ``$X$ is atoroidal'' we mean that every incompressible torus in $X$ is boundary parallel in $X$ or, equivalently, that the triple $(X,\emptyset,\boundary X)$ is simple. Note that other definitions could be inequivalent in general, but they all agree with ours when $X$ is not a Seifert manifold.}, and not a Seifert manifold; additionally, every boundary component of $M$ is a torus. By geometrisation, the interior of $M$ admits a complete finite\=/volume hyperbolic structure. We can now invoke \cite[Theorem 6.1]{kuperberg-homeomorphism} and conclude that:
\begin{itemize}
\item if $(M',R',F')$ is another reasonable strongly simple $3$\=/manifold triple such that each component of $R'$ and $F'$ is a torus, then we can algorithmically decide whether $M$ and $M'$ are homeomorphic or not;
\item the group $\homeo{M}$ is finite and can be computed algorithmically.
\end{itemize}

Note that $\homeo{M,R,F}$ is a subgroup of $\homeo{M}$, and checking if an element of $\homeo{M}$ belongs to $\homeo{M,R,F}$ is trivial; this is enough to prove \cref{thm:strongly simple homeomorphism:2}. As far as \cref{thm:strongly simple homeomorphism:1} is concerned, we simply check whether $M$ and $M'$ are homeomorphic. If they are, we find a homeomorphism $\map{f}{M}{M'}$ and consider the (finitely many) elements of $\homeo{M'}f$: if any of these gives a homeomorphism $\umap{(M,R,F)}{(M',R',F')}$ then the two triples are homeomorphic, otherwise they are not.

\step{Doubling $M$.} We now turn to the more involved case where not all components of $R$ are tori. Let $R_0$ be the union of the components of $R$ which are not tori. Let $\bar{M}$ be the result of gluing two copies of $M$ along $R_0$. More precisely, we define
\[
\bar{M}=\quot{M\times\{0,1\}}{\{(x,0)\sim(x,1):x\in R_0\}}.
\]
By assumption, the surface $R_0$ is non\=/empty, so the $3$\=/manifold $\bar{M}$ is connected. Note that every boundary component of $\bar{M}$ is either a copy of a torus component of $R$, a copy of a torus component of $F$, or the union of two copies of an annulus component of $F$. As a consequence, we see that every component of $\boundary\bar{M}$ is a torus.

The surface $R_0$ has a canonical copy $\bar{R_0}$ in $\bar{M}$ -- specifically, the image in the quotient of $R_0\times\{0\}$; this surface is incompressible in $\bar{M}$. More generally, if $S$ is a surface in $M$, we let $\bar{S}$ be the ``doubled surface'', that is the image in the quotient of the surface $S\times\{0,1\}\subs M\times\{0,1\}$. Similarly, if $\map{f}{(M,R_0)}{(M,R_0)}$ is a function, we let $\map{\bar{f}}{\bar{M}}{\bar{M}}$ be the function defined by $\bar{f}(x,i)=(f(x),i)$ for $(x,i)\in M\times\{0,1\}$.

Note that $\bar{M}$ comes equipped with a natural orientation\=/reversing involution $\map{\iota}{\bar{M}}{\bar{M}}$ defined by $\iota(x,i)=(x,1-i)$; for every homeomorphism $\map{f}{(M,R,F)}{(M,R,F)}$, we have the relation $\bar{f}\iota=\iota\bar{f}$. Finally, let us remark that, since $\bar{R_0}$ is incompressible, the $3$\=/manifold $\bar{M}$ is sufficiently large.

\step{The double is irreducible and boundary irreducible.} Irreducibility of $\bar{M}$ follows immediately from the fact that $M$ is irreducible and $\bar{R_0}$ is incompressible in $\bar{M}$.

Let now $D$ be a disc properly embedded in $\bar{M}$. Up to isotopy, we can assume that $D$ is in general position with respect to $\bar{R_0}$. Additionally, let us isotope $D$ so that the number of components of $D\cap\bar{R_0}$ is as small as possible. A standard innermost circle argument on $D$ shows that we can assume that all these components are arcs. If $D\cap\bar{R_0}$ is empty, then $D$ can be interpreted as a compression disc for either $R$ or $F$ in $M$, hence it must be trivial in $M$ and in $\bar{M}$ as well. Otherwise, let $a$ be an arc in $D\cap\bar{R_0}$ which is outermost in $D$; the arc $a$ cuts off a disc $E$ from $D$ such that $E\cap\bar{R_0}=a$. Note that $a\subs\bar{R_0}$, and $\boundary D\setminus\interior{a}$ is an arc in $\bar{F}$. The disc $E$ can then be interpreted as a boundary compression disc for $F$ in $(M,R)$. By boundary incompressibility of $F$, it is easy to see that $D$ can be isotoped to remove $a$ from the intersection $D\cap\bar{R_0}$, thus contradicting our minimality assumption.

\step{The double is strongly simple.} More precisely, we prove that the triple $(\bar{M},\bar{R_1},\bar{F})$ is strongly simple, where $R_1=R\setminus R_0$ is the unions of the torus components of $R$. First, let us show that $\bar{M}$ is not a Seifert manifold. If by contradiction $\bar{M}$ admits a Seifert fibration, then up to isotopy we can assume that $\bar{R_0}$ is either horizontal or vertical. Since $M$ is homeomorphic to either component of $\bar{M}\cut\bar{R_0}$, it is easy to see that:
\begin{itemize}
\item if $\bar{R_0}$ is horizontal, then $M$ inherits an $I$\=/bundle structure; the surfaces $R$ and $F$ have no torus components, and in fact they are the horizontal and vertical boundary respectively of the fibration inherited by $M$;
\item if $\bar{R_0}$ is vertical, then $M$ is itself Seifert\=/fibred, in such a way that $R$ and $F$ are unions of fibres.
\end{itemize}
Both cases are incompatible with $(M,R,F)$ being strongly simple.

We now address the ``torus'' condition for simple triples. Let $T$ be an incompressible torus in $\bar{M}$. By the equivariant torus theorem of \textcite[Corollary 4.6]{holzmann-equivariant-torus}, we can assume that $T$ is either disjoint from or equal to $\iota(T)$.
\begin{itemize}
\item If $T$ is disjoint from $\iota(T)$, then we can think of $T$ as a torus in $M$; incompressibility of $\bar{R_0}$ in $\bar{M}$ implies that $T$ is incompressible in $M$. Since $(M,R,F)$ is a simple triple, $T$ must be parallel in $M$ to a component of $R$ or $F$. But then $T$ is boundary parallel in $\bar{M}$ as well.
\item If $T=\iota(T)$, then $T=\bar{A}$ for some annulus $A$ properly embedded in $(M,R)$. The annulus $A$ is incompressible in $M$ since $T$ is incompressible in $\bar{M}$. Then $A$ must be parallel in $(M,R)$ to an annulus in $R$ or in $F$; the first case is impossible, for otherwise $T$ would be compressible in $\bar{M}$. Therefore, $A$ is parallel in $(M,R)$ to an annulus component $A'$ of $F$. We conclude that $T=\bar{A}$ is parallel to the boundary component $\bar{A'}$ of $\bar{M}$.
\end{itemize}

Finally, let us consider the ``annulus'' condition. Let $A$ be an incompressible annulus properly embedded in $(\bar{M},\bar{R_1})$. If $A$ is boundary compressible then it is boundary parallel. Otherwise, by the equivariant annulus theorem of \textcite{kobayashi-equivariant-annulus}, we can assume that $A$ is either disjoint from or equal to $\iota(A)$.
\begin{itemize}
\item If $A$ is disjoint from $\iota(A)$, then we can think of $A$ as an incompressible annulus properly embedded in $(M,R)$. By simplicity of $(M,R,F)$, this implies that $A$ is boundary parallel in $M$ and, hence, in $\bar{M}$ as well.
\item If $A=\iota(A)$, then $A=\bar{A_0}$ for some incompressible annulus $A_0\subs M$ whose boundary intersects two boundary components of $M$, which contradicts simplicity of $(M,R,F)$.
\end{itemize}

\step{The double is hyperbolic.} It is not hard to verify that the triple $(\bar{M},\bar{R_1},\bar{F})$ is reasonable. Since, by definition, every component of $\bar{R_1}$ is a torus, we have already proved \cref{thm:strongly simple homeomorphism:1,thm:strongly simple homeomorphism:2} for this triple. In the process, we also noted that the interior of $\bar{M}$ admits a complete finite\=/volume hyperbolic metric. In fact, for our purposes, it will be more convenient to think of $\bar{M}$ as a compact hyperbolic $3$\=/manifold whose boundary components are flat tori -- in other words, as the result of truncating the cusps of a complete finite\=/volume non\=/compact hyperbolic $3$\=/manifold. By combining Mostow\=/Prasad rigidity with \cite[Theorem 7.1]{waldhausen-on-irreducible-manifolds}, we have the following\todo{A bit hand\=/wavy, but I think I have a rigorous proof. Nonetheless, this looks like the sort of statement people will believe without questioning it (\citeauthor{kuperberg-homeomorphism} does).}.
\begin{equation}\label{eqn:strongly simple homeomorphism:hyperbolic rigidity}
\parbox{\dimexpr\linewidth-4em}{%
Every self\=/homeomorphism of $\bar{M}$ is isotopic to a unique isometry $\umap{\bar{M}}{\bar{M}}$.
}
\end{equation}
As a consequence, let us remark that we can choose the hyperbolic metric on $\bar{M}$ in such a way that $\iota$ is an isometry. In fact, the involution $\iota$ is isotopic to an isometry $\map{\hat{\iota}}{\bar{M}}{\bar{M}}$. Note that $\hat{\iota}$ is itself an involution: the isometry $\hat{\iota}\circ\hat{\iota}$ is isotopic to $\iota\circ\iota=\id$ and, therefore, equal to the identity. We now apply \cite[Theorem B]{tollefson-involutions-sufficiently-large}\footnote{Using \citeauthor{tollefson-involutions-sufficiently-large}'s notation, the case where $\beta\neq\id$ does not apply here, since the fundamental group of complete finite volume hyperbolic $3$\=/manifolds has trivial centre.} to deduce that $\hat{\iota}=h\iota h^{-1}$ for some homeomorphism $\map{h}{\bar{M}}{\bar{M}}$. Then $\iota$ is an isometry with respect to the pull\=/back of the hyperbolic metric on $\bar{M}$ by $h$.

\step{Solving the homeomorphism problem.} Consider another reasonable strongly simple $3$\=/manifold triple $(M',R',F')$; suppose that, additionally, there is at least one component of $R'$ which is not a torus -- otherwise $(M,R,F)$ and $(M',R',F')$ could not possibly be homeomorphic.  Define $\bar{M'}$, $\iota'$, and $R_1'$ like we did with $\bar{M}$, $\iota$, and $R_1$. Recall that $\bar{M'}$ can be endowed with a finite\=/volume hyperbolic metric with flat boundary, with respect to which $\iota'$ is an isometry. Naturally, statement \cref{eqn:strongly simple homeomorphism:hyperbolic rigidity} also holds for homeomorphisms $\umap{\bar{M}}{\bar{M'}}$.

If the triples $(\bar{M},\bar{R_1},\bar{F})$ and $(\bar{M'},\bar{R_1'},\bar{F'})$ are not homeomorphic -- which we can algorithmically decide thanks to \cref{thm:strongly simple homeomorphism:1} -- then clearly neither are $(M,R,F)$ and $(M',R',F')$. Otherwise, let
\[
\map{f}{(\bar{M},\bar{R_1},\bar{F})}{(\bar{M'},\bar{R_1'},\bar{F'})}
\]
be a homeomorphism. We claim that the following are equivalent:
\begin{enumroman}
\item\label{thm:strongly simple homeomorphism:homeomorphism problem:1} there is a homeomorphism $\map{g}{(M,R,F)}{(M',R',F')}$ such that
\[
\map{\bar{g}}{(\bar{M},\bar{R_1},\bar{F})}{(\bar{M'},\bar{R_1'},\bar{F'})}
\]
is isotopic to $f$;
\item\label{thm:strongly simple homeomorphism:homeomorphism problem:2} the homeomorphism $\map{f^{-1}\iota'f\iota}{\bar{M}}{\bar{M}}$ is isotopic to the identity.
\end{enumroman}
It is clear that \ref{thm:strongly simple homeomorphism:homeomorphism problem:1} implies \ref{thm:strongly simple homeomorphism:homeomorphism problem:2}. Conversely, if \ref{thm:strongly simple homeomorphism:homeomorphism problem:2} holds, we can assume up to isotopy that $f$ restricts to an isometry on the interior of $\bar{M}$. But then $f^{-1}\iota'f\iota$ is an isometry of $\bar{M}$ which is isotopic to the identity and, hence, equal to it. In other words, the following diagram commutes:
\begin{diagram}
(\bar{M},\bar{R_1},\bar{F})\dar{\iota}\rar{f}&(\bar{M'},\bar{R_1'},\bar{F'})\dar{\iota'}\\
(\bar{M},\bar{R_1},\bar{F})\rar{f}&(\bar{M'},\bar{R_1'},\bar{F'}).
\end{diagram}
We then conclude that $f$ is induced by a homeomorphism $\map{g}{M}{M'}$; it is easy to see that $g$ must send $F$ to $F'$ and, therefore, $R$ to $R'$.

In order to decide whether $(M,R,F)$ and $(M',R',F')$ are homeomorphic, we can then carry out the following procedure. Let $f_0$ be any homeomorphism from $(\bar{M},\bar{R_1},\bar{F})$ to $(\bar{M'},\bar{R_1'},\bar{F'})$. Thanks to \cref{thm:strongly simple homeomorphism:1}, we can compute representatives for the finite group $\homeo{\bar{M'},\bar{R_1'},\bar{F'}}$. For each element $f$ of $\homeo{\bar{M'},\bar{R_1'},\bar{F'}}f_0$, we check whether $f^{-1}\iota'f\iota$ is isotopic to the identity of $\bar{M}$ (we can do this algorithmically since $\homeo{\bar{M},\bar{R_1},\bar{F}}$ is finite and we can compute representatives for it). If this happens for some $f$, then $(M,R,F)$ and $(M',R',F')$ are homeomorphic, otherwise they are not.

\step{Computing the mapping class group.} By the very same argument as the one presented in the previous step, we can compute representatives of the subgroup
\[
\{f\in\homeo{\bar{M},\bar{R_1},\bar{F}}:\text{$f$ is isotopic to $\bar{g}$ for some $g\in\homeo{M,R,F}$}\}\le\homeo{\bar{M},\bar{R_1},\bar{F}}.
\]
We claim that the group homomorphism $\umap{\homeo{M,R,F}}{\homeo{\bar{M},\bar{R_1},\bar{F}}}$ sending $g$ to $\bar{g}$ is injective; in other words, if $\bar{g}$ is isotopic to the identity in $\bar{M}$ then $g$ is isotopic to the identity in $(M,R,F)$.

\begin{substeps}
\item First of all, if $\bar{g}$ is isotopic to the identity, then we can isotope $g$ in $(M,R,F)$ so that it restricts to the identity on $R_1\cup F$. As a consequence, the homeomorphism $\bar{g}$ restricts to the identity on $\boundary\bar{M}$.
\item Let $\map{\pi}{\wtilde{M}}{\bar{M}}$ be the universal covering; we will think of $\wtilde{M}$ as a closed subset of hyperbolic $3$\=/space $\hypspace{3}$, whose complement is a collection of open horoballs. We define a retraction $\map{r}{\hypspace{3}}{\wtilde{M}}$ as follows. If $x$ is a point of $\wtilde{M}$ we set $r(x)=x$. If instead $x$ is contained in a horoball $B$, we set $r(x)=\gamma\cap\boundary B$, where $\gamma$ is the geodesic emanating from the centre of $B$ and passing through $x$. Finally, if $x$ and $y$ are two (not necessarily distinct) points of $\wtilde{M}$, we define $\map{\gamma_{x,y}}{[0,1]}{\hypspace{3}}$ to be the geodesic such that $\gamma_{x}(0)=x$ and $\gamma_{x,y}(1)=y$.
\item Let $\map{\wtilde{g}}{\wtilde{M}}{\wtilde{M}}$ be a homeomorphism that is a lift of $\bar{g}$. Since $\bar{g}$ is isotopic to the identity in $\bar{M}$ and restricts to the identity on $\boundary\bar{M}$, we can choose $\wtilde{g}$ so that it restricts to the identity on $\boundary\wtilde{M}$. Let us define a homotopy $\map{\wtilde{g}_t}{\wtilde{M}}{\wtilde{M}}$ as $\wtilde{g}_t(x)=r(\gamma_{x,\wtilde{g}(x)}(t))$, so that $\wtilde{g}_0=\id$ and $\wtilde{g}_1=\wtilde{g}$. It is easy to check that $\wtilde{g}_t$ is invariant under deck transformations of $\pi$, hence it descends to a homotopy $\map{\bar{g}_t}{\bar{M}}{\bar{M}}$ between the identity and $\bar{g}$. By definition, for each $t\in[0,1]$ the homeomorphism $\bar{g}_t$ fixes $\boundary\bar{M}$ pointwise.
\item Since $\iota\bar{g}=\bar{g}$ and the homotopy $\bar{g}_t$ has been defined only in terms of the metric, we must have that $\iota\bar{g}_t=\bar{g}_t$ for every $t\in[0,1]$. In particular, we find that $\bar{g}_t(\bar{R_0})\subs\bar{R_0}$, from which we deduce that $g|_{R_0}$ is homotopic to the identity in $R_0$, through a homotopy that fixes $\boundary R_0$. By a classical result of \textcite[Theorem 6.4]{epstein-surface-homeomorphisms}, this implies that $g|_{R_0}$ is actually isotopic to the identity in $R_0$, and we can therefore assume that $\map{\trace{g}}{\boundary M}{\boundary M}$ is the identity.
\item Note that the homotopy $\wtilde{g}_t$ preserves the two components of $\bar{M}\cut\bar{R_0}$ for each $t\in[0,1]$. Therefore, there is an induced homotopy $\map{g_t}{M}{M}$ between $g$ and the identity, which restricts to the identity on $\boundary M$ for each $t\in[0,1]$. By \cite[Theorem 7.1]{waldhausen-on-irreducible-manifolds}, this implies that $g$ is isotopic to the identity in $(M,R,F)$.
\end{substeps}

As a final remark, note that, given $f\in\homeo{\bar{M},\bar{R_1},\bar{F}}$, we can algorithmically find $g\in\homeo{M,R,F}$ such that $f$ is isotopic to $\bar{g}$, provided that one exists. Therefore, we have presented a complete algorithm to compute representatives of $\homeo{M,R,F}$.
\end{proof}

\subsection{Computing the JSJ decomposition}\label{sec:homeomorphisms:computing jsj}

After analysing the types of pieces individually, we now show how to actually compute the JSJ decomposition. The next three propositions explain how to recognise Seifert pieces, $I$\=/bundle pieces and (strongly) simple pieces respectively. Once we have these three algorithms, computing the JSJ decomposition is straightforward, as the proof of \cref{thm:computing jsj decomposition} witnesses.

\begin{proposition}\label{thm:seifert recognition}
There is an algorithm which, given a $3$\=/manifold $(n+1)$\=/tuple $(M,\vec{R})$ as input, decides whether it is a Seifert $(n+1)$\=/tuple or not.
\end{proposition}
\begin{proof}
If $R_1\cup\ldots\cup R_n\neq\boundary M$, or one of the surfaces $R_1,\ldots,R_n$ has a component which is not a torus or an annulus, then the answer is no. We can use \cite[Algorithm 8.2]{jaco-jsj-algorithm} to decide whether $M$ is a Seifert manifold or not. If it is not then we are done. If $M$ is homeomorphic to a solid torus, a product $\torus[2]\times I$, or a twisted $I$\=/bundle $\nsurf{2}\twtimes I$, then we can easily decide if $(M,\vec{R})$ is a Seifert $(n+1)$\=/tuple or not. Otherwise, find the unique Seifert fibration for $M$; if it is compatible with the surfaces $R_1,\ldots,R_n$ then the answer is yes; otherwise, the answer is no.
\end{proof}

\begin{proposition}\label{thm:i-bundle recognition}
There is an algorithm which, given a $3$\=/manifold triple $(M,R,F)$ as input, decides whether it is an $I$\=/bundle triple or not.
\end{proposition}
\begin{proof}
If any of the following hold, then $(M,R,F)$ is not an $I$\=/bundle triple:
\begin{itemize}
\item at least one component of $F$ is not an annulus;
\item $R$ has zero or more than two components;
\item $R$ has two non\=/homeomorphic components;
\item there are one component of $R$ and one component of $F$ which are disjoint.
\end{itemize}
Otherwise, let $B$ be one component of $R$ if $R$ has two components, or the non\=/orientable surface doubly covered by $R$ if $R$ has only one component; if $(M,R,F)$ is an $I$\=/bundle triple, then it is the (product or twisted) $I$\=/bundle over $B$. Let $p\subs F$ be the union of the core curves of the annulus components of $F$. Let $X$ be the $I$\=/bundle over $B$, and denote by $q\subs\boundary X$ the union of the core curves of the annuli making up the vertical boundary of $X$. Then $(M,R,F)$ is an $I$\=/bundle triple if and only if the $3$\=/manifolds with boundary pattern $(M,p)$ and $(X,q)$ are homeomorphic.
\end{proof}

\begin{proposition}\label{thm:simple recognition}
There is an algorithm which, given a reasonable $3$\=/manifold triple $(M,R,F)$ as input, decides whether it is a simple triple or not.
\end{proposition}
\begin{proof}
Let us first deal with the cases where $M$ is homeomorphic to a solid torus or to $\torus[2]\times I$.
\begin{substeps}
\item In a solid torus, every incompressible annulus is boundary parallel.
Therefore, there are only finitely many incompressible annuli properly embedded in $(M,R)$ up to isotopy.
We can enumerate them and check whether any of them is not parallel to a component of $R$ or $F$.
\item When $M$ is homeomorphic to $\torus[2]\times I$, the triple $(M,R,F)$ is not simple if and only if one of the following happens:
\begin{itemize}
\item $\torus[2]\times\{i\}\subs R$ and $\torus[2]\times\{i\}\neq R$ for some $i\in\boundary I$;
\item there are isotopic curves $c_0,c_1\subs\torus[2]$ such that $c_i\times\{i\}\subs\boundary R$ for each $i\in\boundary I$;
\item there are two annulus components $R_1$ and $R_2$ of $R\cap(\torus[2]\times\{i\})$ for some $i\in\boundary I$, such that neither component of $(\torus[2]\times\{i\})\setminus\interior{R_1\cup R_2}$ is contained in $F$.
\end{itemize}
\end{substeps}

From now on, we will assume that $M$ is neither a solid torus nor homeomorphic to $\torus[2]\times I$. Moreover, thanks to \cref{thm:reasonable triple alternative}, we can assume that $\boundary M=R\cup F$. Let $p=R\cap F$; using the terminology of \cite{matveev}, we have that the $3$\=/manifold with boundary pattern $(M,p)$ is boundary irreducible.

\step{Detecting essential tori.} We can use \cite[Lemma 6.4.7]{matveev} to decide if $M$ contains an incompressible torus which is not parallel to a boundary component of $M$. If it does, then $(M,R,F)$ is not simple. Likewise, if $M$ has a torus boundary component which is not a component of $R$ or $F$, then the triple $(M,R,F)$ is not simple. If none of these two condition is satisfied, then $(M,R,F)$ satisfies the ``torus'' condition for being simple, and we only need to check for incompressible annuli.

\step{Detecting essential annuli.} We can use \cite[Lemma 6.4.8]{matveev} to decide if $M$ contains an incompressible annulus $A$ with the following properties:
\begin{itemize}
\item the boundary of $A$ lies in $\interior{R}$;
\item $A$ is not parallel to a union of annulus components of $R$ and $F$;
\item every non\=/trivial boundary compression disc for $A$ intersects $F$.
\end{itemize}
If there is such an annulus, then $(M,R,F)$ is not simple. If there is an annulus component of $R$ whose two boundary curves lie in different components of $F$, then $(M,R,F)$ is not simple. If none of these two conditions holds, then $(M,R,F)$ satisfies the ``annulus'' condition for being simple, and is therefore simple.
\end{proof}

\begin{theorem}\label{thm:computing jsj decomposition}
Let $(M,R)$ be an irreducible sufficiently large $3$\=/manifold pair. Suppose that no component of $R$ is a disc and that $\boundary M\setminus\interior R$ is a (possibly empty) collection of annuli. There is an algorithm which, given as input the pair $(M,R)$, returns the JSJ system $F$ of $(M,R)$. Moreover, each component of $(M,R)\cut F$ is reasonable.
\end{theorem}
\begin{proof}
It is easy to check that, under the assumptions made in the statement, the pieces of the JSJ decomposition of $(M,R)$ are reasonable. We can search through all surfaces $F'$ properly embedded in $(M,R)$ whose components are annuli and tori, until we find one satisfying \cref{def:jsj system:1,def:jsj system:2} of \cref{def:jsj-system} and, additionally, the property that each component of $(M,R)\cut F'$ is reasonable. All these conditions can be checked algorithmically, and we are guaranteed to find at least one such surface $F'$. Then, for every union of components of $F'$, we check whether it still satisfies the above properties and, among those that do, we return one that is minimal.
\end{proof}

Let us remark that the additional constraints we put on the pair $(M,R)$ -- namely the fact that $R$ has no disc components and $\boundary M\setminus\interior{R}$ is union of annuli -- are a consequence of the assumption of reasonableness we require for simple pieces.

\subsection{Piecing things together}\label{sec:homeomorphisms:piecing}

For the sake of convenience, we will slightly deviate from the statement of \cref{thm:jsj} and classify pieces of the JSJ decomposition as follows. If $(M,R)$ is an irreducible sufficiently large $3$\=/manifold pair and $F$ is its JSJ system, we say that a component $(N,R',F')$ of $(M,R)\cut F$ is:
\begin{itemize}
\item a \emph{Seifert piece} if $(N,R',F',\boundary N\setminus\interior{R'\cup F'})$ is a Seifert $4$\=/tuple;
\item an \emph{$I$\=/bundle piece} if $(N,R',F')$ is an $I$\=/bundle triple;
\item a \emph{strongly simple piece} if $(N,R',F')$ is a strongly simple triple.
\end{itemize}
It is clear that every component of $(M,R)\cut F$ is either a Seifert piece, an $I$\=/bundle piece, or a strongly simple piece. Note that the three possibilities are not mutually exclusive: it is possible for a component to be a Seifert piece and an $I$\=/bundle piece at the same time.

\begin{theorem}\label{thm:mapping class group of 3-manifold}
Let $(M,R)$ be an irreducible sufficiently large $3$\=/manifold pair. Suppose that no component of $R$ is a disc and that $\boundary M\setminus\interior{R}$ is a (possibly empty) collection of annuli. There is an algorithm which, given as input the pair $(M,R)$, returns
\begin{itemize}
\item a finite collection $\FFF$ of self\=/homeomorphisms of $(\boundary M,R)$, and
\item a finite collection $\CCC=\{(a_1,b_1),\ldots,(a_m,b_m)\}$, where $a_1,\ldots,a_m,b_1,\ldots, b_m$ are curves in $R$ with $a_1\cap b_1=\ldots=a_m\cap b_m=\emptyset$,
\end{itemize}
such that
\[
\trhomeo{M,R}=\bigcup_{f\in\FFF}\langle\twist{a_1}\twist{b_1}^{-1},\ldots,\twist{a_m}\twist{b_m}^{-1}\rangle f
\]
as a subgroup of $\homeo{\boundary M,R}$. Moreover, if every $I$\=/bundle piece in the JSJ decomposition of $(M,R)$ is a product $I$\=/bundle over $\surf{0}$, $\surf{0,1}$, $\surf{0,2}$, or $\surf{0,3}$, or the twisted $I$\=/bundle over $\nsurf{1}$, $\nsurf{1,1}$, or $\nsurf{1,2}$, then the curves in $\CCC$ are pairwise disjoint.
\end{theorem}
\begin{proof}
We use \cref{thm:computing jsj decomposition} to compute the JSJ system $F$ of $(M,R)$, and note that all the pieces are reasonable.

\step{Exceptional cases.} First of all, let us deal with the cases where the JSJ decomposition of $(M,R)$ is trivial or contains exceptional Seifert pieces.
\begin{substeps}
\item Suppose that $F$ is empty. Then $(M,R,\emptyset)$ is either a strongly simple piece, an $I$\=/bundle piece, or a Seifert piece. If it is strongly simple, then we can just return $\FFF=\trhomeo{M,R}$ and $\CCC=\emptyset$ thanks to \cref{thm:strongly simple homeomorphism}. If $(M,R,\boundary M\setminus\interior{R})$ is a Seifert triple, then we conclude immediately by \cref{thm:mapping class group of seifert spaces}. Finally, if $(M,R,\emptyset)$ is an $I$\=/bundle triple, then we apply \cref{thm:mapping class group of i-bundles}.
\item Suppose that a component $(N,R',F')$ is homeomorphic to $(\torus[2]\times I,\emptyset,\torus[2]\times\boundary I)$. The two boundary components of $N$ must come from the same torus in $F$, otherwise the surface $F$ would not be minimal. We then have that $M$ is homeomorphic to the mapping torus of some homeomorphism $\umap{\torus[2]}{\torus[2]}$, and we simply return $\FFF=\{\id\}$ and $\CCC=\emptyset$.
\end{substeps}

From now on, we can therefore assume that $F$ is non\=/empty and that there are no $(\torus[2]\times I,\emptyset,\torus[2]\times\boundary I)$ components in $(M,R)\cut F$. As a consequence of \cref{thm:unique seifert fibration}, it is easy to see that every Seifert piece either has a unique Seifert fibration, or it has two and it is homeomorphic to $(\nsurf{2}\twtimes I,\emptyset,\boundary(\nsurf{2}\twtimes I))$. Since the two fibrations of $\nsurf{2}\twtimes I$ are inequivalent, let us arbitrarily and consistently pick one for every $(\nsurf{2}\twtimes I,\emptyset,\boundary(\nsurf{2}\twtimes I))$ component of $(M,R)\cut F$ -- say the one over the M\"obius band. Then every Seifert piece in the JSJ decomposition of $(M,R)$ has a distinguished fibration, either the unique one -- for pieces which are not homeomorphic to the twisted $I$\=/bundle over a Klein bottle -- or the one we have selected. Moreover, every homeomorphism between Seifert pieces can be isotoped to be fibre\=/preserving.

\step{JSJ graph.} We closely follow the construction of the JSJ graph by \textcite[Section 6.3]{kuperberg-homeomorphism}, with slight changes to accommodate the presence of boundary. More precisely, we define a graph $\Gamma$ as follows.
\begin{itemize}
\item There is one vertex for each component of $(M,R)\cut F$.
\item There is one edge for each component of $F$. Each edge joins the two (not necessarily distinct) vertices corresponding to the JSJ pieces whose boundaries this component lies on.
\end{itemize}

An \emph{automorphism} of the JSJ graph $\Gamma$ is an automorphism $\map{\varphi}{\Gamma}{\Gamma}$ of the underlying graph, together with some additional data:
\begin{itemize}
\item for each vertex $(N,R',F')$ representing a strongly simple piece, a homeomorphism $\varphi|_N$ from $(N,R',F')$ to the component of $(M,R)\cut F$ associated to the vertex $\varphi(N,R',F')$, defined up to isotopy in $\varphi(N,R',F')$;
\item for each  edge $X$, a homeomorphism $\varphi|_X$ from $X$ to the component of $F$ associated to the edge $\varphi(X)$, defined up to isotopy in $\varphi(X)$; note that if $X$ is an annulus, there are only two possible values for $\varphi|_X$.
\end{itemize}
Moreover, we enforce the following consistency condition for each strongly simple piece $(N,R',F')$: for each edge $X$ having $(N,R',F')$ as an endpoint, we ask that the restriction of $\varphi|_N$ to $X$ is isotopic to $\varphi|_X$ in $\varphi(X)$. If both endpoints of $X$ are equal to $(N,R',F')$, then we ask that this condition holds for both copies of $X$ in $F'$.

\step{Finitely many automorphisms.} Every homeomorphism $\map{f}{(M,R)}{(M,R)}$ preserving $F$ induces an automorphism $\varphi_f$ of $\Gamma$ in a natural way. We claim that the set
\[
\Phi=\{\varphi_f:\text{$f$ is a self\=/homeomorphism of $(M,R)$ preserving $F$}\}
\]
is finite. In fact, the only issue that needs addressing is the potentially infinite number of homeomorphisms between JSJ tori connecting two Seifert pieces. Let $T$ be such a torus, connecting Seifert pieces $(N_1,R_1,S_1)$ and $(N_2,R_2,S_2)$. Let $c_1$ and $c_2$ be curves in $T$ which are fibres of $(N_1,R_1,S_1)$ and $(N_2,R_2,S_2)$ respectively; note that $c_1$ and $c_2$ are not isotopic in $T$, otherwise we could remove $T$ from the JSJ system $F$. Consider a homeomorphism $\map{f}{(M,R)}{(M,R)}$ preserving $F$. Necessarily, the curves $c_1$ and $c_2$ will be sent by $f$ to curves in $\varphi_f(T)$ which are isotopic to fibres $c_1'$ and $c_2'$ of $\varphi_f(N_1,R_1,S_1)$ and $\varphi_f(N_2,R_2,S_2)$ respectively. Since the curves $c_1$, $c_2$, $c_1'$, and $c_2'$ do not depend on $f$, and there are only finitely many isotopy classes of homeomorphisms $\umap{T}{\varphi_f(T)}$ sending $c_1$ to $c_1'$ and $c_2$ to $c_2'$, we conclude that there are only finitely many choices for $\varphi_f|_T$.

We remark that this argument is still valid if $(N_1,R_1,S_1)$ and $(N_2,R_2,S_2)$ are the same piece; in this case, we simply have to consider the possibility that $f$ sends $c_1$ to $c_2'$ and $c_2$ to $c_1'$.

\step{Computable automorphisms.} Our argument for the finiteness of $\Phi$ actually provides an algorithm to compute a finite superset $\Phi'$ of $\Phi$. As a consequence, in order to compute $\Phi$, we simply need an algorithm to decide whether an automorphism of $\Gamma$ is induced by a self\=/homeomorphism of $(M,R)$ preserving $F$.

Let $\varphi$ be an automorphism of $\Gamma$. If $(N,R',F')$ is not homeomorphic to $\varphi(N,R',F')$ for some piece $(N,R',F')$ then clearly $\varphi$ is not induced by a homeomorphism. Otherwise, for each piece $(N,R',F')$ which is not strongly simple, let us ask the following question: is there a homeomorphism $\map{f_N}{(N,R',F')}{\varphi(N,R',F')}$ such that, for each component $X$ of $F'$, the homeomorphisms $f_N|_X$ and $\varphi|_X$ are isotopic in $\varphi(X)$? If the answer -- which we can compute thanks to \cref{thm:seifert extension of homeomorphisms,thm:i-bundle extension of homeomorphisms} -- is no, then $\varphi$ cannot be induced by a homeomorphism. Otherwise, we easily see that $\varphi$ belongs to $\Phi$. In fact, we can isotope the homeomorphisms $f_N$ so that they agree on their shared boundary $F$, and then combine them to produce a homeomorphism $\map{f}{(M,R)}{(M,R)}$ preserving $F$; here, for ease of notation, we have used $f_N$ to refer to the homeomorphism $\varphi|_N$ for strongly simple pieces $(N,R',F')$.

We can therefore algorithmically construct a finite set $\FFF_0$ by picking, for each $\varphi\in\Phi$, a self\=/homeomorphism of $(M,R)$ preserving $F$ and inducing $\varphi$. This finite set has the property that every homeomorphism $\map{f}{(M,R)}{(M,R)}$ can be isotoped so that, for some $f_0\in\FFF_0$, the homeomorphism $ff_0$ is the identity on $F$ and on all the strongly simple pieces of the JSJ decomposition. In fact, it is enough to isotope $f$ so to preserve $F$ and pick $f_0\in\FFF_0$ such that $\varphi_{f_0}=\varphi_{f^{-1}}$, so that $ff_0$ induces the trivial automorphism of $\Gamma$. Up to further isotoping $f$, we can then assume that $ff_0$ is the identity on $F$ and on strongly simple pieces, as desired.

\step{Finitely many Dehn twists.} Finally, let us show how to compute the sets $\FFF$ and $\CCC$. For each Seifert or $I$\=/bundle piece $(N,R',F')$, we apply \cref{thm:mapping class group of seifert spaces,thm:mapping class group of i-bundles} to find
\begin{itemize}
\item a finite collection $\FFF_N$ of self\=/homeomorphisms of $(\boundary N,R')$ fixing $F'$ pointwise, and
\item a finite collection $\CCC_N=\{(a_{N,1},b_{N,1}),\ldots,(a_{N,m_N},b_{N,m_N})\}$ where $a_{N,1},\allowbreak\ldots,\allowbreak a_{N,m_N},\allowbreak b_{N,1},\allowbreak\ldots,\allowbreak b_{N,m_N}$ are curves in $\boundary N\setminus(F\cup\boundary R)$ with $a_{N,1}\cap b_{N,1}=\ldots=a_{N,m_N}\cap b_{N,m_N}=\emptyset$,
\end{itemize}
such that
\[
\trace{\fixhomeo{F'}{N,R'}}=\bigcup_{f\in\FFF_N}\langle\twist{a_{N,i}}\twist{b_{N,i}}^{-1}:1\le i\le m_N\rangle f
\]
as a subgroup of $\fixhomeo{F'}{\boundary N,R'}$. Moreover, if the piece $(N,R',F')$ is a product $I$\=/bundle over $\surf{0}$, $\surf{0,1}$, $\surf{0,2}$, or $\surf{0,3}$, or a twisted $I$\=/bundle over $\nsurf{1}$, $\nsurf{1,1}$, or $\nsurf{1,2}$, then the curves in $\CCC_N$ are pairwise disjoint. Let us extend the homeomorphisms in $\FFF_N$ to self\=/homeomorphisms of $(\boundary M,R)$ by setting them equal to the identity on $\boundary M\setminus N$. Finally, if some curve in $\CCC_N$ lies in $\boundary M\setminus R$, we can replace it with a parallel curve in $R$ without changing the isotopy class of the corresponding Dehn twist. It is then easy to see that the algorithm can return the collections
\begin{align*}
\FFF=\{\trace{f_0}:f_0\in\FFF_0\}\cup\bigcup_N \FFF_N&&\text{and}&&\CCC=\bigcup_N\CCC_N.&\qedhere
\end{align*}
\end{proof}

\begin{remark}
Translated in the language of \citeauthor{matveev}'s boundary patterns, \cref{thm:mapping class group of 3-manifold} can be used to compute the mapping class group of a Haken $3$\=/manifold $(M,p)$ with boundary pattern such that $p$ is a collection of simple closed curves and no component of $\boundary M\cut p$ is a disc.
\end{remark}

As a consequence of \cref{thm:mapping class group of 3-manifold} and the work of \textcite{gordon-luecke}, we have the following. Note that \cref{thm:homeomorphisms of knot complement:3} can also be proved using an oriented version of \citeauthor{matveev}'s boundary patterns.

\begin{corollary}\label{thm:homeomorphisms of knot complement}
Let $M\subs\sphere[3]$ be the complement of a non\=/trivial knot.
\begin{enumarabic}
\item\label[statement]{thm:homeomorphisms of knot complement:1} Every self\=/homeomorphism of $M$ sends the meridian curve of $\boundary M$ to itself.
\item\label[statement]{thm:homeomorphisms of knot complement:2} The group $\trhomeo{M}$ is either trivial or generated by the homeomorphism
\[
\map{(-\id)}{\boundary M}{\boundary M}
\]
inducing multiplication by $-1$ on $H_1(\boundary M)$.
\item\label[statement]{thm:homeomorphisms of knot complement:3} There is an algorithm which, given $M$ as input, returns representatives of $\trhomeo{M}$.
\end{enumarabic}
\end{corollary}
\begin{proof}
\Cref{thm:homeomorphisms of knot complement:1} follows from the work of \textcite{gordon-luecke}; \cref{thm:homeomorphisms of knot complement:2} is an immediate consequence. As far as \cref{thm:homeomorphisms of knot complement:3} is concerned, we can apply \cref{thm:mapping class group of 3-manifold} to the pair $(M,\boundary M)$. Since homeomorphisms of the form $\twist{a}\twist{b}^{-1}$ cannot act as $(-\id)$ on $H_1(\boundary M)$, it is enough to check whether the collection $\FFF$ returned by the algorithm contains any non\=/trivial homeomorphism.
\end{proof}

\section{An algorithmic problem on free groups} \label{sec:free groups}

\subsection{Introducing band systems}\label{sec:free groups:band systems}
We now take a break from topology and enter the realm of combinatorics and free groups. The bulk of \cref{sec:free groups} will be devoted to proving \cref{thm:exponential equation free group}. The reader should however not be surprised by how seemingly distant the statement of the theorem is to anything concerning isotopies of surfaces. In fact, an easy consequence of the theorem -- namely \cref{thm:product of dehn twists extension to handlebody} --  will allow us to answer an algorithmic question about extensions of homeomorphisms to the interior of handlebodies.

We start by introducing the notation we will need to talk about cancellation of words in free groups, namely (marked) band systems and bundling maps. \Cref{sec:free groups:the theorem} is entirely dedicated to the proof of \cref{thm:exponential equation free group}, while the topological corollary is presented in \cref{sec:free groups:topology}.

\begin{definition}
Let $n$ be a positive integer. A \emph{band system} of length $2n$ is a collection
\[
B=\{(a_1,b_1),\ldots,(a_n,b_n)\}
\]
such that:
\begin{enumroman}
\item $\{a_1,b_1,\ldots,a_n,b_n\}=\{1,\ldots,2n\}$;
\item for each $1\le i\le n$, the inequality $a_i<b_i$ holds;
\item\label[condition]{def:band system:3} there is no pair of indices $i,j$ such that $a_i<a_j<b_i<b_j$.
\end{enumroman}
The individual pairs $(a_i,b_i)$ are called \emph{bands}, with $a_i$ being the \emph{left endpoint} and $b_i$ being the \emph{right endpoint}.
\end{definition}

\Cref{fig:band system example} shows a graphical depiction of the band system
\[
B=\text{$\{(1,8),(2,5),(3,4),(6,7)\}$ of length $2n=8$},
\]
hopefully justifying the name. If we think of a pair $(a,b)\in B$ as a physical band connecting the elements $a$ and $b$, then \cref{def:band system:3} ensures that these bands do not intersect.

\begin{definition}
A \emph{marked band system} is a pair $(B,p)$ where $B$ is a band system of length -- say -- $2n$, and $p=(p_1,\ldots,p_r)$ is a weakly increasing sequence of numbers with
\begin{align*}
p_i\in\left\{\frac{1}{2},1+\frac{1}{2},\ldots,2n+\frac{1}{2}\right\}&&\text{for $1\le i\le r$.}
\end{align*}
\end{definition}

\Cref{fig:marked band system example} depicts the same band system as \cref{fig:band system example}, equipped with the marking
\[
p=\left(\frac{1}{2},3+\frac{1}{2},7+\frac{1}{2}\right).
\]
We think of an element $p_i=a+1/2$ as a ``separator'' between the integers $a$ and $a+1$.

\begin{figure}\centering
\tikzsetnextfilename{marked-band-system-example}
\def\mybandsystempic{
    \pic{sequence of points={points={n=8}}};
    \pic[default bs band]{bs band={pairs={1-8,2-5,3-4,6-7}}};
    \foreach \i in {1,...,\alen} {
    \pic[bs single box]{bs box={i=\i,label=$\i$}};
    }
}
\subcaptionbox{\label{fig:band system example}}[.45\linewidth]{
\begin{tikzpicture}[bs default name=a]
\mybandsystempic
\end{tikzpicture}
}
\subcaptionbox{\label{fig:marked band system example}}[.45\linewidth]{
\begin{tikzpicture}[bs default name=a]
\mybandsystempic
\pic[default bs mark]{bs mark={after={0,3,7}}};
\end{tikzpicture}
}
\caption{A band system \subref{fig:band system example} and a marked band system \subref{fig:marked band system example}.}
\end{figure}

Let $(B,p)$ be a marked band system. The \emph{maximal bundle} of $(B,p)$ is, loosely speaking, the marked band system $(B',p')$ constructed from $(B,p)$ by bundling together parallel bands as much as possible, while making sure that the bundles don't cross any marked spot. Formally, let $2n$ be the length of $B$. Consider the equivalence relation $\sim$ on $\{1,\ldots,2n\}$ generated by
\[
\text{$a_1\sim a_2$ and $b_1\sim b_2$ whenever }\begin{dcases}(a_1,b_1),(a_2,b_2)\in B,\\\text{$a_2=a_1+1$, $b_2=b_1-1$, and}\\a_1+\frac{1}{2},b_1-\frac{1}{2}\not\in p.\end{dcases}
\]
Let $E=\quot{\{1,\ldots,2n\}}{\sim}$ be the quotient set, and let $2n'$ be its cardinality (which is easily seen to be even). Since equivalence classes consist of consecutive integers, there is a natural bijection between $E$ and $\{1,\ldots,2n'\}$ such that classes containing smaller integers are mapped to smaller integers. Let
\[
\map{\iota}{\{1,\ldots,2n\}}{\{1,\ldots,2n'\}}
\]
be the composition of the projection map $\umap{\{1,\ldots,2n\}}{E}$ with said bijection; we will call $\iota$ the \emph{bundling map}. We can then define
\[
B'=\{(\iota(a),\iota(b)):(a,b)\in B\}.
\]
As far as the marking is concerned, if $p=(p_1,\ldots,p_r)$, for each $1\le i\le r$ we define
\[
p_i'=\begin{dcases*}
\frac{1}{2}&if $p_i=\dfrac{1}{2}$,\\
\iota\left(p_i-\frac{1}{2}\right)+\frac{1}{2}&otherwise,
\end{dcases*}
\]
and set $p'=(p_1',\ldots,p_r')$. It is not hard to verify that $(B',p')$ is a marked band system.

As an example, let us look at \cref{fig:maximal bundle example}. The marked band system on the top is defined by
\begin{align*}
B&=\text{$\{(1,12),(2,11),(3,8),(4,7),(5,6),(9,10)\}$ of length $2n=12$},\\
p&=\left(2+\frac{1}{2},6+\frac{1}{2},9+\frac{1}{2},12+\frac{1}{2}\right).
\end{align*}
In order to construct the maximal bundle, we group together the bands $(1,12)$ and $(3,4)$, since they are parallel and do not cross any marked spot. We do the same for bands $(3,8)$ and $(4,7)$; note that the band $(5,6)$ does not belong to the same bundle, since the spot $6+1/2$ is marked. The maximal bundle of $(B,p)$ is then $(B',p')$, where
\begin{align*}
B'&=\text{$\{(1,8),(2,5),(3,4),(6,7)\}$ of length $2n'=8$},\\
p'&=\left(1+\frac{1}{2},4+\frac{1}{2}, 6+\frac{1}{2},8+\frac{1}{2}\right).
\end{align*}

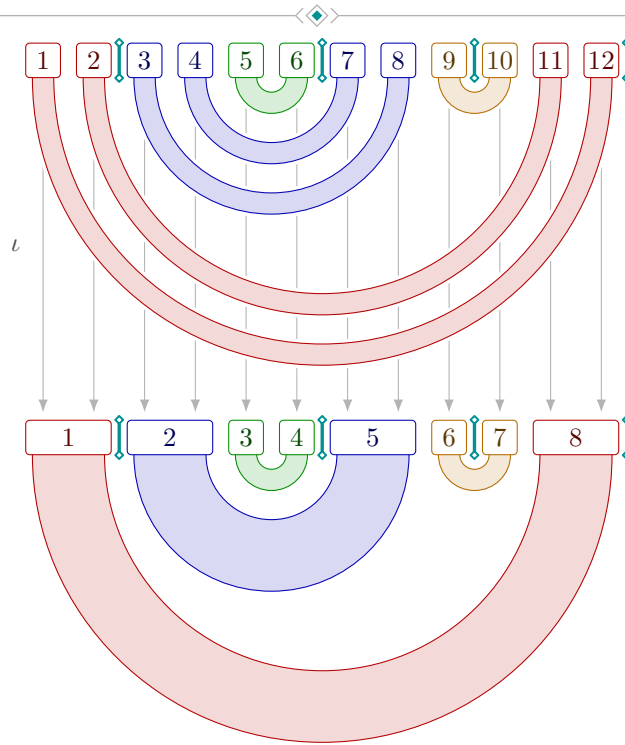
\begin{figure}\centering
\tikzsetnextfilename{maximal-bundle-example}
\begin{tikzpicture}
\def\mymarks{2,6,9,12}
\pic{sequence of points={name=a,points={n=12}}};
\pic[yshift=-5cm]{sequence of points={name=b,points={n=12}}};
\pic{at every point={name=a}{\draw[black!30,->] ($(l)!.5!(r)$) -- +([yshift=16pt]0,-5cm) node[black!70,midway,left=5pt]{\ifnum\i=1$\iota$\fi};}};
\foreach \lb/\le/\rb/\re/\li/\ri[count=\c] in {1/2/11/12/1/8,3/4/7/8/2/5,5/5/6/6/3/4,9/9/10/10/6/7} {
\colorlet{mycol}{main color \c!70!black}
\pic[colored bs band=mycol]{bs band={name=b,wide={\lb-\le}{\rb-\re}}};
\foreach \l[evaluate=\l as \r using int(\re-\l+\lb)] in {\lb,...,\le} {
	\pic[colored bs band=mycol]{bs band={name=a,pairs={\l-\r}}};
}
\foreach \ep in {l,r} {
	\edef\l{\csname\ep b\endcsname}\edef\r{\csname\ep e\endcsname}\edef\i{\csname\ep i\endcsname}
	\pic[bs colored single box=mycol]{bs box={name=b,l=\l,r=\r,label=$\i$}};
	\foreach \j in {\l,...,\r} {
		\pic[bs colored single box=mycol]{bs box={name=a,i=\j,label=$\j$}};
	}
}
}
\foreach \name in {a,b}\pic[default bs mark]{bs mark={name=\name,after={\mymarks}}};
\end{tikzpicture}
\caption{A marked band system (on the top) and its maximal bundle (on the bottom). Bands of the same colour get bundled together; vertical arrows represent the bundling map $\iota$.\label{fig:maximal bundle example}}
\end{figure}

Let us now remark a few properties of the bundling map $\iota$.
\begin{itemize}
\item The bundling map sends bands of $B$ to bands of $B'$. From now on, with slight abuse of notation, for each band $(a,b)\in B$ we will denote by $\iota(a,b)$ the band $(\iota(a),\iota(b))\in B'$.
\item If $\iota(a_1,b_1)=\iota(a_2,b_2)$ where $(a_1,b_1)$ and $(a_2,b_2)$ are bands of $B$ with $a_1<a_2<b_2<b_1$, then $(a_1+i,b_1-i)\in B$ and $\iota(a_1+i,b_1-i)=\iota(a_1,b_1)$ for each $0\le i\le a_2-a_1$. Moreover, no element of $p$ lies between $a_1$ and $a_2$ or between $b_2$ and $b_1$.
\end{itemize}

The following facts are completely elementary, but we prefer to state them here so as not to clutter the proof of \cref{thm:exponential equation free group}.
\begin{lemma}\label{thm:band of width 1}
Let $(B,p)$ be a marked band system.
\begin{enumarabic}
\item\label[statement]{thm:band of width 1:1} If $(a,b)$ is a band of $B$, then there is a band $(c,c+1)\in B$ for some $a\le c<b$.
\item\label[statement]{thm:band of width 1:2} Let $(B',p')$ be the maximal bundle of $(B,p)$, with bundling map $\iota$.
If $(a_1,b_1)$ and $(a_2,b_2)$ are bands of $B$ with $a_1<a_2<b_2<b_1$, then either
\begin{itemize}
\item $\iota(a_1,b_1)=\iota(a_2,b_2)$,
\item there is an element of $p$ lying between $a_1$ and $a_2$ or between $b_2$ and $b_1$, or
\item there is a band $(c,c+1)\in B$ for some
\[
c\in\{a_1,\ldots,a_2-1\}\cup\{b_2,\ldots,b_1-1\}.
\]
\end{itemize}
\end{enumarabic}
\end{lemma}
\begin{proof}
In order to prove \cref{thm:band of width 1:1}, simply consider the narrowest band amongst those having left endpoint in $\{a,\ldots,b-1\}$. We now turn to \cref{thm:band of width 1:2}. If there is a band $(a,b)\in B$ with $a_1<a<b<a_2$ or $b_2<a<b<b_1$ then we can apply \cref{thm:band of width 1:1} and immediately conclude. Otherwise, for every band $(a,b)\in B$ we have that $a_1\le a\le a_2$ if and only if $b_2\le b\le b_1$. But then, using the notations from the definition of maximal bundle, it is clear that either
\[
(a_1,b_1)\sim (a_1+1,b_1-1)\sim\ldots\sim (a_2,b_2),
\]
or there is marked spot (in the marking $p$) lying between $a_1$ and $a_2$ or between $b_2$ and $b_1$.
\end{proof}

Let $(B',p')$ be a marked band system of length $2n'$. We say that $(B',p')$ is \emph{maximal} if it is its own maximal bundle. Of course, if $(B',p')$ is maximal, there are infinitely many marked band systems $(B,p)$ having $(B',p')$ as their maximal bundle; in fact, these marked band systems can be parametrised quite nicely. An \emph{unbundling map} for $(B',p')$ is a function $\map{\varphi}{\{1,\ldots,2n'\}}{\ZZ_{>0}}$ such that $\varphi(a')=\varphi(b')$ for each band $(a',b')\in B'$. There is a one-to-one correspondence between unbundling maps for $(B',p')$ and marked band systems having $(B',p')$ as their maximal bundle.
\begin{substeps}
\item Given a marked band system $(B,p)$ whose maximal bundle is $(B',p')$, we can define an unbundling map $\varphi$ as $\varphi(a')=\card{\iota^{-1}(a')}$, where $\iota$ is the bundling map for $(B,p)$.
\item Conversely, let $\map{\varphi}{\{1,\ldots,2n'\}}{\ZZ_{>0}}$ be an unbundling map for $(B',p')$. Let $2n=\varphi(1)+\ldots+\varphi(2n')$; for each $1\le a\le 2n$, define $\iota(a)$ to be the unique integer in $\{1,\ldots,2n'\}$ such that
\[
\varphi(1)+\ldots+\varphi(\iota(a)-1)<a\le\varphi(1)+\ldots+\varphi(\iota(a)).
\]
It is now easy to construct the unique marked band system $(B,p)$ having $(B',p')$ as maximal bundle and $\iota$ as bundling map. In fact, we can set
\[
B=\bigcup_{(a',b')\in B'}\left\{(a+i,b-i):
\begin{matrix*}[l]
\{a+1,\ldots,a+\varphi(a')\}=\iota^{-1}(a'),\\
\{b-\varphi(a'),\ldots,b-1\}=\iota^{-1}(b'),\\
1\le i\le\varphi(a')
\end{matrix*}
\right\}
\]
and $p=(p_1,\ldots,p_r)$, where $p'=(p_1',\ldots,p_r')$ and
\begin{align*}
&p_i=\begin{dcases*}
\frac{1}{2}&if $\displaystyle p_i'=\frac{1}{2}$,\\
\max\left(\iota^{-1}\left(p_i'-\frac{1}{2}\right)\right)+\frac{1}{2}&otherwise
\end{dcases*}
&\text{for $1\le i\le r$.}
\end{align*}
\end{substeps}

\Cref{fig:unbundling map example} shows an instance of this construction. The marked band system of length $2n'=4$ on the top is defined by $B'=\{(1,4),(2,3)\}$ and $p'=(1+1/2)$, and is clearly maximal. We choose the unbundling map $\varphi$ such that $\varphi(1)=\varphi(4)=4$ and $\varphi(2)=\varphi(3)=2$. If we replace each band $(a,b)\in B'$ with $\varphi(a)$ parallel bands, we obtain the marked band system $(B,p)$ displayed on the bottom, where
\begin{align*}
B&=\text{$\{(1,12),(2,11),(3,10),(4,9),(5,8),(6,7)\}$ of length $2n=12$},\\
p&=\left(4+\frac{1}{2}\right).
\end{align*}

\begin{figure}\centering
\begin{tikzpicture}[remember picture]
\def\mymarks{4}
\pic{sequence of points={name=b,points={n=12}}};
\pic[yshift=-6cm]{sequence of points={name=a,points={n=12}}};
\foreach \lb/\le/\rb/\re/\li/\ri[count=\c] in {1/4/9/12/1/4,5/6/7/8/2/3} {
\colorlet{mycol}{main color \c!70!black}
\pic[colored bs band=mycol]{bs band={name=b,wide={\lb-\le}{\rb-\re}}};
\foreach \l[evaluate=\l as \r using int(\re-\l+\lb)] in {\lb,...,\le} {
	\pic[colored bs band=mycol]{bs band={name=a,pairs={\l-\r}}};
}
\foreach \ep in {l,r} {
	\edef\l{\csname\ep b\endcsname}\edef\r{\csname\ep e\endcsname}\edef\i{\csname\ep i\endcsname}
	\pic[bs colored single box=mycol]{bs box={name=b,l=\l,r=\r,label=$\i$}};
	\draw[mycol,thick,decorate,decoration={brace}] ([yshift=19pt]a-\l-l) -- ([yshift=19pt]a-\r-r) node[midway,anchor=base,yshift=9pt,text=mycol!50!black,evaluate={\j=int(\r-\l+1);}] (phi-\c-\ep) {\ifnum\j>2$\varphi(\subnode{phi arg-\c-\ep}{\i})=\j$\fi};
	\coordinate (arrow start-\c-\ep) at ($(b-\l-l)!.5!(b-\r-r)$);
	\foreach \j in {\l,...,\r} {
		\pic[bs colored single box=mycol]{bs box={name=a,i=\j,label=$\j$}};
	}
}
}
\node[anchor=base,text=main color 2!50!black] at ($(phi-2-l)!.5!(phi-2-r)$) {$\varphi(\subnode{phi arg-2-l}{2})=\varphi(\subnode{phi arg-2-r}{3})=2$};
\foreach \u in {1-l,2-l,2-r,1-r} {
\scoped[on background layer]\draw[rounded corners=1pt,-{[quick]>},shorten >=2pt,black!30] (arrow start-\u) -- +(0,-3.8cm) to[out=-90,in=90,looseness=1.5] (phi arg-\u);
}
\foreach \name in {a,b}\pic[default bs mark]{bs mark={name=\name,after={\mymarks}}};
\end{tikzpicture}
\caption{A maximal marked band system (on the top) and the marked band system associated to an unbundling map $\varphi$ (on the bottom).\label{fig:unbundling map example}}
\end{figure}
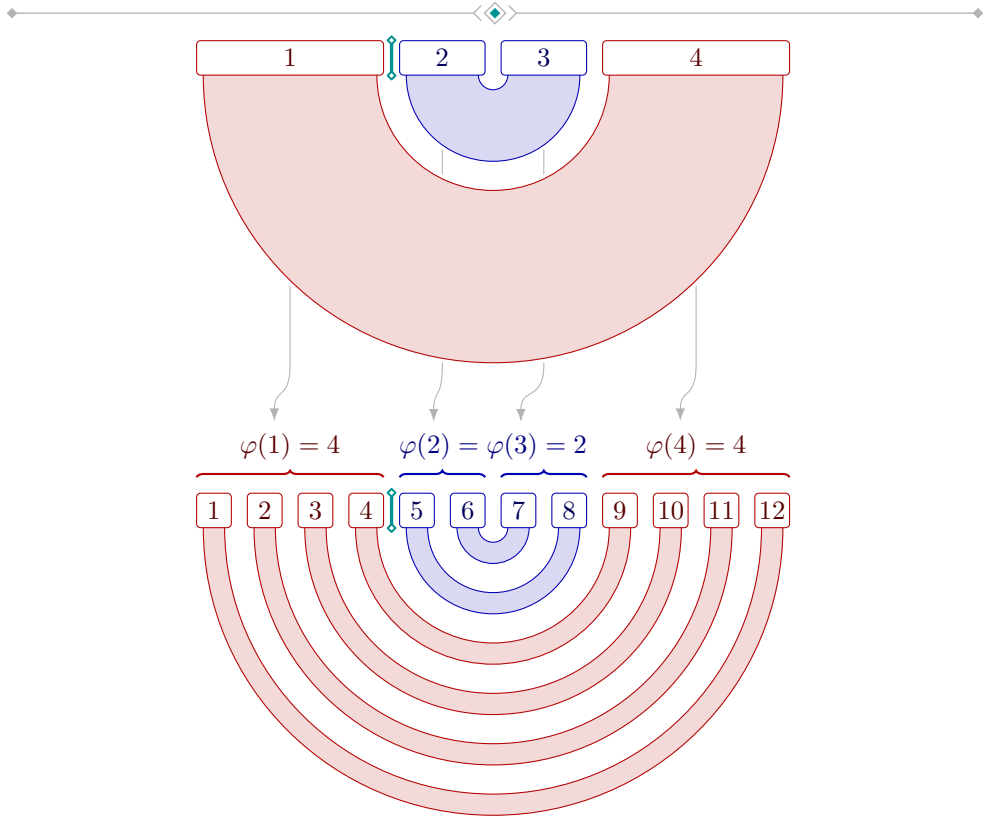

Given an unbundling map $\varphi$ for $(B',p')$, it is sometimes convenient to  define the corresponding \emph{cumulative unbundling map} $\map{\what{\varphi}}{\{1,\ldots,2n'+1\}}{\ZZ_{\ge 0}}$ as
\[
\what{\varphi}(a')=\varphi(1)+\ldots+\varphi(a'-1).
\]
With this definition, we can compactly write $\iota^{-1}(a')$ as $\{\what{\varphi}(a')+1,\ldots,\what{\varphi}(a'+1)\}$, without having to refer to the band system $(B,p)$ associated to $\varphi$.

\subsection{Band systems and free groups}\label{sec:free groups:the theorem}

Let us introduce some notation about free groups which we will employ in this section. Let $g$ be a positive integer. Fix $g$ symbols $\{x_1,\ldots,x_g\}$ and denote by $\free{g}$ the free group over them. If $s$ is a word of length $m$ in the symbols $x_1,\ldots,x_g$ and their inverses, we use the following notations:
\begin{itemize}
\item $s_i$ is the $i$\=/th symbol of $s$, where $1\le i\le m$;
\item $\wordrange{s}{i}{j}$ is the subword $s_is_{i+1}\ldots s_{j}$, where $1\le i\le j\le m$;
\item $s^{-1}$ is the word such that $(s^{-1})_i=(s_{m+1-i})^{-1}$ for $1\le i\le m$;
\item $s^k$ is the concatenation of $k$ copies of $s$, where $k\ge 0$; if $k<0$, we define $s^k=(s^{-1})^{-k}$;
\item $s^{\infty}$ the infinite word obtained by concatenating infinitely many copies of $s$; this does not correspond to an element of $\free{g}$, and we will only ever make use of finite subwords of infinite words.
\end{itemize}

Let $s=s_1s_2\cdots s_{2n}$ be a word which reduces to $1$ in $\free{g}$. Then it must do so via a sequence of cancellations of adjacent symbols of the form $xx^{-1}$ or $x^{-1}x$ (by ``adjacent'' we mean ``adjacent after performing the previous cancellations in the sequence''). For any such sequence, we can define a band system
\[
B=\{(a,b):\text{$1\le a<b\le 2n$, $s_a$ cancels with $s_b$}\}.
\]
We say that $B$ is a \emph{cancellation band system} for $s$; cancellation band systems need not be unique, but each word reducing to $1$ has at least one of them. For example, \cref{fig:cancellation band system example} shows how the band system
\[
B=\{(1,6),(2,3),(4,5),(7,10),(8,9)\}
\]
is a cancellation band system for the word $yy^{-1}yx^{-1}xy^{-1}xyy^{-1}x^{-1}$ in the free group generated by the symbols $x$ and $y$.

\begin{figure}\centering
\tikzsetnextfilename{cancellation-band-system-example}
\begin{tikzpicture}
\pic{sequence of points={points={width=20pt,sep=0pt,n=10}}};
\foreach \l/\r in {1/6,2/3,4/5,7/10,8/9} {
\pic[default bs band] {bs band={pairs={\l-\r}}};
}
\foreach \s[count=\i] in {2,-2,2,-1,1,-2,1,2,-2,-1} {
\pic[bs single box,evaluate={\s={"y^{-1}","x^{-1}","x","y"}[\s+2-(\s>0)];}]{bs box={height=20pt,i=\i,label=$\s$}};
}
\end{tikzpicture}
\caption{A cancellation band system for the word $yy^{-1}yx^{-1}xy^{-1}xyy^{-1}x^{-1}$.\label{fig:cancellation band system example}}
\end{figure}
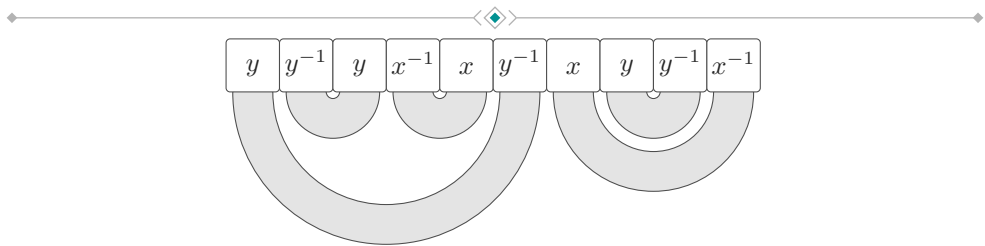

The main technical result of this section is the following.

\begin{theorem}\label{thm:exponential equation free group}
Let $w_0,\ldots,w_l,t_1,\ldots,t_l\in\free{g}$. Consider the set
\[
\AAA=\left\{(k_1,\ldots,k_l)\in\ZZ^l:w_0t_1^{k_1}w_1t_2^{k_2}w_2\cdots t_l^{k_l}w_l=1\right\}.
\]
There is an algorithm which, given as input $w_0,\ldots,w_l$ and $t_1,\ldots,t_l$, returns a collection of sets $A_1,\ldots,A_N$ such that $\AAA=A_1\cup\ldots\cup A_N$, and each $A_j$ is of the form
\[
A_j=\left\{\vec{z}_j+\mat{M}_j\vec{v}:\vec{v}\in\ZZ_{\ge0}^{d_j}\right\}
\]
for some vector $\vec{z}_j\in\ZZ^l$ and some matrix $\mat{M}_j\in\ZZ^{l\times d_j}$.
\end{theorem}
\begin{proof}
Let $w_0,\ldots,w_l$ and $t_1,\ldots,t_l$ be given as reduced words over $\{x_1,\ldots,x_g\}$. For a vector $\vec{k}\in\ZZ^l$, denote by $w(\vec{k})$ the word $w_0t_1^{k_1}w_1\cdots t_l^{k_l}w_l$.

\step{Cyclically reduced words.} For each $1\le i\le l$ we can write $t_i=u_it'_iu_i^{-1}$, where $t'_i$ is a cyclically reduced word and $u_i$ is a  reduced word. Define:
\begin{itemize}
\item $w_0'=w_0u_1$;
\item $w_i'=u_i^{-1}w_iu_{i+1}$ for $1\le i<l$;
\item $w_l'=u_l^{-1}w_l$.
\end{itemize}
Then
\[
\AAA=\left\{(k_1,\ldots,k_l)\in\ZZ^l:w_0'(t_1')^{k_1}w_1'\cdots (t_l')^{k_l}w_l'=1\right\}.
\]
Therefore, up to replacing each $w_i$ with $w_i'$ and each $t_i$ with $t_i'$, we can assume that each $t_i$ is cyclically reduced.

\step{Only non-negative powers.} For each $\vecsym{\epsilon}=(\epsilon_1,\ldots,\epsilon_l)\in\{-1,1\}^l$ define
\[
\AAA^{\vecsym{\epsilon}}=\left\{(\epsilon_1k_1,\ldots,\epsilon_lk_l):\,
\begin{matrix*}[l]
(k_1,\ldots,k_l)\in\ZZ_{\ge0}^l,\\
w_0t_1^{\epsilon_1k_1}w_1\cdots t_l^{\epsilon_lk_l}w_l=1
\end{matrix*}
\right\},
\]
so that
\[
\AAA=\bigcup_{\vecsym{\epsilon}\in\{-1,1\}^l}\AAA^{\vecsym{\epsilon}}.
\]
Denote by $\AAA^+$ the set
\[
\AAA^{(1,\ldots,1)}=\left\{(k_1,\ldots,k_l)\in\ZZ_{\ge 0}:w_0t_1^{k_1}w_1\cdots t_l^{k_l}w_l=1\right\}.
\]
Suppose we have an algorithm to decompose $\AAA^+$ as described in the statement, and fix $\vecsym{\epsilon}=(\epsilon_1,\ldots,\epsilon_l)\in\{-1,1\}^l$. Then we can algorithmically decompose
\[
\left\{(k_1,\ldots,k_l)\in\ZZ_{\ge 0}^l:w_0(t_1^{\epsilon_1})^{k_1}w_1\cdots (t_l^{\epsilon_l})^{k_l}w_l=1\right\}=A_1\cup\ldots\cup A_N,
\]
where each $A_j$ is of the form
\[
A_j=\left\{\vec{z}_j+\mat{M}_j\vec{v}:\vec{v}\in\ZZ_{\ge0}^{d_j}\right\}
\]
for some vector $\vec{z}_j\in\ZZ^l$ and some matrix $\mat{M}_j\in\ZZ^{l\times d_j}$. But then we have the decomposition
\[
\AAA^{\vecsym{\epsilon}}=A_1'\cup\ldots\cup A_N',
\]
where, for each $1\le j\le N$:
\begin{itemize}
\item $A_j=\left\{\vec{z}_j'+\mat{M}'_j\vec{v}:\vec{v}\in\ZZ_{\ge0}^{d_j}\right\}$;
\item $\vec{z}_j'$ is obtained from $\vec{z}_j$ by multiplying the $i$-th coordinate by $\epsilon_i$ for $1\le i\le l$;
\item $\mat{M}_j'$ is obtained from $\mat{M}_j$ by multiplying the $i$-th row by $\epsilon_i$ for $1\le i\le l$.
\end{itemize}
Since there are only finitely many choices of $\vecsym{\epsilon}\in\{-1,1\}^l$, it is enough to describe an algorithm to compute the set $\AAA^+$.

\bgroup
\let\oldpmatrix\pmatrix
\let\oldendpmatrix\endpmatrix
\def\pmatrix{\bgroup\def\arraystretch{1.5}\oldpmatrix}
\def\endpmatrix{\oldendpmatrix\egroup}
\step{Non-empty words.} Suppose we have an algorithm which works when every $t_i$ is a non-empty word (that is, $t_i\neq 1$ as an element of $\free{g}$). If $t_i$ is the empty word for some $1\le i\le l$, by induction on $l$ we can assume that the set
\bgroup
\def\firstline{(k_1,\ldots,k_{i-1},k_{i+1},\ldots,k_l)\in\ZZ_{\ge0}^{l-1}:}
\def\secondline{w_0t_1^{k_1}w_1\cdots t_{i-1}^{k_{i-1}}(w_{i-1}w_i)t_{i+1}^{k_{i+1}}w_{i+1}\cdots t_l^{k_l}w_l=1}
\begin{multline*}
\left\{\firstline\vphantom{\secondline}\right.\\
\left.\vphantom{\firstline}\secondline\right\}
\end{multline*}
\egroup
can be algorithmically decomposed as $A_1\cup\ldots\cup A_N$, where each $A_j$ is of the form
\[
A_j=\left\{\vec{z}_j+\mat{M}_j\vec{v}:\vec{v}\in\ZZ_{\ge0}^{d_j}\right\}
\]
for some vector $\vec{z}_j\in\ZZ^{l-1}$ and some matrix $\mat{M}_j\in\ZZ^{(l-1)\times d_j}$. For each $1\le j\le N$, write
\begin{align*}
\vec{z}_j=\begin{pmatrix}\vec{z}_j^{(1)}\\\vec{z}_j^{(2)}\end{pmatrix}&&\text{and}&&\mat{M}_j=\begin{pmatrix}\mat{M}_j^{(1)}\\\mat{M}_j^{(2)}\end{pmatrix},
\end{align*}
where $\vec{z}_j^{(1)}\in\ZZ^{i-1}$, $\vec{z}_j^{(2)}\in\ZZ^{l-i}$, and similarly $\mat{M}_j^{(1)}\in\ZZ^{(i-1)\times d_j}$, $\mat{M}_j^{(2)}\in\ZZ^{(l-i)\times d_j}$. Define
\begin{align*}
\vec{z}_j'=\begin{pmatrix}\vec{z}_j^{(1)}\\0\\\vec{z}_j^{(2)}\end{pmatrix}\in\ZZ^l&&\text{and}&&\mat{M}_j'=\begin{pmatrix}\mat{M}_j^{(1)}&\vec{0}\\\mat{0}&1\\\mat{M}_j^{(2)}&\vec{0}\end{pmatrix}\in\ZZ^{l\times(d_j+1)}.
\end{align*}
We get the required decomposition
\[
\AAA^+=\bigcup_{j=1}^N\left\{\vec{z}_j'+\mat{M}_j'\vec{v}:\vec{v}\in\ZZ_{\ge0}^{d_j+1}\right\}.
\]
Therefore, we can assume that each $t_i$ is non-empty.
\egroup

\step{Fixed cancellation bundle.} Let $\vec{k}=(k_1,\ldots,k_l)$ be a vector in $\AAA^+$. Consider a cancellation band system $B$ for the word $w(\vec{k})$. Moreover, define the \emph{block marking} $p=(p_1,\ldots,p_{2l+2})$ to separate different ``blocks'' of $w(\vec{k})$, by setting
\begin{align*}
&p_{2i-1}=\card{w_0}+k_1\card{t_1}+\card{w_1}+\ldots+k_{i-1}\card{t_{i-1}}+\frac{1}{2}&&\text{for $1\le i\le l+1$},\\
&p_{2i}=\card{w_0}+k_1\card{t_1}+\card{w_1}+\ldots+\card{w_{i-1}}+\frac{1}{2}&&\text{for $1\le i\le l+1$}.
\end{align*}
Denote by $(B',p')$ the maximal bundle of the marked band system $(B,p)$; we say that $(B',p')$ is a \emph{cancellation bundle} for $\vec{k}$ if it can be obtained in this way (that is, as the maximal bundle of a cancellation band system for $w(\vec{k})$ endowed with the block marking described above). Let
\[
\map{\iota}{\{1,\ldots,2n\}}{\{1,\ldots,2n'\}}
\]
be the bundling map, where $2n$ and $2n'$ are the lengths of $B$ and $B'$ respectively. Let us label the integers $\{1,\ldots,2n'\}$ with the symbols $W_0,\ldots,W_l,T_1,\ldots,T_l$ depending on which ``block'' of $w(\vec{k})$ they come from. More precisely, we label
\begin{itemize}
\item the integers between $p'_{2i+1}$ and $p'_{2i+2}$ with the symbol $W_i$ for $0\le i\le l$, and
\item the integers between $p'_{2i}$ and $p'_{2i+1}$ with the symbol $T_i$ for $1\le i\le l$.
\end{itemize}
We will use the notation $i\prec X$ to signify that the integer $i\in\{1,\ldots,2n'\}$ has label $X\in\{W_0,\ldots,W_l,T_1,\ldots,T_l\}$. See \cref{fig:cancellation bundle example} for a graphical representation of a cancellation bundle.

\begin{figure}
\centering
\tikzsetnextfilename{cancellation-bundle-example}
\begin{tikzpicture}
\pic{sequence of points={name=a,block={n=2},block={n=2},block={n=2},block={n=5},block={n=4},block={n=1}}};
\pic[yshift=-6cm]{sequence of points={name=b,block={n=2},block={n=2},block={n=2},block={n=5},block={n=4},block={n=1}}};
\pic{at every point={name=a}{\draw[black!30,->] ($(l)!.5!(r)$) -- +([yshift=16pt]0,-6cm);}};

\pic[default bs band]{bs band={name=a,pairs={1-16,2-15,3-14,4-9,5-8,6-7,10-13,11-12}}};
\foreach \l/\r/\label/\c in {1/2/w_0/1,3/4/t_1/2,5/6/t_1/2,7/11/w_1/1,12/15/t_2/2,16/16/w_2/1} {
\pic[bs colored multibox={main color \c}]{bs multibox={name=a,l=\l,r=\r,label=\tikzcontour{$\label$}}};
}

\pic[default bs band]{bs band={name=b,pairs={1-16,2-15,3-14},wide={4-6}{7-9},wide={10-11}{12-13}}};
\foreach \i/\n in {1/1,2/2,16/10} \pic[bs colored single box=main color 1]{bs box={name=b,i=\i,label=$\n$}};
\foreach \i\n in {3/3,11,14/8,15/9} \pic[bs colored single box=main color 2]{bs box={name=b,i=\i,label=$\n$}};
\foreach \l/\r/\n/\c in {4/6/4/2,7/9/5/1,10/11/6/1,12/13/7/2} \pic[bs colored single box=main color \c]{bs box={name=b,l=\l,r=\r,label=\n}};

\foreach \n in {a,b} \pic[default bs mark]{bs mark={name=\n,after={0,2,6,11,15,16}}};
\foreach \l/\r/\label/\c in {1/2/W_0/1,3/6/T_1/2,7/11/W_1/1,12/15/T_2/2,16/16/W_2/1} \draw[evaluate={\opacity=(\l==\r?0:1);},main color \c!50!black,opacity=\opacity,decorate,decoration={brace},preaction={draw=white,opacity=\opacity,line width=2pt,line cap=round}] ([yshift=19pt]b-\l-l) -- ([yshift=19pt]b-\r-r) node[opacity=1,midway,above=10pt,rotate=90,anchor=center] (tmp) {\tikzcontour{$\prec$}} node[opacity=1,above=2pt of tmp.center] {\tikzcontour{$\label$}};
\end{tikzpicture}
\caption{The cancellation bundle (on the bottom) induced by a cancellation band system (on the top). In this specific example, we have that $\card{w_0}=2$, $\card{t_1}=2$, $\card{w_1}=5$, $\card{t_2}=4$, and $\card{w_2}=1$. The word $w(\vec{k})$ on the top corresponds to the vector $\vec{k}=(2,1)$, and reduces to $1$ via the cancellation band system represented in the picture -- although this is of course impossible to verify without knowing the actual the words involved. The marked spots separate the different ``blocks'' of $w(\vec{k})$, namely $w_0$, $t_1^2$, $w_1$, $t_2$, and $w_2$. The maximal bundling procedure yields the marked band system depicted on the bottom, with vertical arrows representing the bundling map. Here, the integers in $\{1,\ldots,10\}$ are labelled with $W_0$, $T_1$, $W_1$, $T_2$, or $W_2$ according to the block they come from.\label{fig:cancellation bundle example}}
\end{figure}
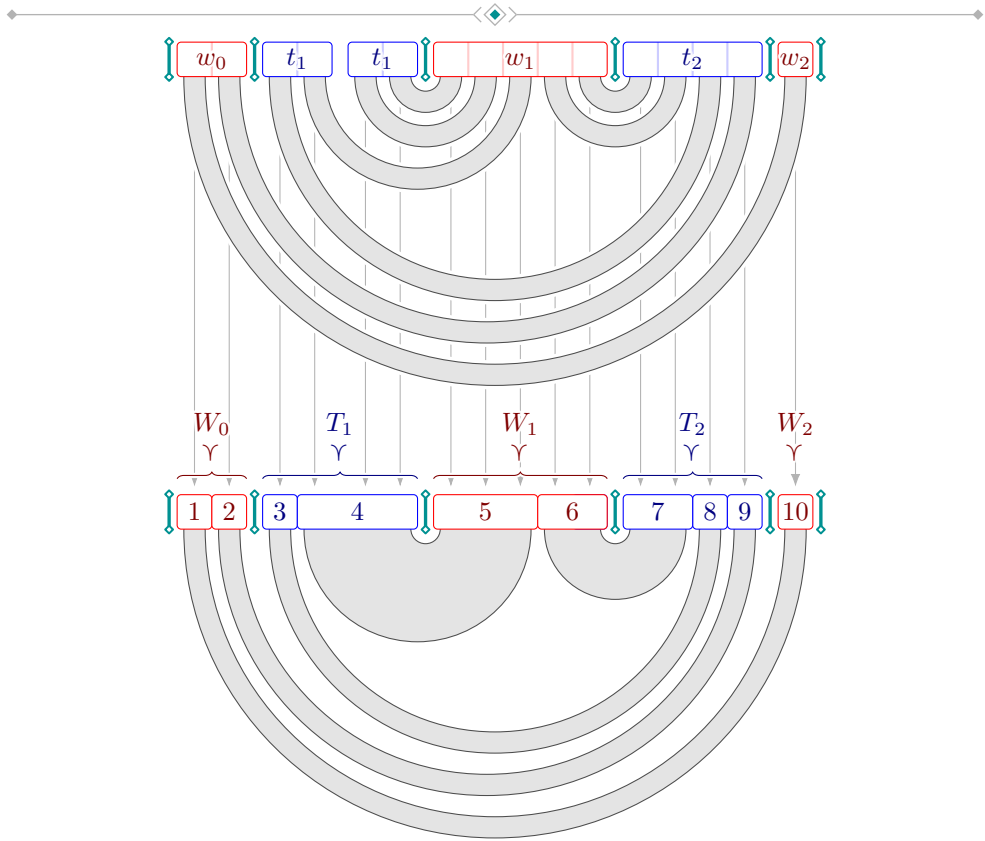

The crucial fact we seek to prove now is that $2n'$ is bounded above by a constant which does not depend on $\vec{k}$. This is a consequence of the following two statements.
\begin{enumarabic}
\item For each $1\le i\le l$, there is no band of $B'$ with both endpoints labelled $T_i$. In fact, suppose that $(a',b')\in B'$ is such a band, and let $(a,b)$ be a band of $B$ with $\iota(a,b)=(a',b')$. By \cref{thm:band of width 1:1} of \cref{thm:band of width 1}, there is a band $(c,c+1)\in B$ with $a\le c<b$. But this is impossible, since positions $a$ and $b$ (and therefore $c$ and $c+1$) belong to the block $t_i^{k_i}$; the word $t_i$ is cyclically reduced, therefore there can be no cancellation between the symbols in positions $c$ and $c+1$.
\item For each $1\le i<j\le l$, there is at most one band of $(a',b')\in B'$ with $a'\prec T_i$ and $b'\prec T_j$. In fact, suppose that $(a_1',b_1')$ and $(a_2',b_2')$ are two such bands, with $a_1'<a_2'<b_2'<b_1'$. Let $(a_1,b_1)$ and $(a_2,b_2)$ be bands of $B$ with $\iota(a_1,b_1)=(a_1',b_1')$ and $\iota(a_2,b_2)=(a_2',b_2')$. Consider \cref{thm:band of width 1:2} of \cref{thm:band of width 1}. Clearly,
\[
\iota(a_1,b_1)=(a_1',b_1')\neq(a_2',b_2')=\iota(a_2,b_2).
\]
Moreover, since $a_1',a_2'\prec T_i$ and $b_1',b_2'\prec T_j$, we have that
\[
p_{2i}<a_1<a_2<p_{2i+1}\le p_{2j}<b_2<b_1<p_{2j+1}.
\]
As a consequence, there must be a band $(c,c+1)\in B$ for some
\[
c\in\{a_1,\ldots,a_2-1\}\cup\{b_2,\ldots,b_1-1\},
\]
but this is once again impossible since $t_i$ and $t_j$ are cyclically reduced words.
\end{enumarabic}

As anticipated, these two facts imply the existence of an upper bound $C$ for the length $2n'$ of $B'$ which is independent of $\vec{k}$. An estimate for $C$ can be found as follows. Say that a band $(a',b')\in B'$ is:
\begin{itemize}
\item of \emph{type 1} if $a'\prec W_i$ or $b'\prec W_i$ for some $0\le i\le l$;
\item of \emph{type 2} if it is not of type 1, that is if $a'\prec T_i$ and $b'\prec T_j$ for some $1\le i,j\le l$.
\end{itemize}
An easy counting argument shows that there are at most $\card{w_0}+\ldots+\card{w_l}$ bands of type 1 and at most $l(l-1)/2$ bands of type 2. Therefore, a suitable upper bound for $2n'$ is
\[
2n'\le C=2(\card{w_0}+\ldots+\card{w_l})+l(l-1).
\]

Let us now forget our choice of $\vec{k}$, and fix any maximal marked band system $(B',p')$ such that $B'$ has length $2n'\le C$ and $p'=(p'_1,\ldots,p_{2l+2}')$ with $p_1'=1/2$ and $p_{2l+2}'=2n'+1/2$. Since there are only finitely many such marked band systems, and since every $\vec{k}\in\AAA^+$ has a cancellation bundle of this form, it is enough to describe an algorithm to compute the set
\[
\left\{\vec{k}=(k_1,\ldots,k_l)\in\ZZ_{\ge 0}^l:\,
\begin{matrix*}[l]
w_0t_1^{k_1}w_1t_2^{k_2}w_2\cdots t_l^{k_l}w_l=1,\\
\text{$(B',p')$ is a cancellation bundle for $\vec{k}$}
\end{matrix*}
\right\},
\]
which will be henceforth referred to as $\AAA^+_{(B',p')}$.

Let us remark that the labelling procedure described above can be carried out without any reference to $\vec{k}$; in fact, only the integer $2n'$ and the marking $p'$ are required. Therefore, we will from now on assume that the integers in $\{1,\ldots,2n'\}$ are labelled with symbols $W_0,\ldots,W_l,T_1,\ldots,T_l$ according to the marking $p'$.

\step{Fixed unbundling class.} Let us dwell some more on what it means for a vector $\vec{k}\in\ZZ_{\ge0}^l$ to have $(B',p')$ as a cancellation bundle. Recall that band systems having $(B',p')$ as their maximal bundle are parametrised by unbundling maps for $(B',p')$. It is easy to see that an unbundling map $\varphi$ for $(B',p')$ describes a cancellation band system for $\vec{k}$ if and only if the following conditions are satisfied:
\begin{enumroman}
\item $\displaystyle\sum_{a'\prec W_i}\varphi(a')=\card{w_i}$ for $0\le i\le l$;
\item\label{it:cancellation map for k:2} $\displaystyle\sum_{a'\prec T_i}\varphi(a')=k_i\card{t_i}$ for $1\le i\le l$;
\item $\displaystyle\wordrange{w(\vec{k})}{\what{\varphi}(a')+1}{\what{\varphi}(a'+1)}=\left(\wordrange{w(\vec{k})}{\what{\varphi}(b')+1}{\what{\varphi}(b'+1)}\right)^{-1}$ for each $(a',b')\in B'$.
\end{enumroman}
Moreover, condition \ref{it:cancellation map for k:2} implies that $\varphi$ describes a cancellation band system for at most one vector $\vec{k}(\varphi)\in\ZZ_{\ge0}^l$, whose coordinates can be easily computed as
\begin{align}\label{eq:definition of k(phi)}
&k_i(\varphi)=\frac{\sum_{a'\prec T_i}\varphi(a')}{\card{t_i}}&\text{for $1\le i\le l$}
\end{align}
(recall that $\card{t_i}\neq 0$ for each $i$).

Let us say that an unbundling map $\varphi$ for $(B',p')$ is \emph{compatible with the labelling} if
\begin{align*}
&\text{$\displaystyle\sum_{a'\prec W_i}\varphi(a')=\card{w_i}$ for $0\le i\le l$}&\text{and}&&\text{$\displaystyle\sum_{a'\prec T_i}\varphi(a')\equiv 0\pmod{\card{t_i}}$ for $1\le i\le l$.}
\end{align*}
If, additionally, we have
\begin{align*}
&\wordrange{w(\vec{k}(\varphi))}{\what{\varphi}(a')+1}{\what{\varphi}(a'+1)}=\left(\wordrange{w(\vec{k}(\varphi))}{\what{\varphi}(b')+1}{\what{\varphi}(b'+1)}\right)^{-1}&\text{for each $(a',b')\in B'$},
\end{align*}
then we say that $\varphi$ is \emph{cancelling}. We can then write the following equivalent definition for $\AAA^+_{(B',p')}$:
\[
\AAA^+_{(B',p')}=\left\{\vec{k}(\varphi):\,
\begin{matrix*}[l]
\text{$\varphi$ is an unbundling map for $(B',p')$ which is}\\\text{compatible with the labelling and cancelling}
\end{matrix*}
\right\}.
\]

Let us define an equivalence relation on the set of unbundling maps for $(B',p')$ which are compatible with the labelling. Given two unbundling maps $\varphi$ and $\psi$, we say that $\varphi\sim\psi$ if the following properties hold for each  $1\le a'\le 2n'$:
\begin{itemize}
\item $\varphi(a')=\psi(a')$ whenever $a'$ is an endpoint of a band of type 1;
\item $\varphi(a')\equiv\psi(a')\pmod{\card{t_i}}$ whenever $a'\prec T_i$ for some $1\le i\le l$.
\end{itemize}
Fix an equivalence class $\Xi$. After choosing a representative $\varphi\in\Xi$, with slight abuse of notation, for each $1\le a'\le 2n'$ let us denote by $\Xi(a')$:
\begin{itemize}
\item the integer $\varphi(a')$ if $a'\prec W_i$ for some $0\le i\le l$;
\item the residue class of $\varphi(a')$ modulo $\card{t_i}$ if $a'\prec T_i$ for some $1\le i\le l$.
\end{itemize}
In both cases, the value of $\Xi(a')$ does not depend on the choice of $\varphi$. The values of $\Xi(a')$ for $1\le a'\le 2n'$ uniquely identify $\Xi$, hence there are only finitely many equivalence classes of unbundling maps. As a consequence, it is enough to describe an algorithm to compute the set
\[
\left\{\vec{k}(\varphi):\text{$\varphi\in\Xi$ is cancelling}\right\},
\]
which will be henceforth referred to as $\AAA^+_{(B',p'),\Xi}$.

\step{Final computation.} The reason why fixing the unbundling class $\Xi$ is beneficial is that conditions of the form
\begin{equation}\label{eq:subword cancellation}
\wordrange{w(\vec{k}(\varphi))}{\what{\varphi}(a')+1}{\what{\varphi}(a'+1)}=\left(\wordrange{w(\vec{k}(\varphi))}{\what{\varphi}(b')+1}{\what{\varphi}(b'+1)}\right)^{-1}
\end{equation}
for bands $(a',b')\in B'$ become mutually independent. More precisely, whether condition \cref{eq:subword cancellation} is satisfied for a band $(a',b')$ or not only depends on the value of $\varphi(a')$, as long as $\varphi\in\Xi$; we will now clarify and justify this claim.

Let $a'\in\{1,\ldots,2n'\}$ be an integer, and let $X\in\{W_0,\ldots,W_l,T_1,\ldots,T_l\}$ be its label. Define
\[
h(a')=\sum_{\substack{b'\prec X\\b'<a'}}\Xi(b').
\]
Note that, when $X=T_i$ for some $1\le i\le l$, the integer $h(a')$ is only well-defined modulo $\card{t_i}$; by convention, we will assume that $0\le h(a')<\card{t_i}$. We can now rewrite the subwords involved in \cref{eq:subword cancellation}: for each $\varphi\in\Xi$ we have
\[
\wordrange{w(\vec{k}(\varphi))}{\what{\varphi}(a')+1}{\what{\varphi}(a'+1)}=\begin{dcases*}
\wordrange{(w_i)}{h(a')+1}{h(a')+\varphi(a')}&if $a'\prec W_i$,\\
\wordrange{(t_i^{\infty})}{h(a')+1}{h(a')+\varphi(a')}&if $a'\prec T_i$.
\end{dcases*}
\]

Consider a band $(a',b')\in B'$. It is now clear that whether $(a',b')$ satisfies \cref{eq:subword cancellation} or not only depends on the value of $\varphi(a')=\varphi(b')$. Our next task it to show how to compute the set
\[
S(a',b')=\{s\in\ZZ_{>0}:\text{$(a',b')$ satisfies \cref{eq:subword cancellation} if and only if $\varphi(a')=s$}\},
\]
and we do so by analysing two cases.
\begin{substeps}
\item Suppose that $(a',b')$ is a band of type 1, with $a'\prec W_i$ for some $0\le i\le l$ (the case where $b'\prec W_i$ is identical). Then $\varphi(a')=\Xi(a')$ is fixed, and condition \cref{eq:subword cancellation} reads
\begin{align*}
&\wordrange{(w_i)}{h(a')+1}{h(a')+\Xi(a')}=\left(\wordrange{(w_j)}{h(b')+1}{h(b')+\Xi(a')}\right)^{-1}&&\text{if $b'\prec W_j$, or}\\
&\wordrange{(w_i)}{h(a')+1}{h(a')+\Xi(a')}=\left(\wordrange{(t_j^{\infty})}{h(b')+1}{h(b')+\Xi(a')}\right)^{-1}&&\text{if $b'\prec T_j$.}
\end{align*}
Either way, whether the equality of words holds or not only depends on $\Xi$. We have that $S(a',b')=\{\Xi(a')\}$ if the two words are equal, and $S(a',b')=\emptyset$ if they are not.
\item Suppose that $(a',b')$ is a band of type 2, with $a'\prec T_i$ and $b'\prec T_j$ for some $1\le i<j\le l$. Then all the possible values for $\varphi(a')=\varphi(b')$ are of the form $q+Lv$, where
\begin{itemize}
\item $L=\lcm(\card{t_i},\card{t_j})$,
\item $v$ is a non-negative integer, and
\item $q$ is the only integer in $\{1,\ldots,L\}$ such that $q\equiv\Xi(a')\pmod{\card{t_i}}$ and $q\equiv\Xi(b')\pmod{\card{t_j}}$;
\end{itemize}
if no such $q$ exists then the class $\Xi$ is empty, and we conclude that $\AAA^+_{(B',p'),\Xi}=\emptyset$. Otherwise, all that is left to do is determining for which values of $v\ge 0$ the band $(a',b')$ satisfies \cref{eq:subword cancellation}. We can rewrite
\[
\wordrange{(t_i^{\infty})}{h(a')+1}{h(a')+\varphi(a')}=\left(\wordrange{(t_i^{\infty})}{h(a')+1}{h(a')+L}\right)^v\wordrange{(t_i^{\infty})}{h(a')+1}{h(a')+q},
\]
and similarly for $\wordrange{(t_j^{\infty})}{h(b')+1}{h(b')+\varphi(b')}$. Therefore, condition \cref{eq:subword cancellation} reads
\begin{multline*}
\left(\wordrange{(t_i^{\infty})}{h(a')+1}{h(a')+L}\right)^v\wordrange{(t_i^{\infty})}{h(a')+1}{h(a')+q}=\\
\left(\wordrange{(t_j^{\infty})}{h(b')+q+1}{h(b')+q+L}\right)^{-v}\left(\wordrange{(t_j^{\infty})}{h(b')+1}{h(b')+q}\right)^{-1}.
\end{multline*}
There are three cases.
\begin{itemize}
\item If
\[
\wordrange{(t_i^{\infty})}{h(a')+1}{h(a')+q}\neq\left(\wordrange{(t_j^{\infty})}{h(b')+1}{h(b')+q}\right)^{-1},
\]
then \cref{eq:subword cancellation} is not satisfied for any value of $v$, and $S(a',b')=\emptyset$.
\item If
\[
\wordrange{(t_i^{\infty})}{h(a')+1}{h(a')+q}=\left(\wordrange{(t_j^{\infty})}{h(b')+1}{h(b')+q}\right)^{-1}
\]
but
\[
\wordrange{(t_i^{\infty})}{h(a')+1}{h(a')+L}\neq\left(\wordrange{(t_j^{\infty})}{h(b')+q+1}{h(b')+q+L}\right)^{-1},
\]
then \cref{eq:subword cancellation} is satisfied only for $v=0$, leading to $S(a',b')=\{q\}$.
\item If
\[
\wordrange{(t_i^{\infty})}{h(a')+1}{h(a')+L}=\left(\wordrange{(t_j^{\infty})}{h(b')+q+1}{h(b')+q+L}\right)^{-1},
\]
then \cref{eq:subword cancellation} is satisfied for every value of $v\ge 0$, so $S(a',b')=\{q+Lv:v\ge0\}$.
\end{itemize}
\end{substeps}

In conclusion, we have shown how to compute a set $S(a',b')$ for each band $(a',b')\in B'$ with the following property: a function $\map{\varphi}{\{1,\ldots,2n'\}}{\ZZ_{>0}}$ is a cancelling representative of $\Xi$ if and only if $\varphi(a')=\varphi(b')\in S(a',b')$ for every $(a',b')\in B'$. If $S(a',b')=\emptyset$ for some band $(a',b')$ then clearly $\AAA^+_{(B',p'),\Xi}=\emptyset$. Otherwise, each set $S(a',b')$ can be written in the form
\[
S(a',b')=\{q(a',b')+L(a',b')v:v\in\ZZ_{\ge0}\}
\]
for some integers $q(a',b')>0$ and $L(a',b')\ge 0$. Therefore, every cancelling unbundling map $\varphi\in\Xi$ is described by a (not necessarily unique) vector $\vec{v}\in\ZZ_{\ge0}^{n'}$ whose coordinates $v_{(a',b')}$ are indexed by bands $(a',b')\in B'$, such that
\begin{align*}
&\varphi(a')=\varphi(b')=q(a',b')+L(a',b')v_{(a',b')}&\text{for each $(a',b')\in B'$}.
\end{align*}
Conversely, each vector $\vec{v}\in\ZZ_{\ge 0}^{n'}$ describes a cancelling unbundling map $\varphi(\vec{v})\in\Xi$, defined by the previous equations. As a consequence, we find that
\[
\AAA^+_{(B',p'),\Xi}=\left\{\vec{k}(\varphi(\vec{v})):\vec{v}\in\ZZ_{\ge0}^{n'}\right\}.
\]
Finally, recalling formula \cref{eq:definition of k(phi)}, it is easy to see that $\vec{k}(\varphi(\vec{v}))$ can be computed as $\vec{k}(\varphi(\vec{v}))=\vec{z}+\mat{M}\vec{v}$, where:
\begin{itemize}
\item $\vec{z}\in\ZZ_{\ge 0}^l$ is the vector whose coordinates are defined by
\begin{align*}
&z_i=\frac{\sum_{a'\prec T_i}q(a')}{\card{t_i}}&\text{for $1\le i\le l$},
\end{align*}
where for ease of notation we have defined $q(a')=q(b')=q(a',b')$ for each band $(a',b')\in B'$;
\item $\mat{M}\in\ZZ^{l\times n'}$ is the matrix whose entries are indexed by $\{1,\ldots,l\}\times B'$ and defined by
\begin{align*}
&M_{i,(a',b')}=\begin{dcases*}
\frac{L(a',b')}{\card{t_i}}&if $a'\prec T_i$ or $b'\prec T_i$,\\
0&otherwise
\end{dcases*}&\text{for $1\le i\le l$, $(a',b')\in B'$.}
\end{align*}
\end{itemize}
Therefore, we find that
\[
\AAA^+_{(B',p'),\Xi}=\left\{\vec{z}+\mat{M}\vec{v}:\vec{v}\in\ZZ_{\ge 0}^{n'}\right\}.\qedhere
\]
\end{proof}

\subsection{Topological applications}\label{sec:free groups:topology}

As anticipated, the reason why we are interested in \cref{thm:exponential equation free group} is the following topological consequence. In \cref{sec:classification algorithm:small i-bundles}, we will make use of \cref{thm:product of dehn twists extension to handlebody} to solve the isotopy problem for surfaces which bound a genus\=/two handlebody on one side, while the other side is such that the trace of its mapping class group can be described in terms of Dehn twists about disjoint curves. Even though we will only apply this corollary to handlebodies of genus $g=2$, we state and it and prove it in greater generality, at no additional cost in terms of effort or simplicity\footnote{In fact, the same proof works for a statement which is even more general. Given pairwise disjoint curves $a_1,\ldots,a_m\subs\boundary V$, a homeomorphism $\map{f}{\boundary V}{\boundary V}$, and a matrix $\mat{A}$ and a vector $\vec{b}$ with integer entries of suitable sizes, we can decide whether there exists a vector $\vec{h}\in\ZZ^m$ such that $\mat{A}\vec{h}\le\vec{b}$ and $\twist{a_1}^{h_1}\cdots\twist{a_m}^{h_m}f$ extends to a homeomorphism $\umap{V}{V}$.}.

\begin{corollary}\label{thm:product of dehn twists extension to handlebody}
Let $V$ be a handlebody of genus $g\ge 1$. Let $m$ be a non\=/negative integer, and let $a_1,\ldots,a_{2m}$ be pairwise disjoint curves in $\boundary{V}$. Let $\map{f}{\boundary{V}}{\boundary{V}}$ be a homeomorphism. There is an algorithm which, given as input $a_1,\ldots,a_{2m}$ and $f$, decides whether there exist integers $h_1,\ldots,h_m$ such that
\[
\map{\twist{a_1}^{h_1}\cdots\twist{a_m}^{h_m}\twist{a_{m+1}}^{-h_1}\cdots\twist{a_{2m}}^{-h_m}f}{\boundary V}{\boundary V}
\]
extends to a homeomorphism $\umap{V}{V}$.
\end{corollary}
\begin{proof}
Fix a basepoint $x_0\in\boundary{V}$ once and for all; additionally, let us arbitrarily pick orientations for $a_1,\ldots,a_{2m}$. Consider a complete system of oriented meridians $b_1,\ldots,b_g\subs\boundary{V}$. For each $b_r$, we give a few definitions. Let $f(b_r)$ intersect $a_1\cup\ldots\cup a_{2m}$ transversely at points $p_{r,1},\ldots,p_{r,l_r}$, numbered in the order they appear on $f(b_r)$. For $j=1,\ldots,l_r$, let (when $j=l_r$, we take $j+1$ to mean $1$):
\begin{itemize}
\item $\beta_{r,j}$ be the subarc of $f(b_r)$ going from $p_{r,j}$ to $p_{r,j+1}$;
\item $\gamma_{r,j}\subs\boundary{V}$ be an arc joining $x_0$ to $p_{r,j}$;
\item $i(r,j)$ be the only integer in $\{1,\ldots,2m\}$ such that $p_{r,j}\in a_{i(r,j)}$;
\item $w_{r,j}$ be the element of $\pi_1(V,x_0)$ represented by $\gamma_{r,j}\ast\beta_{r,j}\ast\gamma_{r,j+1}^{-1}$, where ``$\ast$'' denotes the concatenation of paths;
\item $t_{r,j}$ be the element of $\pi_1(V,x_0)$ represented by $\gamma_{r,j}\ast a_{i(r,j)}^\epsilon\ast\gamma_{r,j}^{-1}$, where we consider $a_{i(r,j)}$ as a loop based at $p_{r,j}$, and $\epsilon$ is $1$ or $-1$ depending on the ``sign'' of the intersection between $f(b_r)$ and $a_{i(r,j)}$ at $p_{r,j}$, as depicted in \cref{fig:product of dehn twists extension to handlebody:sign}.
\end{itemize}
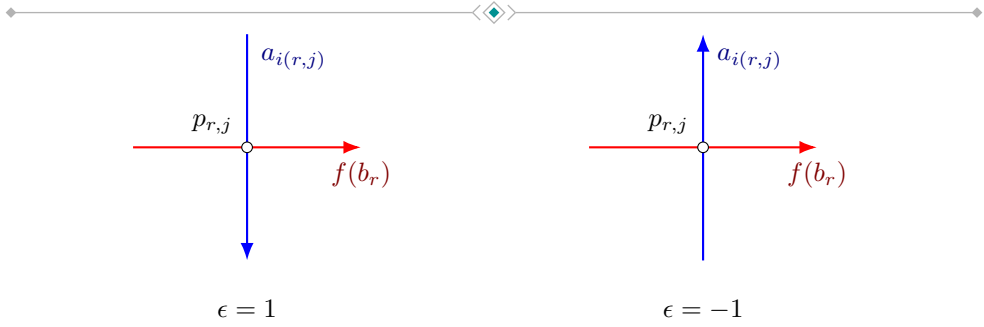
\begin{figure}
\centering
\tikzsetnextfilename{product-of-dehn-twists-extension-to-handlebody-sign}
\begin{tikzpicture}[thick,every node/.style={outer sep=2pt},evaluate={\r=1.5;}]
\foreach \xshift/\ep/\ar in {0/1/<-,6 cm/-1/->} {
\begin{scope}[xshift=\xshift]
\draw[->,main color 1] (-\r,0) -- (\r,0) node[below,main color 1!50!black] {$f(b_r)$};
\draw[\ar,main color 2] (0,-\r) -- (0,\r) node[below right,main color 2!50!black] {$a_{i(r,j)}$};
\draw[thin,black,fill=white] circle (2pt) node[above left] {$p_{r,j}$};
\node[anchor=base] at (0,-2.25) {$\epsilon=\ep$};
\end{scope}
}
\end{tikzpicture}
\caption{Choice of $\epsilon$ at the intersection $p_{r,j}$ between $f(b_r)$ and $a_{i(r,j)}$.\label{fig:product of dehn twists extension to handlebody:sign}}
\end{figure}

It is not hard to see that, given integers $h_1,\ldots,h_{2m}$, the conjugacy class of the curve $\twist{a_1}^{h_1}\cdots\twist{a_{2m}}^{h_{2m}}f(b_r)$ in $\pi_1(V,x_0)$ is represented by the element
\[
t_{r,1}^{h_{i(r,1)}}w_{r,1}\cdots t_{r,l_r}^{h_{i(r,l_r)}}w_{r,l_r};
\]
we hope that \cref{fig:product of dehn twists extension to handlebody:effect of dehn twists} will be convincing enough for the reader. Consider the set
\[
\AAA_r=\left\{(k_1,\ldots,k_{l_r})\in\ZZ^{l_r}\colon t_{r,1}^{k_1}w_{r,1}\cdots t_{r,l_r}^{k_{l_r}} w_{r,l_r}=1\right\}.
\]
By the previous remark, we have that $\twist{a_1}^{h_1}\cdots \twist{a_{2m}}^{h_{2m}}f(b_r)$ is a meridian if and only if $\left(h_{i(r,j)}\right)_{j=1}^{l_r}\in\AAA_r$. Consequently, $\twist{a_1}^{h_1}\cdots\twist{a_{2m}}^{h_{2m}}f$ extends to a homeomorphism of $V$ if and only if $\left(h_{i(r,j)}\right)_{j=1}^{l_r}\in\AAA_1\cap\ldots\cap\AAA_g$.

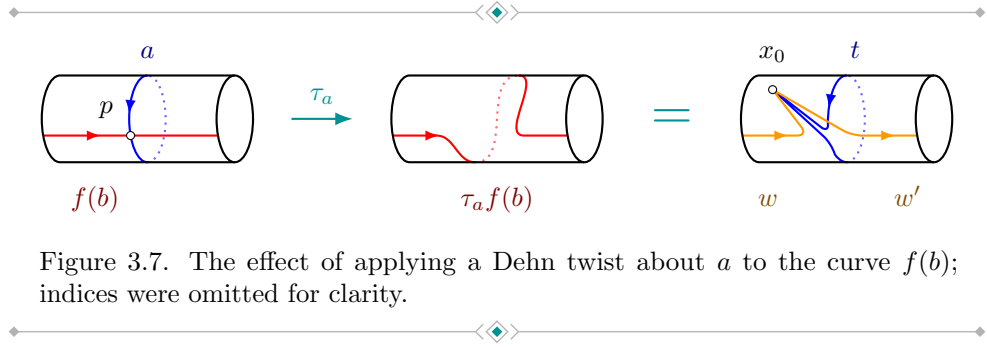
\begin{figure}
\centering
\tikzsetnextfilename{effect-of-dehn-twist}
\begin{tikzpicture}[scale=1.15,thick,
    vertex/.pic={\fill[pic actions,postaction={draw=black,thin}] circle (1.5pt);},
    place arrow at/.style={postaction={decorate,decoration={markings,mark=at position #1 with {\arrow[xshift=2pt]{>[scale=.8]};}}}},
    place reversed arrow at/.style={postaction={decorate,decoration={markings,mark=at position #1 with {\arrowreversed[xshift=-2pt]{>[scale=.8]};}}}}]
\tikzmath{\rr=.5/sqrt(1/.385^2-1);\sh=4 cm;}
\draw[main color 1,place arrow at=.3] (-1,0,.5) -- ++(2,0);
\draw[main color 2!60,dotted] (0,-.5) arc[start angle=-90,end angle=90,x radius=\rr,y radius=.5];
\draw[main color 2,place reversed arrow at=.6] (0,-.5) arc[start angle=270,end angle=90,x radius=\rr,y radius=.5];
\node[main color 1!50!black,anchor=base] at (-.6,-1) {$f(b)$};
\node[main color 2!50!black,anchor=base] at (0,.7) {$a$};
\pic[white] at (0,0,.5) {vertex};
\node[black,above left=4pt] at (0,0,.5) {$p$};
\begin{scope}[shift={(\sh pt,0)}]
\draw[main color 1!60,dotted] (-.25,-.5,0) to[out=0,in=180] (.25,.5,0);
\draw[main color 1,place arrow at=.2] (-1,0,.5) -- (-.5,0,.5) to[out=0,in=180] (-.25,-.5,0) (.25,.5,0) to[out=0,in=180] (.5,0,.5) -- (1,0,.5);
\node[main color 1!50!black,anchor=base] at (0,-1) {$\tau_af(b)$};
\end{scope}
\begin{scope}[shift={(2*\sh pt,0)}]
\coordinate (x0) at (-.75,{.5*sin(60)},{.5*cos(60)});
\draw[main color 2!60,dotted] (0,-.5,0) arc[start angle=-90,end angle=90,x radius=\rr,y radius=.5];
\draw[main color 2,rounded corners,line join=round,place arrow at=.6] (0,-.5,0) arc[start angle=-90,end angle=-145,x radius=\rr,y radius=.5] to (x0) (0,.5,0) arc[start angle=90,end angle=200,x radius=\rr,y radius=.5] to (x0);
\draw[main color 4,rounded corners,place arrow at=.15,place arrow at=.9] (-1,0,.5) -- (-.25,0,.5) to (x0) (x0) -- (.25,0,.5) -- (1,0,.5);
\pic[white] at (x0) {vertex};
\node[black,anchor=base] at (x0|-0,.7) {$x_0$};
\node[main color 4!50!black,anchor=base] at (-.9,-1) {$w$};
\node[main color 4!50!black,anchor=base] at (.7,-1) {$w'$};
\node[main color 2!50!black,anchor=base] at (.1,.7) {$t$};
\end{scope}
\foreach \i in {0,1,2} {
    \draw[xshift={\i*\sh}] (1,0) circle[x radius=\rr,y radius=.5];
    \draw[xshift={\i*\sh}] (1,-.5) -- (-1,-.5) arc[start angle=270,end angle=90,x radius=\rr,y radius=.5] -- (1,.5);
}

\draw[theme color,->] ({\sh/2-10pt},0) -- ({\sh/2+10pt},0) node[midway,above=2pt] {$\twist{a}$};
\node[theme color,scale=2] at ({3*\sh/2 pt},0) {$=$};

\end{tikzpicture}
\caption{The effect of applying a Dehn twist about $a$ to the curve $f(b)$; indices were omitted for clarity.}
\label{fig:product of dehn twists extension to handlebody:effect of dehn twists}
\end{figure}

By \cref{thm:exponential equation free group}, the sets $\AAA_r$ can be algorithmically decomposed as
\[
\AAA_r=A_{r,1}\cup\ldots\cup A_{r,N_r},
\]
where
\[
A_{r,q}=\left\{\vec{z}_{r,q}+\mat{M}_{r,q}\vec{v}:\vec{v}\in\ZZ_{\ge0}^{d_{r,q}}\right\}
\]
for some vector $\vec{z}_{r,q}\in\ZZ^{l_r}$ and some matrix $\mat{M}_{r,q}\in\ZZ^{l_r\times d_{r,q}}$. Therefore, we finally find that there exist integers $h_1,\ldots,h_m$ such that
\[
\twist{a_1}^{h_1}\cdots\twist{a_m}^{h_m}\twist{a_{m+1}}^{-h_1}\cdots\twist{a_{2m}}^{-h_m}f
\]
extends to a homeomorphism of $V$ if and only if, for some choice of indices $q_1\in\{1,\ldots,N_1\},\allowbreak{}\ldots,\allowbreak{}q_g\in\{1,\ldots,N_g\}$, the following system of linear equations has at least one solution in the variables $\vec{v}_1\in\ZZ_{\ge 0}^{d_{1,q_1}},\ldots,\vec{v}_g\in\ZZ_{\ge 0}^{d_{g,q_g}}$, $h_1,\ldots,h_m\in\ZZ$:
\[
\begin{dcases*}
\vec{z}_{r,q_r}+\mat{M}_{r,q_r}\vec{v}_r=\left(h_{i(r,j)}\right)_{j=1}^{l_r}&for $r=1,\ldots,g$,\\
h_r=-h_{m+r}&for $r=1,\ldots,m$.
\end{dcases*}
\]
This condition can be checked algorithmically (a classical result due to \textcite{integer-programming}) for every choice of indices $q_1,\ldots,q_g$, providing an algorithmic solution to the question in the statement.
\end{proof}

\section{The classification algorithm}\label{sec:classification algorithm}

\subsection{Outline}

\Cref{sec:classification algorithm} will be entirely devoted to providing a proof of the following.

\begin{theorem}\label{thm:classification algorithm}
There is an algorithm to decide whether two genus-two oriented surfaces embedded in $\sphere[3]$ are isotopic or not.
\end{theorem}

More precisely, let $S_1$ and $S_2$ be genus\=/two surfaces embedded in $\sphere[3]$, and let us fix an orientation for $S_1$ and for $S_2$. We say that $S_1$ and $S_2$ are isotopic (as oriented surfaces) if there is an isotopy $\map{f_t}{\sphere[3]}{\sphere[3]}$ such that $f_0$ is the identity and $f_1$ maps $S_1$ to $S_2$ orientation\=/preservingly. Since every self\=/homeomorphism of $\sphere[3]$ is isotopic to the identity, the oriented surfaces $S_1$ and $S_2$ are isotopic if and only if there exists a homeomorphism $\map{f}{\sphere[3]}{\sphere[3]}$ sending $S_1$ to $S_2$ orientation\=/preservingly.

Another equivalent definition is the following. Let $M_1$ be the closure of the component of $\sphere[3]\setminus S_1$ which lies on the positive side of $S_1$, and let $N_1$ be the closure of the other component; we will sometimes call $M_1$ and $N_1$ the \emph{sides} of $S_1$. Define $M_2$ and $N_2$ in a similar way. Then $S_1$ and $S_2$ are isotopic if and only if there exist homeomorphisms $\map{f}{M_1}{M_2}$ and $\map{g}{N_1}{N_2}$ such that $\trace{f}$ and $\trace{g}$ are isotopic as homeomorphisms $\umap{S_1}{S_2}$.

\begin{remark}
A non\=/oriented version of \cref{thm:classification algorithm} states that there is an algorithm to decide whether two genus\=/two surfaces embedded in $\sphere[3]$ are isotopic or not. Of course, the non\=/oriented version is an immediate corollary of the oriented one. In fact, in order to decide whether the surfaces $S_1$ and $S_2$ are isotopic without any constraint on orientation, we can simply check if they are orientation\=/preservingly isotopic for at least one choice of orientations.
\end{remark}

From now on, we will take $S_1$, $S_2$, $M_1$, $N_1$, $M_2$, and $N_2$ to refer to the objects we have already introduced; in particular, we will assume orientations for $S_1$ and $S_2$ have been fixed. Our aim will be to provide a proof of \cref{thm:classification algorithm}. The full algorithm is quite involved, and for the sake of convenience we will split it into several cases. Before we start describing the actual algorithm, let us briefly outline the structure of the following sections.
\begin{substeps}
\item Several cases of the proof of \cref{thm:classification algorithm} will be addressed by finding \emph{canonical} compression discs for $S_1$ and $S_2$, then compressing $S_1$ and $S_2$ along these discs and -- carefully -- reducing the question to the isotopy problem of tori in $\sphere[3]$. In \cref{sec:classification algorithm:general strategies} we provide a general framework for dealing with these canonical compression discs; by abstracting the repetitive parts of the algorithm away, we will later be able to focus on what is meaningfully different in each of these cases.
\item In \cref{sec:classification algorithm:compression discs} we deal with the problem of finding such canonical compression discs. The recipe we present relies on showing the existence of discs satisfying properties which are \emph{stable under boundary compressions}; if these properties are suitably chosen, uniqueness will easily follow.
\item After the introductory sections, we finally start describing the classification algorithm. \Cref{sec:classification algorithm:one non-separating compression disc,sec:classification algorithm:one separating compression disc} deal with the case where $S_1$ and $S_2$ are \emph{partially compressible on one side}. Specifically, the case where one side has at least one -- and, \emph{a posteriori}, only one -- non\=/separating compression disc is addressed in \cref{sec:classification algorithm:one non-separating compression disc}; \cref{sec:classification algorithm:one separating compression disc} explains the strategy to follow when one side only has separating compression discs. In both cases, we heavily rely on the tools developed in the previous two sections.
\item We are then left to address the case where one side of $S_1$ is a handlebody and the other is boundary irreducible. The algorithm we use here will depend on the JSJ decomposition of the boundary irreducible component. If all the $I$\=/bundle pieces are ``small'', then the results presented in \cref{sec:free groups} will quickly lead to a solution of the isotopy problem, as described in \cref{sec:classification algorithm:small i-bundles}. Otherwise, only two cases can arise: either a product $I$\=/bundle over a punctured torus appears in the JSJ decomposition -- and we deal with this in \cref{sec:classification algorithm:product bundle over punctured torus} -- or one of the pieces is a twisted $I$\=/bundle over a punctured Klein bottle; the latter situation is addressed in \cref{sec:classification algorithm:twisted bundle over punctured klein bottle}. In both cases, an \emph{ad hoc} discussion is enough to settle the isotopy problem.
\end{substeps}

\subsection{General strategy for canonical compression discs}\label{sec:classification algorithm:general strategies}

We say that non\=/trivial compression discs $D_1$ for $S_1$ and $D_2$ for $S_2$ are \emph{canonical} if every self\=/homeomorphism of $\sphere[3]$ sending $S_1$ to $S_2$ orientation\=/preservingly sends $D_1$ to $D_2$, up to isotopy preserving $S_2$. As anticipated, we will often rely on compressing the surfaces along canonical compression discs in order to reduce the isotopy question to the more manageable problem for tori in $\sphere[3]$. Let us begin our discussion with a definition.

\begin{definition}
Let $X_1$ and $X_2$ be $3$\=/manifolds, and let $p_1\subs\boundary X_1$ and $p_2\subs\boundary X_2$ be simple closed curves. Let $X_1'$ and $X_2'$ be the $3$\=/manifolds obtained by attaching a $2$\=/handle, respectively, to $X_1$ along $p_1$ and to $X_2$ along $p_2$. Define the \emph{handle extension map}
\[
\map{\hext{p_1,p_2}}{\homeo{(X_1,p_1);(X_2,p_2)}}{\homeo{X_1';X_2'}}
\]
by extending homeomorphisms $\umap{(X_1,p_1)}{(X_2,p_2)}$ to the $2$\=/handles.
\end{definition}

It is easy to see that this map is well\=/defined, in that its output does not depend on the choice of representative for the isotopy class nor on the specific extension to the $2$\=/handles. Moreover, the handle extension map is functorial in the following sense: for $3$\=/manifolds $X_1$, $X_2$, $X_3$, curves $p_1\subs\boundary X_1$, $p_2\subs\boundary X_2$, $p_3\subs\boundary X_3$, and homeomorphisms
\begin{align*}
\map{f}{(X_1,p_1)}{(X_2,p_2)},&&\map{g}{(X_2,p_2)}{(X_3,p_3)},
\end{align*}
we have the equality
\[
\hext{p_1,p_3}(gf)=\hext{p_2,p_3}(g)\circ\hext{p_1,p_2}(f).
\]

Suppose now that we can find canonical compression discs $D_1$ and $D_2$ for $S_1$ and $S_2$ respectively. Up to flipping the orientations of $S_1$ and $S_2$, we can assume that $D_1$ and $D_2$ lie in $M_1$ and $M_2$ respectively. Let $p_1\subs S_1$ and $p_2\subs S_2$ be the boundary curves of $D_1$ and $D_2$ respectively. Let $P_1=M_1\cut D_1$ (which may be a disconnected $3$\=/manifold) and $Q_1=\closure{\sphere[3]\setminus P_1}$. Note that $Q_1$ can also be obtained by attaching the $2$\=/handle $\closure{Q_1\setminus N_1}$ to $N_1$ along the curve $p_1$. The $3$\=/manifolds $P_1$ and $Q_1$ have the same boundary $T_1$, which is either a torus or the union of two tori. Similarly, define $P_2=M_2\cut D_2$, $Q_2=\closure{\sphere[3]\setminus P_2}$, and $T_2=\boundary P_2=\boundary Q_2$.

\begin{proposition}\label{thm:general strategy}
In the situation described in the previous paragraph, the oriented surfaces $S_1$ and $S_2$ are isotopic if and only if there is a homeomorphism $\map{f}{(N_1,p_1)}{(N_2,p_2)}$ such that $\trace{\hext{p_1,p_2}(f)}$, seen as a map from $T_1$ to $T_2$, extends to a homeomorphism $\umap{P_1}{P_2}$.
\end{proposition}
\begin{proof}
If $S_1$ and $S_2$ are isotopic, let $\map{f}{\sphere[3]}{\sphere[3]}$ be a homeomorphism sending $S_1$ to $S_2$ orientation\=/preservingly. Since $D_1$ and $D_2$ are canonical, we can isotope $f$ so that it sends $D_1$ to $D_2$ (and, hence, $p_1$ to $p_2$). Then $\map{f|_{N_1}}{(N_1,p_1)}{(N_2,p_2)}$ is a homeomorphism such that $\trace{\hext{p_1,p_2}(f|_{N_1})}$ extends to $\map{f|_{P_1}}{P_1}{P_2}$.

Conversely, let $\map{f}{(N_1,p_1)}{(N_2,p_2)}$ and $\map{g}{P_1}{P_2}$ be homeomorphisms such that $\hext{p_1,p_2}(f)$ and $g$ have the same trace $\umap{T_1}{T_2}$. Then $\hext{p_1,p_2}(f)$ and $g$ can be combined to construct a self\=/homeomorphism of $\sphere[3]$ sending $S_1$ to $S_2$ orientation\=/preservingly.
\end{proof}

\Cref{thm:general strategy} will occasionally be useful on its own. However, for most of our applications, we will refer to the following result. The proof exploits the functoriality of the handle extension map -- together with the output of \cref{thm:mapping class group of 3-manifold} -- to translate the isotopy problem into a group\=/theoretic question which, while unsolvable in general, can be easily answered in the special cases we need.

\begin{proposition}\label{thm:general strategy for knot complements}
With the same notations as above, we can algorithmically decide whether the oriented surfaces $S_1$ and $S_2$ are isotopic provided that:
\begin{enumroman}
\item the $3$\=/manifold pair $(N_1,\boundary N_1\cut p_1)$ is irreducible;
\item $P_1$ is either a (possibly trivial) knot complement, or the union of two non\=/trivial knot complements.
\end{enumroman}
\end{proposition}
\begin{proof}
Since $(N_1,\boundary N_1\cut p_1)$ is irreducible, we can decide whether $(N_1,p_1)$ is homeomorphic to $(N_2,p_2)$ or not. If it is not, then clearly $S_1$ and $S_2$ are not isotopic. Otherwise, let $\map{f_0}{(N_1,p_1)}{(N_2,p_2)}$ be a homeomorphism. Similarly, since $P_1$ is a union of possibly trivial knot complements, we can decide if $P_1$ and $P_2$ are homeomorphic. If they are not, then once again $S_1$ and $S_2$ are not isotopic. Otherwise, let $\map{g_0}{P_1}{P_2}$ be a homeomorphism. By \cref{thm:general strategy}, we can solve the isotopy problem if we can answer the following question: is there a homeomorphism $f\in\homeo{N_2,p_2}$ such that
\[
\trace{\hext{p_2,p_2}(f)}\circ(\trace{\hext{p_1,p_2}(f_0)}\circ\trace{g_0}^{-1})\in\trhomeo{P_2}?
\]
By applying \cref{thm:mapping class group of 3-manifold}, we can further reduce the question to the following: given homeomorphisms $f_1,\ldots,f_n\in\trhomeo{Q_2}$ which are products of powers of Dehn twists of $T_2$, and $h\in\homeo{T_2}$, does there exist $f\in\langle f_1,\ldots, f_n\rangle$ such that $fh\in\trhomeo{P_2}$?

\step{When $P_2$ and $Q_2$ are solid tori.} Let $\ell,m\subs T_2$ be an oriented longitude and an oriented meridian of $P_2$ respectively; denote by $[\ell]$ and $[m]$ the corresponding homology classes in $H_1(T_2)$. The Dehn twists of $T_2$ which extend to $Q_2$ generate a subgroup of $\homeo{T_2}$ which is isomorphic to $\ZZ$ -- namely, the homeomorphisms that fix $[\ell]$. We can therefore compute an integer $k_0$ such that\footnote{More explicitly, if $f_i$ sends $[m]$ to $k_i[\ell]+[m]$ for $1\le i\le n$, then $k_0=\gcd(k_1,\ldots,k_n)$. In fact, since the homeomorphisms $f_1,\ldots,f_n$ are actually products of two Dehn twists about disjoint curves, the integers $k_i$ will always be equal to $-1$, $0$, or $1$.} the group $\langle f_1,\ldots,f_n\rangle$ is generated by the self\=/homeomorphism of $T_2$ sending $[\ell]$ to $[\ell]$ and $[m]$ to $k_0[\ell]+[m]$. It is immediate to check that the answer to the isotopy question is positive if and only if $h$ sends $[m]$ to $k[\ell]\pm[m]$ where $k$ is a multiple of $k_0$.

\step{When $P_2$ is a solid torus but $Q_2$ is not.} In this case, the group $\langle f_1,\ldots, f_n\rangle$ is actually trivial, since $\trhomeo{Q_2}\le\langle -\id\rangle$ and Dehn twists cannot act as $(-\id)$ on $H_1(T_2)$. Therefore, it is sufficient to check whether $h\in\trhomeo{P_2}$ or not.

\step{When $P_2$ is a non\=/trivial knot complement.} By \cref{thm:homeomorphisms of knot complement}, the group $\trhomeo{P_2}$ is finite and can be computed. Therefore, it is sufficient to be able to decide if $h\in\langle f_1,\ldots, f_n\rangle$. Note that, in this case, $Q_2$ is a solid torus. As argued above, the homeomorphisms $f_1,\ldots, f_n$ belong to the subgroup of $\homeo{T_2}$ fixing the homology class of the meridian of $Q_2$. This group is isomorphic to $\ZZ$, and we can compute a generator of $\langle f_1,\ldots, f_n\rangle$. By studying the action of $h$ on $H_1(T_2)$, we can easily decide whether $h$ belongs to $\langle f_1,\ldots, f_n\rangle$ or not.

\step{When $P_2$ is the union of two non\=/trivial knot complements.} Like before, it is enough to be able to decide if $h\in\langle f_1,\ldots,f_n\rangle$. In this case, the $3$\=/manifold $Q_2$ is homeomorphic to the complement of the $2$\=/component unlink in $\sphere[3]$. Denote by $T$ and $T'$ the two components of $T_2$. Let $m\subs T$ and $m'\subs T'$ the two unique curves which are homologically trivial in $Q_2$ but not in $T_2$, and fix an orientation for each of them. Clearly, every self\=/homeomorphism of $Q_2$ preserves $m\cup m'$ as a set up to isotopy. Moreover, Dehn twists of $T_2$ cannot swap $T$ and $T'$ or invert the orientation of $m$ or $m'$. Therefore, the homeomorphisms $f_1,\ldots, f_n$ belong to the subgroup of $\homeo{T_2}$ fixing the homology classes $[m]$ and $[m']$ in $H_1(T_2)$. This group is isomorphic to $\ZZ\times\ZZ$; hence, by studying the action of $h$ on $H_1(T_2)$, we can decide whether $h$ belongs to $\langle f_1,\ldots,f_n\rangle$ or not.
\end{proof}

\subsection{Compression discs for the boundary}\label{sec:classification algorithm:compression discs}

The reader should by now reasonably believe that canonical compression discs will play a crucial role in our classification algorithm. A good strategy to prove that two discs -- say -- in $M_1$ and $M_2$ are canonical consists in characterising them as the unique discs satisfying some property which only depends on the intrinsic topologies of $M_1$ and $M_2$. The results in this section will provide us with useful tools to prove such a characterisation.

Let $(M,R)$ be a $3$\=/manifold pair. Consider a property $\PPP$ for discs properly embedded in $M$ which is invariant under isotopies in $(M,R)$, such as ``being separating''. We say that $\PPP$ is \emph{stable under boundary compressions} if
\begin{itemize}
\item for every disc $D$ properly embedded in $(M,R)$, and
\item for every possibly trivial boundary compression disc $E\subs M$ for $D$ with $\boundary E\subs D\cup R$,
\end{itemize}
we have that at least one of the two discs obtained by boundary compressing $D$ along $E$ satisfies $\PPP$. The crucial fact is that, (very) loosely speaking, discs satisfying a property which is stable under boundary compressions can be made disjoint. More precisely, we have the following.

\begin{proposition}\label{thm:discs with boundary compression-stable property}
Let $M$ be an irreducible $3$\=/manifold, and let $R\subs\boundary M$ be a surface. Let $F$ be a disc properly embedded in $(M,R)$. Suppose that there is a disc properly embedded in $(M,R)$ not isotopic to $F$ in $M$ and satisfying some property $\PPP$ which is stable under boundary compressions. Then there is a disc properly embedded in $(M,R)$ which is not isotopic to $F$ in $M$, is disjoint from $F$, and satisfies $\PPP$.
\end{proposition}
\begin{proof}
Among all discs properly embedded in $(M,R)$ which are in general position with respect to $F$, are not isotopic to $F$, and satisfy $\PPP$, pick $D$ to minimise the number of components of $F\cap D$; we will show that $D$ is in fact disjoint from $F$. Since $M$ is irreducible, a standard innermost circle argument shows that $F\cap D$ is a collection of arcs. Suppose that this intersection is non\=/empty, and let $a\subs F\cap D$ be an outermost arc in $F$. The arc $a$ cuts a disc $E$ off of $F$, such that $E$ is a (possibly trivial) boundary compression disc for $D$ and $\boundary E\subs D\cup R$. Let $D_1$ and $D_2$ be the two discs obtained by boundary compressing $D$ along $E$, with $D_1$ satisfying $\PPP$.

If $D_1$ is isotopic to $F$, then $D$ can in fact be isotoped to be disjoint from $F$, contradicting our minimality assumption. Otherwise, $D_1$ is a disc satisfying $\PPP$ which is not isotopic to $F$, and its intersection with $F$ has strictly fewer components than $F\cap D$. Again, this contradicts the minimality of $D$. It follows that $D$ must be disjoint from $F$, as required.
\end{proof}

The next proposition shows that some of the properties we care about are, in fact, stable under boundary compressions.

\begin{proposition}\label{thm:boundary compression-stable properties}
Let $M$ be an irreducible $3$\=/manifold, and let $R\subs\boundary M$ be a surface. 
The following properties for discs properly embedded in $(M,R)$ are stable under boundary compressions:
\begin{enumarabic}
\item being non\=/separating;
\item having boundary which is a non\=/trivial curve in $R$.
\end{enumarabic}
\end{proposition}
\begin{proof}
Let $D$ be a disc properly embedded in $(M,R)$, and let $E\subs M$ be a possibly trivial boundary compression disc for $D$ with $\boundary E\subs D\cup\interior{R}$. Denote by $D_1$ and $D_2$ the two discs obtained by boundary compressing $D$ along $E$.
\begin{enumarabic}
\item The disc $D$ is non\=/separating if and only if $[D]\in H^1(M)$ is non\=/trivial. But $[D]=[D_1]+[D_2]\in H^1(M)$, therefore if $D$ is non\=/separating then one of $D_1$ and $D_2$ must be as well.
\item Suppose that $D_1$ cobounds a ball $B\subs M$ with some disc in $R$. If $B$ is disjoint from $E$, then $D$ is actually isotopic to $D_2$, so $\boundary D_2$ is trivial in $R$ if and only if $\boundary D$ is. If instead $B$ contains $E$ then it also contains $D_2$, and it is easy to see that $\boundary D$ is in trivial in $R$. \qedhere
\end{enumarabic}
\end{proof}

This is a good time to remark that, when we say that a disc properly embedded in a $3$\=/manifold is separating, we mean that it splits the $3$\=/manifold itself, and not only the boundary, in two connected components. However, as the following proposition shows, there is no difference for $3$\=/manifolds which can be embedded in $\sphere[3]$. Since all the $3$\=/manifolds we will work with have this property -- and, usually, already lie inside $\sphere[3]$ -- we will freely use the term ``separating'' to denote discs which separate the $3$\=/manifold they lie in and its boundary.

\begin{proposition}
Let $M$ be a $3$\=/manifold embedded in $\sphere[3]$ with connected boundary, and let $D\subs M$ be a properly embedded disc. Then $D$ separates $M$ if and only if $\boundary D$ separates $\boundary M$.
\end{proposition}
\begin{proof}
Clearly if $D$ separates $M$ then $\boundary D$ separates $\boundary M$. Conversely, suppose that $\boundary D$ separates $\boundary M$ but $D$ does not separate $M$. Let $P=M\cut D$ be the result of cutting $M$ along $D$; by assumption, we have that $P$ is connected, but $\boundary P$ is the union of two components. We can interpret $\closure{M\setminus P}$ as a $1$\=/handle which, when attached to $P$, yields the original $3$\=/manifold $M$. Let $a\subs M$ be the co\=/core arc of this $1$\=/handle, and let $b$ be an arc properly embedded in $P$ connecting the two endpoints of $a$. Then $a\cup b$ is a closed curve in $M$ which intersects each boundary component of $P$ exactly once. This provides a contradiction, since $M$ is embedded in $\sphere[3]$.
\end{proof}

\subsection{One \texorpdfstring{non\=/separating}{non-separating} compression disc}\label{sec:classification algorithm:one non-separating compression disc}

We finally start presenting the actual classification algorithm. As anticipated, this section addresses the case where one side of $S_1$ -- say $M_1$ -- has a properly embedded non\=/separating disc $D_1$ but is not a handlebody; see \cref{fig:non-separating compression disc example} for an example. Define $p_1$, $P_1$, $Q_1$, and $T_1$ as described in \cref{sec:classification algorithm:general strategies} (even though we don't know that $D_1$ belongs to a pair of canonical discs yet). The torus $T_1$ bounds a solid torus in $\sphere[3]$. But $P_1$ cannot be a solid torus, for otherwise $M_1$ would be a handlebody. Hence, the solid torus must be the other component of $\sphere[3]\cut T_1$, namely $Q_1$. In other words, $P_1$ is the complement of a non\=/trivial knot and, as such, is boundary irreducible. As described in \cref{fig:non-separating compression disc M1 N1}, the $3$\=/manifold $M_1$ is obtained by attaching a $1$\=/handle to $P_1$, while drilling an arc from the solid torus $Q_1$ yields $N_1$. The following lemma applied to $(P_1,\boundary P_1)$ implies, \emph{a posteriori}, that no arbitrary choice was made when selecting the disc $D_1$. We remark that the lemma is stated in greater generality then we actually need here, but we will need this more general version later.

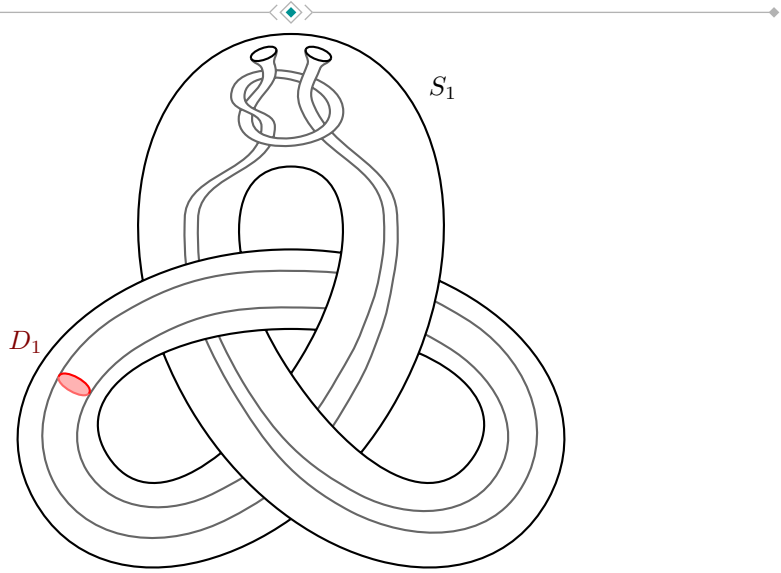
\begin{figure}
\centering
\tikzsetnextfilename{non-separating-compression-disc-example}
\begin{tikzpicture}[use Hobby shortcut,thick]
\path[scale=1.67,spath/save=trefoil] ([closed]90:2) foreach \k in {1,...,3} { .. (-30+\k*240:.5) .. (90+\k*240:2) } (90:2);
\tikzset{ks/double={closed,Hobby,path=trefoil,to=thick trefoil,width={50+80*\t*(\t-1)}}}
\tikzset{ks/extract components={draft mode=false,draft mode scale=.3,path=thick trefoil,to={thick trefoil transverse}{3,7,11,17,21,25},not={thick trefoil}{5,19,13,15,9,23}}}
\pic at (spath cs:trefoil 0) {code={\begin{scope}[yshift=-2.35cm,scale=1,rotate=90]\path[spath/save global=curve endpoint 1] (2.8,.3) to[out=200,in=0,out looseness=2,in looseness=.6] (2.1,.6) to[out=180,in=180,looseness=.6] (2.2,-.6) to[out=0,in=0,in looseness=.5] (2.4,.7) to[out=180,in=0] (2.0,.3); \path[spath/save global=curve endpoint 2] (2.3,-.2) to[out=20,in=160,out looseness=1.8,in looseness=1](2.8,-.3);\end{scope}}};
\path[spath/save=curve] (60:1) arc(60:120-360:1);
\tikzset{ks/subdivide={path=curve,to=curve,n=3},ks/path along={path=curve,to=curve,along=trefoil,width to=30pt}}
\path[spath/save=curve][spath/use=curve endpoint 1] to[out=-90,in=90,out looseness=.5,in looseness=1.5] (spath cs:curve 0)  [spath/use=curve] to[out=90,in=-70,out looseness=1.5,in looseness=.5] (spath cs:{curve endpoint 2} 0) [spath/use={curve endpoint 2,weld}];
\tikzset{spath/remove empty components=curve,spath/spot weld=curve}
\tikzset{ks/double={path=curve,to=thick curve,width={5+max(0,-(\t-.3)*(\t-.9)*100)}}}
\tikzset{ks/extract components={draft mode=false,draft mode scale=.4,path=thick curve,split with self=false,split with={thick trefoil transverse},to={thick curve}{3,10,5,12},to={thick curve b}{1,8},to={thick curve e}{7,14}}}
\tikzset{ks/extract components={draft mode=false,draft mode scale=.2,path=thick curve b,to={thick curve b}{1,2,3,14,15,16},to={thick curve bxx}{9,10,11,13,22,23,24,26},to={thick curve bx}{5,6,7,18,19,20}}}
\tikzset{ks/extract components={draft mode=false,draft mode scale=.2,path=thick curve e,split with self=false,split with={thick curve bx},not={thick curve ex}{2,7}}}
\tikzset{ks/extract components={draft mode=false,draft mode scale=.2,path=thick curve bx,split with self=false,split with={thick curve e},not={thick curve bx}{4,9}}}
\path[spath/save=thick curve][spath/use/.list={thick curve b,thick curve bx,thick curve bxx,thick curve,thick curve ex}];
\draw[black!60][spath/use=thick curve];
\path[spath/use=curve] \foreach \pos/\xs in {0/1,0.9999/-1} {pic[sloped,xscale=\xs,draw=black!60,pos=\pos] {ks/tube end={5pt}{arcs/.style={black}}}};
\draw[black][spath/use=thick trefoil];
\path[spath/use=curve] pic[sloped,main color 1,pos=.6,yshift=.5pt]{ks/disc={13pt}{disc/.style={fill=main color 1!30},arc/.style={main color 1!60}}} node[pos=.6,above left=12pt,main color 1!50!black] {$D_1$};
\node at (2,3.5) {$S_1$};
\end{tikzpicture}
\caption{The surface $S_1$ admits a non\=/separating compression disc $D_1$.\label{fig:non-separating compression disc example}}
\end{figure}

\begin{figure}
\centering
\begin{subcaptionblock}{.45\linewidth}
\centering
\tikzsetnextfilename{non-separating-compression-disc-M1}
\begin{tikzpicture}[thick]
\draw[black,fill=main color 2!30] circle(2cm);
\node[main color 2!50!black] at (-2.1,1.5) {$M_1$};
\path[spath/save=hole] (0,1.2) to[out=-90,in=90,out looseness=2] (-.75,-.4) to[out=-90,in=-90,out looseness=1.2,in looseness=2] (.75,0);
\path[spath/save=hole] [spath/use=hole,spath/transform={hole}{yscale=-1},spath/append reverse=hole];
\tikzset{ks/subdivide={path=hole,to=hole,n=2},ks/double={path=hole,to=thick hole,width=10,to contour=hole contour}}
\fill[main color 2!15] [spath/use=hole contour];
\tikzset{ks/extract components={draft mode=false,draft mode scale=.4,path=thick hole,not={thick hole}{8,21,12,25,4,17}}}
\path[spath/use=hole] \foreach \pos/\xs/\as in {0/1/{black},0.9999/-1/{black!60}} {pic[sloped,xscale=\xs,draw=black!60,pos=\pos] {ks/tube end={10pt}{arcs/.style={\as},disc/.style={fill=main color 2!15},tube/.style={fill=main color 2!15}}}};
\draw[black!60][spath/use=thick hole];
\path[spath/save=handle] (-1,1.5) to[bend left=120,looseness=2.5] (1,1.5);
\tikzset{ks/subdivide={path=handle,to=handle,n=5},ks/double={path=handle,to=thick handle,to contour=handle contour,width={5+28*\t*(1-\t)}}}
\fill[main color 2!30][spath/use=handle contour];
\path[spath/use=handle] \foreach \pos/\xs in {0/-1,0.9999/1} {pic[sloped,xscale=\xs,draw=black,pos=\pos] {ks/tube end={5pt}{tube above,arc 1/.style={black!60}}}};
\draw[black][spath/use=thick handle];
\path[spath/use=handle] pic[sloped,pos=.5,yshift=.25,main color 1] {ks/disc={11.75pt}{arc/.style={main color 1!60},disc/.style={fill=main color 1!30}}} node[main color 1!50!black,pos=.5,above=6pt] {$D_1$};
\end{tikzpicture}
\caption{\label{fig:non-separating compression disc M1}}
\end{subcaptionblock}
\begin{subcaptionblock}{.45\linewidth}
\centering
\tikzsetnextfilename{non-separating-compression-disc-N1}
\begin{tikzpicture}[thick]
\draw[black,fill=main color 3!30,even odd rule] circle (2cm) circle (.5cm);
\node[main color 3!50!black] at (-2.1,1.5) {$N_1$};
\pic at (0,1.25) {code={\begin{scope}[yshift=-2.4cm,scale=1,rotate=90]\path[spath/save global=curve endpoint 1] (2.8,.3) to[out=200,in=0,out looseness=2,in looseness=.6] (2.1,.6) to[out=180,in=180,looseness=.6] (2.2,-.6) to[out=0,in=0,in looseness=.5] (2.4,.7) to[out=180,in=0] (2.0,.3); \path[spath/save global=curve endpoint 2] (2.3,-.2) to[out=20,in=160,out looseness=1.8,in looseness=1](2.8,-.3);\end{scope}}};
\path[spath/save=curve] [spath/use=curve endpoint 1] to[out=-90,in=120,in looseness=2] (-150:1.25) arc(-150:15:1.25) to[out=105,in=-70] (spath cs:curve endpoint 2 0) [spath/use={weld,curve endpoint 2}];
\tikzset{ks/double={path=curve,to=thick curve,to contour=curve contour,width={5+max(0,-(\t-.5)*(\t-.8)*600)}}}
\fill[main color 3!15] [spath/use=curve contour];
\tikzset{ks/extract components={draft mode=false,draft mode scale=.2,path=thick curve,not={thick curve}{12,33,16,37,18,39,4,25,10,31}}}
\path[spath/use=curve] \foreach \pos/\xs in {0/1,0.9999/-1} {pic[sloped,xscale=\xs,draw=black!60,pos=\pos] {ks/tube end={5pt}{arcs/.style={black},tube/.style={fill=main color 3!15},disc/.style={fill=main color 3!15}}}};
\draw[black!60][spath/use=thick curve];
\path[spath/use=curve] pic[draw=main color 1,sloped,pos=.65] {ks/disc={16.5pt}{arc/.style={main color 1!60}}} node[main color 1!50!black,pos=.65,above right=-1 pt and 4 pt] {\tikzcontour{$p_1$}};
\end{tikzpicture}
\caption{\label{fig:non-separating compression disc N1}}
\end{subcaptionblock}
\caption{The $3$\=/manifold $M_1$ \subref{fig:non-separating compression disc M1} is homeomorphic to a knot complement $P_1$ with a $1$\=/handle attached. The $3$\=/manifold $N_1$ \subref{fig:non-separating compression disc N1} is homeomorphic to a solid torus $Q_1$ with an arc drilled out.\label{fig:non-separating compression disc M1 N1}}
\end{figure}

\begin{lemma}\label{thm:unique non-separating compression disc}
Let $(P,R)$ be an irreducible $3$\=/manifold pair. Let $M$ be the result of attaching a $1$\=/handle $H$ to $P$ along two discs in $\interior{R}$, and define $R'=\boundary M\setminus(\boundary P\setminus R)$. Then the co\=/core of $H$ is the unique compression disc for $R'$ which does not separate $M$.
\end{lemma}
\begin{proof}
Let $F$ be the co\=/core of $H$, and suppose there is another non\=/separating compression disc $D$ for $R'$ which is not isotopic to $F$. By \cref{thm:discs with boundary compression-stable property,thm:boundary compression-stable properties}, we may assume that $D$ is disjoint from $F$ and, therefore, contained in $P$. Since the pair $(P,R)$ is irreducible, $D$ must cobound a ball in $P$ with some disc $D'\subs R$. But $D$ is non\=/separating in $M$, and this is only possible if $D'$ contains exactly one of the attaching discs of $H$. This would imply that $D$ is isotopic to $F$, contrary to our assumption.
\end{proof}

Now, if $M_2$ is a handlebody or $S_2$ has no non\=/separating compression discs in $M_2$ then clearly $S_1$ and $S_2$ are not isotopic. Otherwise, by imitating the above procedure for $S_2$ instead of $S_1$, we define $p_2$, $D_2$, $P_2$, $Q_2$, and $T_2$. Since $D_1$ (respectively $D_2$) can be intrinsically defined as the unique non\=/separating compression disc for $S_1$ (respectively $S_2$) in $M_1$ (respectively $M_2$), every homeomorphism $\umap{M_1}{M_2}$ must send $D_1$ to a disc isotopic to $D_2$. If the $3$\=/manifold pair $(N_1,\boundary N_1\cut p_1)$ is irreducible, we immediately conclude by applying \cref{thm:general strategy for knot complements}; note that this is the case for the surface depicted in \cref{fig:non-separating compression disc example}, since the pair $(N_1,\boundary N_1\cut p_1)$ shown in \cref{fig:non-separating compression disc N1} is irreducible.

Otherwise, the surface $\boundary N_1\cut p_1$ is compressible in $N_1$, and we can find a compression disc $E_0\subs N_1$ for it. The disc $E_0$ is properly embedded in the solid torus $Q_1$; if it is a meridian disc of $Q_1$, then let us define $E_1=E_0$. Otherwise, $E_0$ must cut a ball off of $Q_1$; this ball necessarily contains the curve $p_1$, for otherwise $E_0$ would be a trivial compression disc for $\boundary N_1\cut p_1$. We can therefore find a meridian disc $E_1$ of $Q_1$ which is disjoint from $B$ (and, hence, from $p_1$).

Either way, we have found a disc $E_1$ properly embedded in $(N_1,\boundary N_1\cut p_1)$ which is non\=/separating in $Q_1$. Similarly, let $E_2$ be a disc properly embedded in $(N_2,\boundary N_2\cut p_2)$ which is non\=/separating in $Q_2$; see the top left of \cref{fig:non-separating compression disc h} for an example. We claim that $S_1$ and $S_2$ are isotopic if and only if $P_1$ is homeomorphic to $P_2$ and $(N_1,p_1)$ is homeomorphic to $(N_2,p_2)$. The forward implication is trivial. Conversely, suppose that the $3$\=/manifolds with boundary pattern $(N_1,p_1)$ and $(N_2,p_2)$ are homeomorphic. Note that $\boundary E_2$ does not separate $\boundary N_2\cut p_2$; hence, as shown in \cref{fig:non-separating compression disc h}\todo{Remake picture with \texttt{torus} TikZ library.}, it is easy to construct a self\=/homeomorphism $h$ of $(N_2,p_2)$ which maps $E_2$ to itself with the opposite orientation, and is the identity in a neighbourhood of $p_2$. Then the Dehn twists about $E_2$, together with $\hext{p_2,p_2}(h)$, generate the whole mapping class group of the solid torus $Q_2$. As a consequence, the map
\[
\map{\hext{p_1,p_2}}{\homeo{(N_1,p_1);(N_2,p_2)}}{\homeo{Q_1;Q_2}}
\]
is surjective. If, moreover, the knot complements $P_1$ and $P_2$ are homeomorphic, then every homeomorphism $\umap{P_1}{P_2}$ extends to a homeomorphism $\umap{Q_1}{Q_2}$ which, in turn, is the image under $\hext{p_1,p_2}$ of some homeomorphism $\umap{(N_1,p_1)}{(N_2,p_2)}$. By \cref{thm:general strategy}, this shows that $S_1$ and $S_2$ are isotopic.

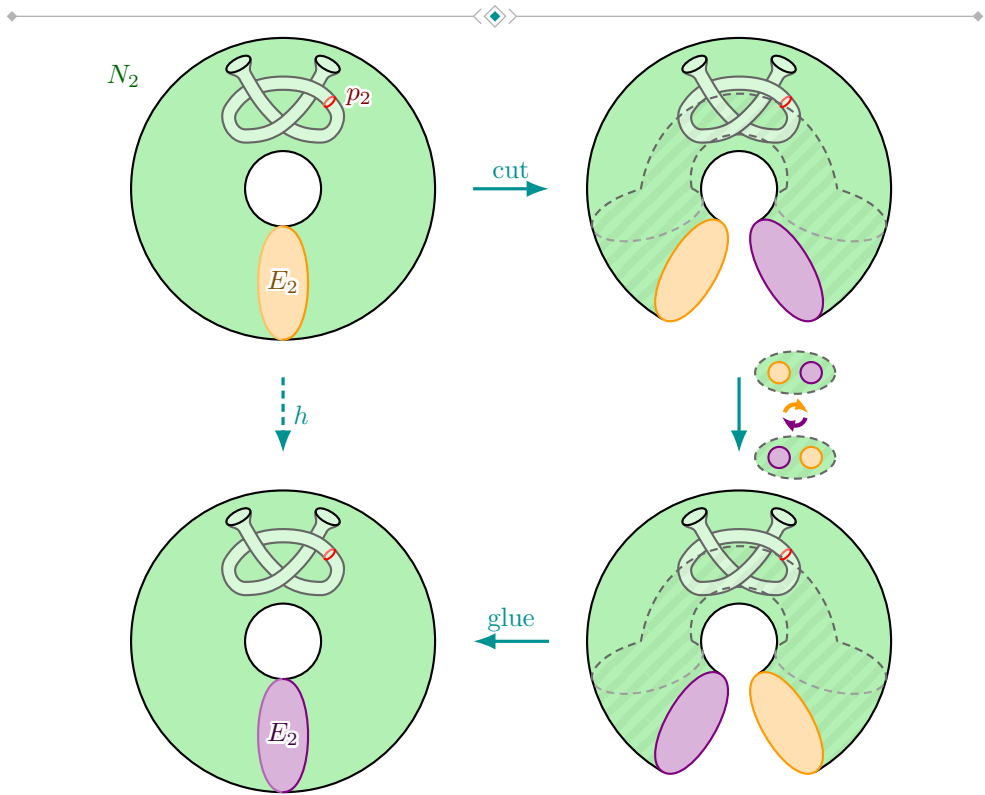
\begin{figure}
\centering
\tikzsetnextfilename{non-separating-compression-disc-h}
\begin{tikzpicture}[thick]
\path[spath/save=curve] (-.5,1.5) to[out=-60,in=180] (.5,.7) to[out=0,in=0,out looseness=1.5,in looseness=2] (0,1.4);
\path[spath/save=curve] [spath/use=curve,spath/transform={curve}{xscale=-1},spath/append reverse=curve];
\tikzset{ks/subdivide={path=curve,to=curve,n=2},ks/double={path=curve,to=thick curve,to contour=curve contour,width=5pt}}
\tikzset{ks/extract components={draft mode=false,draft mode scale=.2,path=thick curve,not={thick curve}{8,21,12,25,4,17}}}
\def\commonpic#1{
	\draw[black,fill=main color 3!30,even odd rule] circle (2cm) circle (.5cm);
	\fill[main color 3!15][spath/transform={curve contour}{shift={#1}},spath/use=curve contour];
	\path[spath/transform={curve}{shift={#1}},spath/use=curve] \foreach \pos/\xs in {0/1,0.9999/-1} {pic[sloped,xscale=\xs,draw=black!60,pos=\pos] {ks/tube end={5pt}{arcs/.style={black},tube/.style={fill=main color 3!15},disc/.style={fill=main color 3!15}}}};
	\draw[spath/transform={thick curve}{shift={#1}},black!60][spath/use=thick curve];
	\path[spath/transform={curve}{shift={#1}},spath/use=curve] pic[sloped,draw=main color 1,pos=.4] {ks/disc={5pt}{arc/.style={draw=main color 1!60}}};
}
\def\cutpic#1#2{
	\foreach \a/\xs/\name in {-15/.44/r,195/-.44/l} \pic[rotate={90+\a},xscale=\xs,densely dashed] at (\a:1.25) {code={\draw[black!60,spath/save global=\name-a] (60:.75) coordinate (\name-a) arc (60:90:.75);\draw[black!36] (90:.75) arc (90:270:.75); \draw[black!60,spath/save global=\name-b] (-90:.75) arc (-90:0:.75) coordinate (\name-b);}};
	\path[pattern={Lines[angle=45,line width=3pt,distance=6pt]},opacity=.2,even odd rule][spath/use=l-a] +(0,0) arc(-165:-15:.5) [spath/use={r-a,weld,reverse}] to[bend right=100,looseness=2] (spath cs:l-a 0) [spath/use={l-b,reverse}] +(0,0) arc(-165:-15:2) [spath/use={r-b,weld}] to[bend right=80,looseness=1.7] (spath cs:l-b 1);
	\fill[white] (0,0) -- (-60:3) -- (-120:3) -- cycle;
	\pic[rotate=30,draw=#1] at (-60:1.25) {ks/disc={1.5cm}{disc/.style={fill=#1!30}}};
	\pic[rotate=-30,draw=#2] at (-120:1.25) {ks/disc={1.5cm}{disc/.style={fill=#2!30}}};
	\draw[black!60,densely dashed] (l-a) to[bend left=100,looseness=2] (r-a) (l-b) to[bend left=80,looseness=1.7] (r-b);
}
\begin{scope}[shift={(0,0)}]
\commonpic{(0,0)}
\node[main color 3!50!black] at (-2.1,1.5) {$N_2$};
\pic[draw=main color 4] at (0,-1.25) {ks/disc={1.5cm}{disc/.style={fill=main color 4!30},arc/.style={main color 4!60}}};
\node[main color 4!50!black] at (0,-1.25) {\tikzcontour{$E_2$}};
\node[main color 1!50!black] at (1,1.2) {\tikzcontour{$p_2$}};
\end{scope}
\begin{scope}[shift={(6,0)}]
\commonpic{(6,0)}
\cutpic{main color 5}{main color 4}
\end{scope}
\begin{scope}[shift={(6,-6)}]
\commonpic{(6,-6)}
\cutpic{main color 4}{main color 5}
\end{scope}
\begin{scope}[shift={(0,-6)}]
\commonpic{(0,-6)}
\pic[draw=main color 5] at (0,-1.25) {ks/disc={1.5cm}{disc/.style={fill=main color 5!30},arc/.style={main color 5!60}}};
\node[main color 5!50!black] at (0,-1.25) {\tikzcontour{$E_2$}};
\end{scope}
\tikzset{arrow/.style={->,very thick,theme color}}
\draw[arrow] (2.5,0) -- (3.5,0) node[midway,above] {cut};
\draw[arrow] (6,-2.5) -- (6,-3.5) pic[thick,midway,shift={(21pt,0)}] {code={
	\draw[ultra thick,-{Latex[length=6pt,bend]},main color 4] (165:4pt) arc(165:0:4pt);
	\draw[ultra thick,-{Latex[length=6pt,bend]},main color 5] (-15:4pt) arc(-15:-180:4pt);
	\foreach \ys/\coll/\colr in {1/main color 4/main color 5,-1/main color 5/main color 4} {
	\begin{scope}[yshift=16pt*\ys]
	\draw[black!60,densely dashed,preaction={fill=main color 3!30},preaction={solid,opacity=.2,black,pattern={Lines[angle=45,line width=3pt,distance=6pt]}}] circle(15pt and 8pt);
	\draw[\coll,fill=\coll!30] (-6pt,0) circle (4pt);
	\draw[\colr,fill=\colr!30] (6pt,0) circle (4pt);
	\end{scope}
	}
}};
\draw[arrow] (3.5,-6) -- (2.5,-6) node[midway,above] {glue};
\draw[arrow,densely dashed] (0,-2.5) -- (0,-3.5) node[midway,right] {$h$};
\end{tikzpicture}
\caption{There exists a self\=/homeomorphism $h$ of $(N_2,p_2)$ which flips the disc $E_2$ and restricts to the identity in a neighbourhood of $p_2$. This homeomorphism can be constructed by cutting $N_2$ along $E_2$, swapping the two sides of $E_2$ by sliding in the shaded region of $\boundary N_2$, and finally gluing the two sides back together.\label{fig:non-separating compression disc h}}
\end{figure}

Finally, we explain how to algorithmically decide whether $(N_1,p_1)$ and $(N_2,p_2)$ are homeomorphic or not. Note that $N_1\cut E_1$ is a knot complement, since $E_1$ is a meridian disc for $Q_1$; obviously, the same holds for $N_2\cut E_2$.
\begin{substeps}
\item If $N_1$ and $N_2$ are not handlebodies, then $N_1\cut E_1$ and $N_2\cut E_2$ are non\=/trivial knot complements, and $p_1$ and $p_2$ are meridian curves of $N_1\cut E_1$ and $N_2\cut E_2$ respectively. \Cref{thm:unique non-separating compression disc} then implies that $E_1$ and $E_2$ are the unique non\=/separating compression discs for $\boundary N_1\cut p_1$ and $\boundary N_2\cut p_2$ in $N_1$ and $N_2$ respectively. As a consequence, every homeomorphism $\umap{(N_1,p_1)}{(N_2,p_2)}$ maps $E_1$ to $E_2$; we deduce that $(N_1,p_1)$ and $(N_2,p_2)$ are homeomorphic if and only if $N_1\cut E_1$ and $N_2\cut E_2$ are.
\item If $N_1$ and $N_2$ are handlebodies, we claim that $(N_1,p_1)$ and $(N_2,p_2)$ are always homeomorphic. Note that attaching a $2$\=/handle to $N_1\cut E_1$ along $p_1$ yields a $3$\=/ball. In other words, the curve $p_1$ is a longitude of the solid torus $N_1\cut E_1$. As a consequence, there exists a homeomorphism $\umap{(N_1\cut E_1,p_1)}{(N_2\cut E_2,p_2)}$, and every such homeomorphism extends to a homeomorphism $\umap{(N_1,p_1)}{(N_2,p_2)}$.
\end{substeps}

\subsection{One separating compression disc}\label{sec:classification algorithm:one separating compression disc}

We now deal with the case where one of the sides of $S_1$ -- again, say $M_1$ -- has compressible boundary, but all the compression discs are separating; an example of this situation is depicted in \cref{fig:separating compression disc example}. Let $D_1$ be a compression disc for $S_1$ in $M_1$. Define $p_1$, $P_1$, $Q_1$, and $T_1$ as described in \cref{sec:classification algorithm:general strategies}. In this case, the $3$\=/manifold $P_1$ is the union of two connected components $P_{1,1}$ and $P_{1,2}$, with boundary tori $T_{1,1}$ and $T_{1,2}$ respectively, so that $T_1=T_{1,1}\cup T_{1,2}$.

\begin{figure}
\centering
\tikzsetnextfilename{separating-compression-disc-example}
\begin{tikzpicture}[thick]
\path[spath/save=hole 1] (0,1.2) to[out=-90,in=90,out looseness=2] (-.75,-.4) to[out=-90,in=-90,out looseness=1.2,in looseness=2] (.75,0);
\path[spath/save=hole 1] [spath/use=hole 1,spath/transform={hole 1}{yscale=-1},spath/append reverse=hole 1];
\tikzset{ks/subdivide={path=hole 1,to=hole 1,n=2},ks/double={path=hole 1,to=thick hole 1,width=12}}
\tikzset{ks/extract components={draft mode=false,draft mode scale=.2,path=thick hole 1,not={thick hole 1}{8,21,12,25,4,17},to={hole transversals 1}{2,15,10,23,6,19}}}
\path[spath/use=hole 1] pic[sloped,pos=0,draw=black!60] {ks/tube end={12 pt}{disc/.style={spath/save global=tube end u 1},arcs/.style={draw=none}}} pic[sloped,pos=.999,xscale=-1,draw=black!60] {ks/tube end={12 pt}{disc/.style={spath/save global=tube end l 1},arcs/.style={draw=none}}};
\draw[black!60][spath/use=thick hole 1];
\begin{scope}[shift={(6,0)}]
\path[spath/save=hole 2] (0,1.2) to[out=-90,in=90,out looseness=2.2] (-.6,-.33) to[out=-90,in=-135] (.6,-.7) to[out=45,in=-15,out looseness=2] (0,0);
\path[spath/save global=hole 2] [spath/use=hole 2,spath/transform={hole 2}{xscale=-1,yscale=-1},spath/append reverse=hole 2];
\tikzset{ks/subdivide={path=hole 2,to=hole 2,n=2},ks/double={path=hole 2,to=thick hole 2,width=12}}
\tikzset{ks/extract components={draft mode=false,draft mode scale=.2,path=thick hole 2,not={thick hole 2}{12,29,16,33,4,21,8,25},to={hole transversals 2}{2,19,10,27,6,23,14,31}}}
\path[spath/save global=hole transversals 2,spath/use=hole transversals 2];
\path[spath/use=hole 2] pic[sloped,pos=0,draw=black!60] {ks/tube end={12 pt}{disc/.style={spath/save global=tube end u 2},arcs/.style={draw=none}}} pic[sloped,pos=.999,draw=black!60] {ks/tube end={12 pt}{disc/.style={spath/save global=tube end l 2},arcs/.style={draw=none}}};
\draw[black!60][spath/use=thick hole 2];
\end{scope}
\path[ks/subdivide={path=hole 2,to=hole 2,n=2},spath/save=curve] (45:1.6) to[out=30,in=90] (spath cs:{hole 2} 0) [spath/append=hole 2] to[out=-90,in=0,in looseness=.5] (3,-2.2) to[out=180,in=-90,out looseness=.5] (spath cs:hole 1 1) [spath/append={hole 1,reverse}] to[out=90,in=150] ($(6,0)+(135:1.6)$);
\tikzset{ks/double={path=curve,to=thick curve,width={and(\t>.42,\t<.6)?20:5}}}
\tikzset{ks/extract components={draft mode=false,draft mode scale=.2,path=thick curve,split with self=false,split with={hole transversals 1,hole transversals 2},to={thick curve}{3,18,5,20,7,22,11,26,13,28},to={thick curve ul}{1,16},to={thick curve ur}{15,30},to={thick curve d}{9,24}}}
\path[spath/save=spheres] (0,0) circle(2) (6,0) circle(2);
\tikzset{
ks/extract components={draft mode=false,draft mode scale=.2,path=thick curve d,split with self=false,split with={tube end l 1,tube end l 2,spheres},to={thick curve dx}{4,11},to={thick curve dxx}{1,2,8,9,6,7,13,14}},
ks/extract components={draft mode=false,draft mode scale=.2,path=tube end l 1,split with self=false,split with={thick curve d},not={tube end l 1}{2}},
ks/extract components={draft mode=false,draft mode scale=.2,path=tube end l 2,split with self=false,split with={thick curve d},not={tube end l 2}{2}},
ks/extract components={draft mode=false,draft mode scale=.2,path=thick curve ul,split with self=false,split with={tube end u 2},not={thick curve ul}{3,6},to={thick curve ulx}{3,6}},
ks/extract components={draft mode=false,draft mode scale=.2,path=thick curve ur,split with self=false,split with={tube end u 1,thick curve ul},not={thick curve ur}{1,6,4,9},to={thick curve urx}{1,6}},
ks/extract components={draft mode=false,draft mode scale=.2,path=tube end u 1,split with self=false,split with={thick curve ur},not={tube end u 1}{2}},
ks/extract components={draft mode=false,draft mode scale=.2,path=tube end u 2,split with self=false,split with={thick curve ul},not={tube end u 2}{3}},
spath/remove empty components=spheres,
ks/extract components={draft mode=false,draft mode scale=.2,path=spheres,split with self=false,split with={thick curve d,thick curve ul,thick curve ur},not={spheres}{3,5,9,11}},
}
\draw[black!60][spath/use/.list={thick curve ulx,thick curve urx,thick curve,thick curve dxx,tube end l 1,tube end l 2}];
\draw[black][spath/use/.list={thick curve dx,thick curve ul,thick curve ur,tube end u 1,tube end u 2,spheres}];
\path[spath/use=curve] \foreach \pos/\xs in {0/1,0.9999/-1} {pic[sloped,xscale=\xs,draw=black,pos=\pos] {ks/tube end={5pt}{tube above,arc 1/.style={black!60}}}};
\path[spath/use=curve] pic[sloped,main color 1,pos=.585,yshift=0pt]{ks/disc={20pt}{disc/.style={fill=main color 1!30},arc/.style={main color 1!60}}} node[pos=.585,above=12pt,main color 1!50!black] {$D_1$};
\node[black] at (-2.1,1.5) {$S_1$};
\end{tikzpicture}
\caption{The surface $S_1$ admits a separating compression disc $D_1$.\label{fig:separating compression disc example}}
\end{figure}
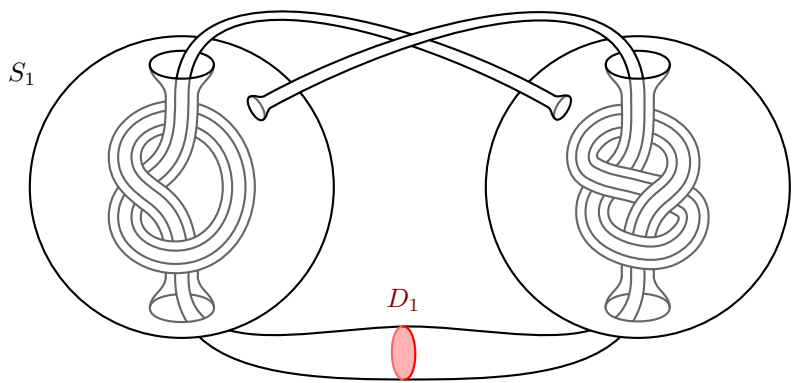

The torus $T_{1,1}$ bounds a solid torus $\storus$ in $\sphere[3]$. But $P_{1,1}$ cannot be a solid torus, since $S_1$ has no non\=/separating compression discs in $M_1$. Hence, the solid torus $\storus$ must be the closure of the other component of $\sphere[3]\setminus T_{1,1}$. In particular, we have that $P_{1,2}\subs\storus$. Now, the torus $T_{1,2}$ must be compressible in $\storus$, since it is not $\pi_1$\=/injective. It cannot be compressible in $P_{1,2}$ though, because $S_1$ has no non\=/separating compression discs in $M_1$. Therefore, it is easy to see that $P_{1,2}$ is contained in a ball and, moreover, it is a (non\=/trivial) knot complement. Clearly, the same holds for $P_{1,1}$. From this discussion, it follows that $Q_1$ is homeomorphic to the complement in $\sphere[3]$ of the $2$\=/component unlink. As described in \cref{fig:separating compression disc M1 N1}, the $3$\=/manifold $M_1$ is obtained by joining $P_{1,1}$ and $P_{1,2}$ by a $1$\=/handle, while drilling an arc from $Q_1$ yields $N_1$.

\begin{figure}
\centering
\begin{subcaptionblock}{\linewidth}
\begin{minipage}[c]{.09\linewidth}
\centering
\caption{\label{fig:separating compression disc M1}}
\end{minipage}
\begin{minipage}[c]{.9\linewidth}
\tikzsetnextfilename{separating-compression-disc-M1}
\centering
\begin{tikzpicture}[thick]
\draw[black,fill=main color 2!30] circle (2cm);
\path[spath/save=hole 1] (0,1.2) to[out=-90,in=90,out looseness=2] (-.75,-.4) to[out=-90,in=-90,out looseness=1.2,in looseness=2] (.75,0);
\path[spath/save=hole 1] [spath/use=hole 1,spath/transform={hole 1}{yscale=-1},spath/append reverse=hole 1];
\tikzset{ks/subdivide={path=hole 1,to=hole 1,n=2},ks/double={path=hole 1,to=thick hole 1,to contour=hole contour 1,width=12}}
\tikzset{ks/extract components={draft mode=false,draft mode scale=.2,path=thick hole 1,not={thick hole 1}{8,21,12,25,4,17}}}
\fill[main color 2!15][spath/use=hole contour 1];
\path[spath/use=hole 1] \foreach \pos/\xs/\as in {0/1/{black},0.9999/-1/{black!60}} {pic[sloped,xscale=\xs,draw=black!60,pos=\pos] {ks/tube end={12pt}{arcs/.style={\as},disc/.style={fill=main color 2!15},tube/.style={fill=main color 2!15}}}};
\draw[black!60][spath/use=thick hole 1];
\begin{scope}[shift={(6,0)}]
\draw[black,fill=main color 2!30] circle (2cm);
\path[spath/save=hole 2] (0,1.2) to[out=-90,in=90,out looseness=2.2] (-.6,-.33) to[out=-90,in=-135] (.6,-.7) to[out=45,in=-15,out looseness=2] (0,0);
\path[spath/save=hole 2] [spath/use=hole 2,spath/transform={hole 2}{xscale=-1,yscale=-1},spath/append reverse=hole 2];
\tikzset{ks/subdivide={path=hole 2,to=hole 2,n=2},ks/double={path=hole 2,to=thick hole 2,to contour=hole contour 2,width=12}}
\tikzset{ks/extract components={draft mode=false,draft mode scale=.2,path=thick hole 2,not={thick hole 2}{12,29,16,33,4,21,8,25}}}
\fill[main color 2!15][spath/use=hole contour 2];
\path[spath/use=hole 2] \foreach \pos/\as in {0/{black},0.9999/{black!60}} {pic[sloped,draw=black!60,pos=\pos] {ks/tube end={12pt}{arcs/.style={\as},disc/.style={fill=main color 2!15},tube/.style={fill=main color 2!15}}}};
\draw[black!60][spath/use=thick hole 2];
\end{scope}
\path[spath/save=handle] (1.7,-.2) to[out=-15,in=-165] ({6-1.7},-.2);
\tikzset{ks/subdivide={path=handle,to=handle,n=2},ks/double={path=handle,to=thick handle,to contour=handle contour,width={5+15*sin(\t*180)}}}
\fill[main color 2!30] [spath/use=handle contour];
\draw[black] [spath/use=thick handle];
\path[spath/use=handle] \foreach \pos/\xs in {0/1,0.9999/-1} {pic[sloped,xscale=\xs,draw=black,pos=\pos] {ks/tube end={5pt}{tube above,arc 1/.style={black!60}}}};
\path[spath/use=handle] pic[sloped,main color 1,pos=.5]{ks/disc={20pt}{disc/.style={fill=main color 1!30},arc/.style={main color 1!60}}} node[pos=.5,above=12pt,main color 1!50!black] {$D_1$};
\node[main color 2!50!black] at (-2.1,1.5) {$M_1$};
\end{tikzpicture}
\end{minipage}
\end{subcaptionblock}\\[3em]
\begin{subcaptionblock}{\linewidth}
\centering
\tikzsetnextfilename{separating-compression-disc-N1}
\begin{minipage}[c]{.09\linewidth}
\centering
\caption{\label{fig:separating compression disc N1}}
\end{minipage}
\begin{minipage}[c]{.9\linewidth}
\centering
\begin{tikzpicture}[thick]
\draw[black,fill=main color 3!30,even odd rule] circle (4 cm and 2cm) (-2.2,0) circle (.8cm and .4cm);
\path[spath/save=torus,fill=main color 3!15,even odd rule] (1.5,0) circle (1.8cm and 1.2cm) circle (.8cm and .4cm);
\path[draw=red,spath/save=curve] (.9,1.3) to[out=-105,in=90] (1.5,0) to[out=-90,in=-90,out looseness=1.2,in looseness=.7] (-3.5,0) to[out=90,in=135,out looseness=.7,in looseness=1.2] (.0,.5);
\tikzset{ks/subdivide={path=curve,to=curve,n=4},ks/double={path=curve,to=thick curve,to contour=curve contour,width={10-5*\t}}}
\tikzset{
ks/extract components={draft mode=false,draft mode scale=.3,path=thick curve,split with self=false,split with={torus},not={thick curve x}{4,10}},
spath/remove empty components=torus,
ks/extract components={draft mode=false,draft mode scale=.3,path=torus,split with self=false,split with={thick curve},not={torus}{1,7,3}},
}
\fill[main color 3!15][spath/use=curve contour];
\path[spath/use=curve] pic[sloped,draw=black!60,pos=0,xscale=-1] {ks/tube end={10pt}{disc/.style={fill=main color 3!15},tube/.style={fill=main color 3!15},arcs/.style=black}} pic[sloped,draw=black!60,pos=.999,xscale=-1] {ks/tube end={5pt}{arc 1/.style={draw=black!30},tube above}};
\draw[black!60][spath/use/.list={thick curve x,torus}];
\path[spath/use=curve] pic[draw=main color 1,sloped,pos=.5] {ks/disc={7.5pt}{arc/.style={main color 1!60}}} node[main color 1!50!black,pos=.5,above=4pt] {\tikzcontour{$p_1$}};
\node[main color 3!50!black] at (-3.5,1.6) {$N_1$};
\end{tikzpicture}
\end{minipage}
\end{subcaptionblock}
\caption{The $3$\=/manifold $M_1$ \subref{fig:separating compression disc M1} is homeomorphic to two knot complements $P_{1,1}$ and $P_{1,2}$ joined by a $1$\=/handle. The $3$\=/manifold $N_1$ \subref{fig:separating compression disc N1} is homeomorphic to the complement $Q_1$ of a $2$\=/unlink, from which an arc connecting the two boundary components has been drilled out.\label{fig:separating compression disc M1 N1}}
\end{figure}

The following lemma implies, \emph{a posteriori}, that no arbitrary choice was made when selecting the disc $D_1$.

\begin{lemma}
Let $P_1$ and $P_2$ be irreducible boundary irreducible $3$\=/manifolds, and let $M$ be the result of joining them by a $1$\=/handle $H$. Then the co\=/core of $H$ is the unique non\=/trivial compression disc for $\boundary M$ in $M$.
\end{lemma}
\begin{proof}
Let $F$ be the co\=/core of $H$, and suppose there is another (non\=/trivial) compression disc $D$ for $\boundary M$ which is not isotopic to $F$; in particular, $D$ is not boundary\=/parallel. By \cref{thm:discs with boundary compression-stable property,thm:boundary compression-stable properties} applied to the pair $(M,\boundary M)$, we may assume that $D$ is disjoint from $F$ and therefore, without loss of generality, contained in $P_1$. Since $P_1$ is irreducible and boundary irreducible, $D$ must cobound a ball in $P_1$ with some disc $D'\subs\boundary P_1$. But $D$ is not boundary parallel in $M$, and this is only possible if $D'$ contains one of the attaching discs of $H$. This would imply that $D$ is isotopic to $F$, contrary to our assumption.
\end{proof}

Now, if $M_2$ is boundary irreducible or $S_2$ has a non\=/separating compression disc in $M_2$, then clearly $S_1$ and $S_2$ are not isotopic. Otherwise, by imitating the above procedure for $S_2$ instead of $S_1$, we define $p_2$, $D_2$, $P_2$, $Q_2$, $T_2$, $P_{2,1}$, $P_{2,2}$, $T_{2,1}$, and $T_{2,2}$. Since $D_1$ (respectively $D_2$) can be intrinsically defined as the unique compression disc for $S_1$ (respectively $S_2$) in $M_1$ (respectively $M_2$), every homeomorphism $\umap{M_1}{M_2}$ must send $D_1$ to a disc isotopic to $D_2$. If the $3$\=/manifold pair $(N_1,p_1)$ is irreducible, we immediately conclude by applying \cref{thm:general strategy for knot complements}.

Otherwise, we can assume that the surfaces $\boundary N_1\cut p_1$ and $\boundary N_2\cut p_2$ are compressible in $N_1$ and $N_2$ respectively. We will show that, in this case, the surfaces $S_1$ and $S_2$ are isotopic if and only if $P_1$ and $P_2$ are homeomorphic. The forward implication is trivial. Conversely, suppose that $P_1$ and $P_2$ are homeomorphic. Let $E$ be a non\=/trivial compression disc for $\boundary N_1\cut p_1$ in $N_1$. The disc $E$ is properly embedded in $Q_1$, and its boundary lies in a component of $T_1$, say $T_{1,1}$.
\begin{substeps}
\item If $\boundary E$ is non\=/trivial in $T_{1,1}$, then $N_1\cut E$ is homeomorphic to a solid torus $\storus$ via a homeomorphism sending $p_1$ to a trivial curve in $\boundary\storus$. Therefore, $N_1$ is a handlebody and $p_1$ is a meridian curve of $N_1$.
\item If $\boundary E$ is trivial in $T_{1,1}$, up to isotoping $E$, we can assume that $\boundary E=p_1$. The sphere $S=E\cup D_1$ is embedded in $Q_1$, and does not bound a ball, hence it separates $Q_1$ into two components, each homeomorphic to a punctured solid torus. As a consequence, the disc $E$ separates $N_1$ into two components, each homeomorphic to a solid torus. We conclude that $N_1$ is a handlebody, and $p_1$ is a meridian curve of $N_1$.
\end{substeps}

Either way, we have shown that $(N_1,p_1)$ and $(N_2,p_2)$ are both homeomorphic to a handlebody endowed with a separating meridian curve as a boundary pattern. Let $m_{1,1}\subs T_{1,1}$ and $m_{1,2}\subs T_{1,2}$ be the unique curves which are homologically trivial in $Q_1$ but not in $T_1$; define $m_{2,1}\subs T_{2,1}$ and $m_{2,2}\subs T_{2,2}$ analogously. By construction, the curves $m_{1,1}$, $m_{1,2}$, $m_{2,1}$ and $m_{2,2}$ are meridians of the knot complements $P_{1,1}$, $P_{1,2}$, $P_{2,1}$ and $P_{2,2}$ respectively. By \cref{thm:homeomorphisms of knot complement}, every homeomorphism $\umap{P_1}{P_2}$ sends $m_{1,1}\cup m_{1,2}$ to $m_{2,1}\cup m_{2,2}$ up to isotopy. Moreover, it is not hard to see that every homeomorphism $\umap{T_1}{T_2}$ with this property is induced by $\hext{p_1,p_2}(f)$ for some $f\in\homeo{(N_1,p_1);(N_2,p_2)}$; see \cref{fig:separating compression disc homeomorphisms of Q} for an explanation. Therefore, the trace of every homeomorphism $\umap{P_1}{P_2}$ extends to a homeomorphism $\umap{Q_1}{Q_2}$ which is the image under $\hext{p_1,p_2}$ of some homeomorphism $\umap{(N_1,p_1)}{(N_2,p_2)}$. Thanks to \cref{thm:general strategy}, this is enough to conclude that $S_1$ and $S_2$ are isotopic.

\begin{figure}
\centering
\def\handlebodysetup{
\pgfinterruptboundingbox
\tikzset{
torus/new={name=T1,at={(0,0)},R=.75,r=.25},
torus/new={name=T2,at={(1.5,0)},R=.75,r=.25},
torus/set canvas box={T1}{.7}{0}{.7}{1},
torus/set canvas box={T2}{.8}{0}{.8}{1}
}
\path[spath/save=m1] (.7,0) to +(0,1);
\path[spath/save=m2] (.8,0) to +(0,1);
\tikzset{
torus/split visible={m1}{torus=T1,to visible=m1 front,to invisible=m1 back},
torus/split visible={m2}{torus=T2,to visible=m2 front,to invisible=m2 back},
torus/path={m1 front}{to=m1 front,torus=T1},
torus/path={m2 front}{to=m2 front,torus=T2},
torus/path={m1 back}{to=m1 back,torus=T1},
torus/path={m2 back}{to=m2 back,torus=T2},
}
\draw[main color 1!60,dotted,dash expand off][spath/use=m1 back][spath/use=m2 back];
\draw[main color 1][spath/use=m1 front][spath/use=m2 front];
\endpgfinterruptboundingbox
}
\def\sh{6cm}
\begin{subcaptionblock}{\linewidth}
\begin{minipage}[c]{.09\linewidth}
\caption{}
\label{fig:separating compression disc homeomorphisms of Q:a}
\end{minipage}
\begin{minipage}[c]{.9\linewidth}
\centering
\tikzsetnextfilename{separating-compression-disc-homeomorphisms-of-Q-a}
\begin{tikzpicture}[thick,torus/setup=65]
\path[use as bounding box] (-1,.-1.5) rectangle (8.5,1.5);
\foreach \i/\ll/\lr in {0/1/2,1/2/1} {
    \begin{scope}[shift={({\i*\sh},0)}]
        \handlebodysetup
        \tikzset{
        torus/get right boundary crit={T1}{\ucrit},
        torus/point={(-\ucrit,{vcrit1(-\ucrit)})}{torus=T1,to=b}
        }
        \path [torus/right boundary point={T1}{90}] coordinate (u);
        \tikzset{torus/point={(u)}{torus=T1,to=u}}
        \draw[main color 3] (.75,-1) -- (b) edge[main color 3!60,dotted] (u);
        \draw[line cap=round][torus/outline={torus=T1,cut right}][torus/outline={torus=T2,cut left}];
        \draw[line cap=round,main color 3,preaction={draw=white,line width=2pt}] (.75,1) -- (u);
        \node[anchor=base,main color 1!50!black] at (-.25,-1) {$m_\ll$};
        \node[anchor=base,main color 1!50!black] at (1.75,-1) {$m_\lr$};
    \end{scope}
    \draw[shift={({.75cm+\sh/2},0)},->,theme color] (-15pt,0) -- (15pt,0) pic[midway,yshift=12pt] {code={\draw[-,line cap=round,main color 3] (0,-6pt) -- (0,6pt);\draw[yscale={cos(65)},-{>[scale=.6,bend]},main color 4,preaction={draw=white,-,line width=2pt}] (150:6pt) arc(150:390:6pt);}};
}
\end{tikzpicture}
\end{minipage}
\end{subcaptionblock}\\
\begin{subcaptionblock}{\linewidth}
\begin{minipage}[c]{.09\linewidth}
\caption{}
\label{fig:separating compression disc homeomorphisms of Q:b}
\end{minipage}
\begin{minipage}[c]{.9\linewidth}
\centering
\tikzsetnextfilename{separating-compression-disc-homeomorphisms-of-Q-b}
\begin{tikzpicture}[thick,torus/setup=65]
\path[use as bounding box] (-1,.-1.5) rectangle (8.5,1.5);
\foreach \i/\arrowd in {0/,1/reversed} {
    \begin{scope}[shift={({\i*\sh},0)}]
        \handlebodysetup
        \path[/pgf/fpu/install only={reciprocal},spath/use=m1 front,decorate,decoration={markings,mark=at position .5 with {\arrow[main color 1]{>[\arrowd,scale=.6]}}}];
        \begin{scope}[on background layer]
        \begin{scope}
            \clip (-1,-1) rectangle (.75,1);
            \fill[even odd rule,main color 3!30] [torus/contour={T1}];
        \end{scope}
        \end{scope}
        \draw[line cap=round][torus/outline={torus=T1,cut right}][torus/outline={torus=T2,cut left}];
        \node[anchor=base,main color 1!50!black] at (-.25,-1) {$m_1$};
        \node[anchor=base,main color 1!50!black] at (1.75,-1) {$m_2$};
    \end{scope}
}
\draw[shift={({.75cm+\sh/2},0)},->,theme color] (-15pt,0) -- (15pt,0) pic[midway,yshift=16pt] {code={
    \tikzset{
    torus/new={name=t1,at={(-.3cm,0)},R=.3pt,r=.1pt},
    torus/new={name=t2,at={(.3cm,0)},R=.3pt,r=.1pt},
    }
    \begin{scope}[on background layer]
    \begin{scope}
        \clip (-1,-1) rectangle (0,1);
        \fill[even odd rule,main color 3!30] [torus/contour={t1}];
    \end{scope}
    \end{scope}
    \draw[-,black,thin,line cap=round][torus/outline={torus=t1,cut right}][torus/outline={torus=t2,cut left}];
    \draw[xshift=-.8cm,xscale={cos(65)},-{>[scale=.6,bend]},main color 3] ($(60:6pt)$) arc (60:300:6pt);
}};
\end{tikzpicture}
\end{minipage}
\end{subcaptionblock}\\
\begin{subcaptionblock}{\linewidth}
\begin{minipage}[c]{.09\linewidth}
\caption{}
\label{fig:separating compression disc homeomorphisms of Q:c}
\end{minipage}
\begin{minipage}[c]{.9\linewidth}
\centering
\tikzsetnextfilename{separating-compression-disc-homeomorphisms-of-Q-c}
\begin{tikzpicture}[thick,torus/setup=65]
\path[use as bounding box] (-1,.-1.5) rectangle (8.5,1.5);
\foreach \i in {0,1} {
    \begin{scope}[shift={({\i*\sh},0)}]
    \handlebodysetup
    \path[spath/save=l] \ifnum\i=0
        (.2,.5) -- (1.2,.5)
    \else
        (.2,.5) -- (.45,.5) to[out=0,in=-180] (.95,1.5) -- (1.2,1.5)
    \fi;
    \tikzset{
        torus/set canvas box={T1}{.2}{.5}{1.2}{1.5},
        torus/split visible={l}{torus=T1,to visible=l front,to invisible=l back},
        torus/path={l front}{to=l front,torus=T1,samples=25},
        torus/path={l back}{to=l back,torus=T1,samples=25},
        }
    \scoped[on background layer] \draw[thick,dotted,main color 2!60,dash expand off] [spath/use=l back];
    \draw[main color 2][spath/use=l front];
    \draw[line cap=round][torus/outline={torus=T1,cut right}][torus/outline={torus=T2,cut left}];
    \node[anchor=base,main color 1!50!black] at (-.25,-1) {$m_1$};
    \node[anchor=base,main color 1!50!black] at (1.75,-1) {$m_2$};
    \end{scope}
}
\draw[shift={({.75cm+\sh/2},0)},->,theme color] (-15pt,0) -- (15pt,0) node[midway,above=2pt] {$\tau$};
\end{tikzpicture}
\end{minipage}
\end{subcaptionblock}
\caption{Three self\=/homeomorphisms of a handlebody $N$ endowed with a separating meridian curve $p$ as a boundary pattern (the boundary of the green region in \subref{fig:separating compression disc homeomorphisms of Q:b}): \subref{fig:separating compression disc homeomorphisms of Q:a} swaps the two meridian curves $m_1$ and $m_2$; \subref{fig:separating compression disc homeomorphisms of Q:b} inverts the orientation of $m_1$ and leaves $m_2$ unchanged; \subref{fig:separating compression disc homeomorphisms of Q:c} is a Dehn twist about the meridian disc bounded by $m_1$. If $Q$ denotes the $3$\=/manifold obtained by attaching a $2$\=/handle to $N$ along $p$ and $T=\boundary Q$, then the images of \subref{fig:separating compression disc homeomorphisms of Q:a}, \subref{fig:separating compression disc homeomorphisms of Q:b}, and \subref{fig:separating compression disc homeomorphisms of Q:c} under $\trace{(-)}\circ\hext{p,p}$ generate the subgroup of $\homeo{T}$ of homeomorphisms preserving the isotopy class of $m_1\cup m_2$.
}
\label{fig:separating compression disc homeomorphisms of Q}
\end{figure}

\subsection{Small \texorpdfstring{$I$\=/bundles}{I-bundles} in the JSJ decomposition}\label{sec:classification algorithm:small i-bundles}

Finally, we deal with the case where $S_1$ is not partially compressible on one side; in other words, we assume that each component of $\sphere[3]\cut S_1$ is either boundary irreducible or a handlebody. Since $S_1$ is compressible in $\sphere[3]$, at least one of the sides -- say $N_1$ -- must be a handlebody. If $M_1$ is a handlebody too, then $\sphere[3]=M_1\cup N_1$ is a Heegaard splitting of $\sphere[3]$. By a theorem of \textcite{waldhausen-heegaard-splitting-s3}, all genus\=/two Heegaard surfaces of $\sphere[3]$ are isotopic, hence $S_1$ and $S_2$ are isotopic if and only if $M_2$ and $N_2$ are handlebodies.

We can therefore assume that $M_1$ is boundary irreducible. If $M_2$ is not homeomorphic to $M_1$ or $N_2$ is not a handlebody, then clearly $S_1$ and $S_2$ are not isotopic. Otherwise, $M_2$ is boundary irreducible too. Let $F$ be the JSJ system of $M_2$. The boundary curves of the annuli components of $F$ split the surface $S_2$ into subsurfaces $R_1,\ldots,R_r$ (namely, the components of $S_2\cut\boundary F$) such that $\chi(R_1)+\ldots+\chi(R_r)=-2$. A simple analysis shows that there are only a handful of possible surface types for each component $R_i$:
\begin{itemize}
\item the annulus $\surf{0,2}$, with $\chi(\surf{0,2})=0$;
\item the pair of pants $\surf{0,3}$, with $\chi(\surf{0,3})=-1$;
\item the sphere with four punctures $\surf{0,4}$, with $\chi(\surf{0,4})=-2$;
\item the punctured torus $\surf{1,1}$, with $\chi(\surf{1,1})=-1$;
\item the torus with two punctures $\surf{1,2}$, with $\chi(\surf{1,2})=-2$;
\item the closed surface of genus two $\surf2$, with $\chi(\surf2)=-2$.
\end{itemize}

We are interested in the types of $I$\=/bundles which can occur in the JSJ decomposition of $M_2$.
\begin{enumarabic}
\item If one of the components is a product bundle $\surf{1,1}\times I$, then $S_2\cut\boundary F$ is the union of two surfaces homeomorphic to $\surf{1,1}$ (that is, the horizontal boundary of the $I$\=/bundle) and a positive number of annuli.
\item If one of the components is a twisted bundle $\nsurf{2,1}\twtimes I$, then $S_2\cut\boundary F$ is the union of a surface homeomorphic to $\surf{1,2}$ (that is, the horizontal boundary of the $I$\=/bundle) and a positive number of annuli.
\item Otherwise, each $I$\=/bundle component of the JSJ decomposition of $M_2$ is either a product $I$\=/bundle over $\surf{0,2}$ or $\surf{0,3}$, or a twisted $I$\=/bundle over $\nsurf{1,1}$ or $\nsurf{1,2}$.
\end{enumarabic}
We will now provide a solution to the isotopy problem for the last case, deferring the analysis of the first two to \cref{sec:classification algorithm:product bundle over punctured torus,sec:classification algorithm:twisted bundle over punctured klein bottle}.

Let $\map{f_0}{M_1}{M_2}$ be a homeomorphism. By \cref{thm:mapping class group of 3-manifold}, we can algorithmically compute
\begin{itemize}
\item a finite collection $\FFF$ of self\=/homeomorphisms of $\boundary M_2$, and
\item a finite collection $a_1,\ldots,a_m,b_1,\ldots,b_m$ of pairwise disjoint curves in $\boundary M_2$,
\end{itemize}
such that
\[
\trhomeo{M_1;M_2}=\bigcup_{f\in\FFF}\langle\twist{a_1}\twist{b_1}^{-1},\ldots,\twist{a_m}\twist{b_m}^{-1}\rangle f\trace{f_0}.
\]
Let $\map{g_0}{N_1}{N_2}$ be a homeomorphism. We have that $S_1$ and $S_2$ are isotopic if and only if some element of $\trhomeo{M_1;M_2}\trace{g_0}^{-1}$, seen as a self\=/homeomorphism of $\boundary N_2$, extends to a homeomorphism $\umap{N_2}{N_2}$. For each $f\in\FFF$, thanks to \cref{thm:product of dehn twists extension to handlebody}, we can algorithmically decide whether there is an element of
\[
\langle\twist{a_1}\twist{b_1}^{-1},\ldots,\twist{a_m}\twist{b_m}^{-1}\rangle f\trace{(f_0g_0^{-1})}
\]
which extends to a self\=/homeomorphism of $N_2$. Since $\FFF$ is finite, this is enough to solve the isotopy problem for $S_1$ and $S_2$.

\subsection{Product \texorpdfstring{$I$\=/bundle}{I-bundle} over punctured torus} \label{sec:classification algorithm:product bundle over punctured torus}

Only two very special cases are left -- namely, those where the JSJ decomposition of $M_2$ contains a ``large'' $I$\=/bundle piece. In this section, we address the case where a component $Z$ of $M_2\cut F$ is homeomorphic to $\surf{1,1}\times I$, where the horizontal boundary $\boundary_h Z$ consists of -- say -- the surfaces $R_1$ and $R_2$; see \cref{fig:product bundle over punctured torus example} for an example. Due to the Euler characteristic constraint, the complement $\closure{S_2\setminus\boundary_h Z}$ is an annulus. By gluing this annulus to the vertical boundary of $Z$, we get a torus $T_2=(S_2\setminus\boundary_h Z)\cup\boundary_v Z$ embedded in $\sphere[3]$. This torus separates $\sphere[3]$ into two components: let $P_2$ be the one lying inside $M_2$, and $Q_2=N_2\cup Z$ be the other one. If $P_2$ were a solid torus, then $\boundary R_1\subs\boundary P_2$ would be a longitude of $P_2$, since it bounds a surface in the complement of $P_2$. But then the meridian disc of $P_2$ would be a boundary compression disc for $\boundary_v Z$ in $M_2$, contradicting the fact that $\boundary_v Z$ is an annulus in a JSJ system and, hence, not boundary parallel. Therefore, $Q_2$ must be a solid torus, and $P_2$ is the complement of a non\=/trivial knot. The situation is described in \cref{fig:product bundle over punctured torus T2}.

\begin{figure}
\centering
\tikzsetnextfilename{product-bundle-over-punctured-torus-example}
\begin{tikzpicture}[thick]
\node[black] at (-1.9,1.7) {$S_2$};
\node[anchor=west,main color 1!50!black] at (3.3,1.7) {$Z$};
\path[use as bounding box] (-1.8,-2.2) (4.2,2.2);
\path[spath/save=curve,use Hobby shortcut] ([out angle=180]1,1) .. (120:.5) .. (-120:2) .. (0:.5) .. (120:2) .. (-120:.5) .. ([in angle=180]1,-1);
\tikzset{ks/subdivide={path=curve,to=curve,n=2},ks/double={path=curve,to=thick curve,width=10,to 1=thick curve 1,to 2=thick curve 2}}
\path[spath/save=ring 1,use Hobby shortcut] ([out angle=0]spath cs:thick curve 1 0) .. (1.5,.5) .. (2.5,-.4) .. (3.3,.4) .. (2.5,1.2) .. ([in angle=0]spath cs:thick curve 2 0) (2.5,.4) circle (.5);
\tikzset{spath/remove empty components=ring 1,spath/clone={ring 2}{ring 1},spath/transform={ring 2}{yscale=-1}}
\draw[main color 1,fill=main color 1!30] (spath cs: thick curve 1 0) to[bend left=45] (spath cs:thick curve 1 1) -- (spath cs:thick curve 2 1) to[bend right=120,looseness=5] (spath cs:thick curve 2 0) -- cycle;
\foreach \i in {1,2} {\fill[main color 1!15,even odd rule][spath/use=ring \i];}
\tikzset{
ks/extract components={draft mode=false,draft mode scale=.3,path=ring 1,split with self=false,split with=ring 2,not={ring 1}{2,6}},
ks/extract components={draft mode=false,draft mode scale=.3,path=ring 2,split with self=false,split with=ring 1,not={ring 2}{2,5}},
ks/extract components={draft mode=false,draft mode scale=.3,path=thick curve,not={thick curve}{8,21,4,17,12,25}}
}
\draw[black!60] [spath/use/.list={ring 1,ring 2}];
\draw[black] [spath/use=thick curve];
\path[spath/use=curve] \foreach \pos in {0,1} {pic[draw=main color 1,pos=\pos] {ks/disc={10pt}{arc/.style={main color 1!60},disc/.style={fill=main color 1!15}}}};
\end{tikzpicture}
\caption{The surface $S_2$ splits $\sphere[3]$ into two components: one is a handlebody, and the other has a JSJ piece $Z$ which is an $I$\=/bundle over a punctured torus.\label{fig:product bundle over punctured torus example}}
\end{figure}
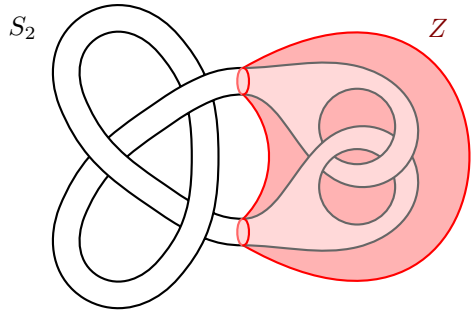

Let $K_2$ be a section of the projection $\umap{Z}{\surf{1,1}}$ lying in $Z\setminus\boundary_h Z$ (in other words, a surface of the form $\surf{1,1}\times\{1/2\}\subs\surf{1,1}\times I$). Note that $K_2$ is properly embedded in $Q_2$, and $Q_2\cut K_2$ is the handlebody $N_2$.

\begin{figure}
\centering
\tikzsetnextfilename{product-bundle-over-punctured-torus-T2}
\begin{tikzpicture}[thick]
\node[black] at (-1.9,1.7) {$T_2$};
\node[main color 2!50!black] at (1.1,-1.6) {$\boundary_h Z$};
\path[use as bounding box] (-1.8,-2.2) (4.2,2.2);
\path[spath/save=curve,use Hobby shortcut] ([out angle=180]1,1) .. (120:.5) .. (-120:2) .. (0:.5) .. (120:2) .. (-120:.5) .. ([in angle=180]1,-1);
\tikzset{ks/subdivide={path=curve,to=curve,n=2},ks/double={path=curve,to=thick curve,width=10,to 1=thick curve 1,to 2=thick curve 2}}
\path[spath/save=ring 1,use Hobby shortcut] ([out angle=0]spath cs:thick curve 1 0) .. (1.5,.5) .. (2.5,-.4) .. (3.3,.4) .. (2.5,1.2) .. ([in angle=0]spath cs:thick curve 2 0) (2.5,.4) circle (.5);
\tikzset{spath/remove empty components=ring 1,spath/clone={ring 2}{ring 1},spath/transform={ring 2}{yscale=-1}}
\foreach \i in {1,2} {\fill[main color 2!30,even odd rule][spath/use=ring \i];}
\tikzset{
ks/extract components={draft mode=false,draft mode scale=.3,path=ring 1,split with self=false,split with=ring 2,not={ring 1}{2,6}},
ks/extract components={draft mode=false,draft mode scale=.3,path=ring 2,split with self=false,split with=ring 1,not={ring 2}{2,5}},
ks/extract components={draft mode=false,draft mode scale=.3,path=thick curve,not={thick curve}{8,21,4,17,12,25}}
}
\draw[black] [spath/use=thick curve];
\draw[main color 2!60] [spath/use/.list={ring 1,ring 2}];
\draw[black] (spath cs: thick curve 1 0) to[bend left=90,looseness=.5] (spath cs:thick curve 1 1) (spath cs:thick curve 2 1) to [out=0,in=150] ++(.3,-.1) to[bend right=120,looseness=4.2] ($(spath cs:thick curve 2 0)+(.3,.1)$) to[out=-150,in=0] +(-.3,-.1);
\path[spath/use=curve] \foreach \pos in {0,1} {pic[draw=main color 2,pos=\pos] {ks/disc={10pt}{arc/.style={main color 2!60},disc/.style={fill=main color 2!15}}}};
\end{tikzpicture}
\caption{The torus $T_2$ splits $\sphere[3]$ into two components. One of them, namely $P_2$ (the ``outside'' in this picture), is a non\=/trivial knot complement; the other, namely $Q_2$, is a solid torus. The surface $\boundary_h Z$ is a union of two punctured tori, it lies inside $Q_2$, and splits it into two components: the handlebody $N_2$ and the product bundle $Z$.\label{fig:product bundle over punctured torus T2}}
\end{figure}
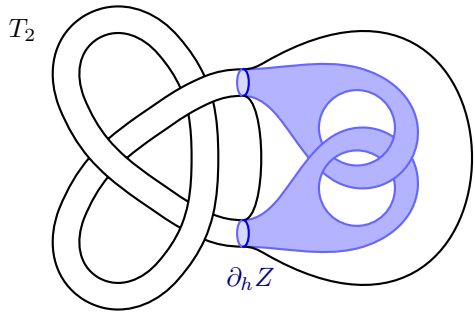

\begin{lemma}\label{thm:unique punctured torus in solid torus}
Let $K$ be a punctured torus properly embedded in a solid torus $\storus$ such that $\storus\cut K$ is a handlebody. Then:
\begin{enumarabic}
\item\label[statement]{thm:unique punctured torus in solid torus:1} the curve $\boundary K\subs\boundary\storus$ is a meridian;
\item\label[statement]{thm:unique punctured torus in solid torus:2} every punctured torus $K'$ properly embedded in $\storus$ such that $\storus\cut K'$ is a handlebody is isotopic to $K$.
\end{enumarabic}
\end{lemma}
\begin{proof}
\Cref{thm:unique punctured torus in solid torus:1} follows immediately from the fact that the meridian in $\boundary\storus$ is the only non\=/trivial curve which is trivial in $H_1(\storus)$. As far as \cref{thm:unique punctured torus in solid torus:2} is concerned, there is a ``standard'' embedding of a punctured torus in $\storus$, as depicted in \cref{fig:standard punctured torus in solid torus}. We aim to show that $K$ is isotopic to this standard punctured torus. Since $K$ is not $\pi_1$\=/injective, it is compressible in $\storus$. Let $D\subs\storus$ be a compression disc for $K$. We now analyse two cases, depending on whether $\boundary D$ separates $K$ or not (although, \emph{a posteriori}, we could always pick $\boundary D$ to be non\=/separating).

\begin{figure}
\centering
\tikzsetnextfilename{standard-punctured-torus-in-solid-torus}
\begin{tikzpicture}[thick,torus/setup=65,torus/default name=T]
\tikzset{torus/new={R=2,r=.75}}
\pic[main color 2,tdplot_main_coords] at (0,-2) {ks/disc={1.5cm}{disc/.style={fill=main color 2!30}}};
\path[spath/save=curve,tdplot_main_coords] (.2,-2,.4) to[bend left=90,looseness=2.5] (.2,-2,-.4);
\tikzset{ks/subdivide={path=curve,to=curve,n=3},ks/double={path=curve,to=thick curve,to contour=curve contour,width=5pt}}
\fill[main color 2!30][spath/use=curve contour];
\path[spath/use=curve] \foreach \pos in {0,0.9999} {pic[sloped,draw=main color 2,pos=\pos] {ks/tube end={5pt}{tube above,arc 1/.style={main color 2!60}}}};
\draw[main color 2][spath/use=thick curve];
\draw[torus/outline={torus=T}];
\node[main color 2!50!black,left] at (-.4,-.8) {$K$};
\node[right] at (2.4,1.2) {$\storus$};
\end{tikzpicture}
\caption{A ``standard'' punctured torus $K$ embedded in a solid torus $\storus$.}
\label{fig:standard punctured torus in solid torus}
\end{figure}

\begin{substeps}
\item If $\boundary D$ is non\=/separating in $K$, then compressing $K$ along $D$ yields a meridian disc $E\subs\storus$. The punctured torus $K$ can be recovered by removing $\boundary_v\nbhd{a}$ and adding $\boundary_h\nbhd{a}$ to $E$, where $a$ is a suitable arc in $\storus$ such that $a\cap E=\boundary a$. Note that $\storus\cut E$ is a $3$\=/ball, and $P=(\storus\cut E)\setminus\onbhd{a}$ is a knot complement. Moreover, if we attach a $1$\=/handle to $P$, we obtain a $3$\=/manifold which is homeomorphic to the handlebody $\storus\cut K$. Looking at fundamental groups, this implies that $\pi_1(P)\ast\ZZ=\ZZ\ast\ZZ$, hence the knot complement $P$ must in fact be a solid torus. This readily implies that the arc $a$ must be trivial, by which we mean that it must cobound a disc $D'$ with an arc in $E$, such that $\interior{D'}\cap E=\emptyset$. There is only one meridian disc $E$ up to isotopy, and there are two trivial arcs, one on each side of $E$. But $K$ is completely determined by $E$ and $a$, and it is easy to see that both choices of $a$ yield standard punctured tori in $\storus$.
\item If $\boundary D$ is separating in $K$, then compressing $K$ along $D$ yields the union of a meridian disc $E$ and a torus $T$. The punctured torus $K$ can be recovered by removing $\boundary_v\nbhd{a}$ and adding $\boundary_h\nbhd{a}$ to $E\cup T$, where $a$ is a suitable arc in $\storus$ which has one endpoint in $E$, one in $T$, and is otherwise disjoint from $E\cup T$. Let $Q$ be the closure of the component of $\storus\setminus T$ which is disjoint from $E$. If $Q$ is a solid torus, then $K$ admits a compression disc whose boundary does not separate it; by our analysis of the first case, it follows that $K$ is the standard punctured torus. Otherwise, $Q$ is a non\=/trivial knot complement, and $P=(\storus\cut E)\setminus\interior{Q}$ is a punctured solid torus. Note that if we join the knot complement $Q$ and the solid torus $P\setminus\onbhd{a}$ with a $1$\=/handle, we obtain a $3$\=/manifold which is homeomorphic to the handlebody $\storus\cut K$. Looking at fundamental groups, this implies that $\pi_1(Q)\ast\ZZ=\ZZ\ast\ZZ$, contradicting the fact that $Q$ is a non\=/trivial knot complement.
\qedhere
\end{substeps}
\end{proof}

We know that $M_1$ is homeomorphic to $M_2$; in particular, it has the same JSJ decomposition. Therefore, we can define $T_1$, $P_1$, $Q_1$ and $K_1$ like we did with $T_2$, $P_2$, $Q_2$ and $K_2$. These definitions are all canonical up to isotopy (note that there is exactly one component in the JSJ decomposition of $M_2$ which is an $I$\=/bundle $\surf{1,1}\times I$). Consequently, every homeomorphism $\umap{\sphere[3]}{\sphere[3]}$ sending $S_1$ to $S_2$ orientation\=/preservingly can be isotoped so that it sends $P_1$ to $P_2$, $Q_1$ to $Q_2$, and $K_1$ to $K_2$. Clearly, if $S_1$ and $S_2$ are isotopic then the $3$\=/manifolds with boundary pattern $(P_1,\boundary K_1)$ and $(P_2,\boundary K_2)$ are homeomorphic; we claim that the converse is also true. In fact, let $\map{f}{(P_1,\boundary K_1)}{(P_2,\boundary K_2)}$ be a homeomorphism. Since $\boundary K_1$ and $\boundary K_2$ are meridians of the solid tori $Q_1$ and $Q_2$ respectively, $f$ extends to a homeomorphism $\map{f}{\sphere[3]}{\sphere[3]}$ sending $Q_1$ to $Q_2$. By \cref{thm:unique punctured torus in solid torus}, we can isotope $f$ so that it sends $K_1$ to $K_2$. Since $N_1=Q_1\setminus\onbhd{K_1}$ and $N_2=Q_2\setminus\onbhd{K_2}$, the homeomorphism $f$ sends $N_1$ to $N_2$ and, hence, $S_1$ to $S_2$ orientation\=/preservingly.

The knot complements $P_1$ and $P_2$ are irreducible and boundary irreducible, therefore we can algorithmically decide whether $(P_1,\boundary K_1)$ and $(P_2,\boundary K_2)$ are homeomorphic or not. As we have shown, this provides a solution to the isotopy problem for $S_1$ and $S_2$.

\subsection{Twisted \texorpdfstring{$I$\=/bundle}{I-bundle} over punctured Klein bottle} \label{sec:classification algorithm:twisted bundle over punctured klein bottle}

Finally, suppose there is a component $Z$ of $M_2\cut F$ homeomorphic to $\nsurf{2,1}\twtimes I$, where the horizontal boundary $\boundary_h Z$ is a torus with two punctures $K_2$ embedded in $S_2$. Due to the Euler characteristic constraint, the complement $\closure{S_2\setminus K_2}$ is an annulus. We know that $M_1$ is homeomorphic to $M_2$; in particular, it has the same JSJ decomposition. Therefore, we can define $K_1$ like we did with $K_2$. There are two cases.

\step{When $K_2$ is incompressible in $N_2$.} By gluing the annulus $\closure{S_2\setminus K_2}$ to the vertical boundary of $Z$, we get a torus $T_2=(S_2\setminus K_2)\cup\boundary_v Z$ embedded in $\sphere[3]$. This torus separates $\sphere[3]$ into two components: let $P_2$ be the one lying inside $M_2$, and $Q_2=N_2\cup Z$ be the other one. Note that $K_2$ is properly embedded in $Q_2$, and moreover it is incompressible in $Q_2$, since it is incompressible on both sides (that is, in $N_2$ and in $Z$). This implies that $Q_2$ cannot be a solid torus, because $\pi_1(K_2)=\ZZ\ast\ZZ\ast\ZZ$ does not embed in $\ZZ$. Therefore, $Q_2$ is a non\=/trivial knot complement and $P_2$ is a solid torus.

If $K_1$ is compressible in $N_1$, then clearly $S_1$ and $S_2$ are not isotopic. Otherwise, let us define $T_1$, $P_1$ and $Q_1$ like we did with $T_2$, $P_2$ and $Q_2$. These definitions are all canonical up to isotopy (note that there is exactly one component in the JSJ decomposition of $M_2$ which is an $I$\=/bundle $\nsurf{2,1}\twtimes I$). It is then easy to see that $S_1$ and $S_2$ are isotopic if and only if there is a homeomorphism $\umap{\sphere[3]}{\sphere[3]}$ sending $Q_1$ to $Q_2$ and $K_1$ to $K_2$. If $Q_1$ is not homeomorphic to $Q_2$ then once again $S_1$ and $S_2$ are not isotopic. Otherwise, by \cref{thm:homeomorphisms of knot complement}, we can algorithmically produce a list $\FFF$ of representatives of isotopy classes of homeomorphisms $\umap{Q_1}{Q_2}$. Note that every $f\in\FFF$ extends to a homeomorphism $\umap{\sphere[3]}{\sphere[3]}$ by \cref{thm:homeomorphisms of knot complement}. If for some $f\in\FFF$ the surfaces $f(K_1)$ and $K_2$ are isotopic in $Q_2$ (which we can check, since the two surfaces are incompressible), then $S_1$ and $S_2$ are isotopic. Otherwise, they are not.

\step{When $K_2$ is compressible in $N_2$.} Let $D\subs N_2$ be a compression disc for $K_2$. Let $N_2'=N_2\cut D$, and let $K_2'=\boundary N_2'\setminus(\boundary N_2\setminus K_2)$ be the result of compressing $K_2$ along $D$. If $D$ is separating in the handlebody $N_2$, then one of the two components of $N_2'$ is a solid torus whose boundary is fully contained in $K_2'$. Therefore, up to replacing $D$ with the meridian disc of this solid torus, we can assume that $D$ is non\=/separating.

\begin{substeps}
\item Let us first assume that $K_2'$ is incompressible in $N_2'$. \Cref{thm:unique non-separating compression disc} implies that $D$ is the unique non\=/separating compression disc for $K_2$ in $N_2$, up to isotopy in $N_2$; let $D_2=D$. By repeating the above procedure for $K_1$ in $N_1$, we can check whether $K_1$ admits a unique non\=/separating compression disc in $N_1$. If this is not the case, then clearly $S_1$ and $S_2$ are not isotopic. Otherwise, let $D_1\subs N_1$ be this unique non\=/separating compression disc. By construction, the discs $D_1$ and $D_2$ are canonical.
\item Assume now that $K_2'$ is compressible in $N_2'$. Since $K_2'$ is an annulus, this immediately implies that the annulus $\closure{\boundary N_2\setminus K_2}$ is compressible in $N_2$ or, equivalently, that the boundary components of $K_2$ are meridian curves of $N_2$. Let $D_2$ be a disc properly embedded in $N_2$ whose boundary $\boundary D_2$ is a boundary component of $K_2$; obviously, this disc is unique up to isotopy. If the boundary components of $K_1$ do not bound discs in $N_1$, then clearly $S_1$ and $S_2$ are not isotopic. Otherwise, let $D_1$ be a disc properly embedded in $N_1$ whose boundary is a boundary component of $K_1$. As we said before, the discs $D_1$ and $D_2$ are canonical.
\end{substeps}

Either way, we have found two canonical non\=/separating compression discs $D_1$ and $D_2$ for $S_1$ and $S_2$ respectively. Moreover, $M_1$ is boundary irreducible and the $3$\=/manifold $N_1\cut D_1$ is a solid torus. By \cref{thm:general strategy for knot complements}, we can algorithmically decide whether $S_1$ and $S_2$ are isotopic.




\printbibliography[heading=bibintoc]

\end{document}